\documentclass[12pt,a4paper]{amsart}

\usepackage[utf8]{inputenc}
\usepackage[english]{babel}
\usepackage[T1]{fontenc}
\usepackage{amsmath}
\usepackage{mathtools}
\usepackage{amsfonts}
\usepackage{amssymb}
\usepackage{setspace}
\usepackage{amsfonts, dsfont}
\usepackage{mathrsfs, enumerate, csquotes, color}
\usepackage{xcolor}
\usepackage{bm}
\usepackage[bb=boondox]{mathalfa}
\definecolor{vertfonce}{rgb}{0.20, 0.46, 0.25}
\definecolor{rougefonce}{rgb}{0.64, 0.09, 0.20}
\usepackage[breaklinks=true,
colorlinks=true,
linkcolor=rougefonce,
citecolor=vertfonce]{hyperref}

\theoremstyle{plain}
\newtheorem{theorem}{Theorem}[section]

\newtheorem{lemma}[theorem]{Lemma}
\newtheorem{proposition}[theorem]{Proposition}
\theoremstyle{definition}
\newtheorem{definition}[theorem]{Definition}

\theoremstyle{remark}
\newtheorem{remark}[theorem]{Remark}
\numberwithin{equation}{section}

\theoremstyle{definition}

\DeclarePairedDelimiterX\braket[2]{\langle}{\rangle}{#1\,\delimsize\vert\,\mathopen{}#2}

\def\eps{\varepsilon}
\def\R{{\mathbb R}}
\def\C{{\mathbb C}}
\def\N{{\mathbb N}}
\def\Z{{\mathbf Z}}
\mathchardef\mhyphen="2D 
\def \Op{{\rm Op}}
\def \jac {{\rm jac}}
\def \supp {{\rm supp}}

\def \div{{\rm div}}

\def \Exp{{\rm Exp}}
\def \Ln{{\rm Ln}}
\def \exp{{\rm exp}}

\def \O{\mathscr{O}}

\def \cD{\mathcal{D}}
\def \cK{\mathcal{K}}
\def \cV{\mathcal{V}}
\def \cS{\mathcal{S}}
\def \bX{\mathbb{X}}

\def \bS{\mathbb{S}^{1}}
\def \pr{{\rm pr}}

\def \Heis{\mathbb{H}}
\def \Ghat{\widehat{G}}
\def \Hhat{\widehat{\Heis}}

\def \cH{\mathcal{H}}
\def \cF{\mathcal{F}}

\def \fg{\mathfrak{g}}
\def \fh{\mathfrak{h}}

\def \fv{\mathfrak{v}}
\def \fr{\mathfrak{r}}

\def \princ{{\rm princ}_{\hbar}}

\def\Tend#1#2{\mathop{\longrightarrow}\limits_{#1\rightarrow#2}}

\title[]{Quantum ergodicity for contact metric structures}
\author[L. Benedetto]{Lino Benedetto}
\address[L. Benedetto]{DMA, École normale supérieure, Université PSL, CNRS, 75005 Paris, France \& Univ Angers, CNRS, LAREMA, SFR MATHSTIC, F-49000 Angers, France} 
\email{lbenedetto@dma.ens.fr}

\numberwithin{equation}{section}

\begin{document}

\begin{abstract}
    This paper is dedicated to the proof of a Quantum Ergodicity (QE) theorem for the eigenfunctions of subLaplacians on contact metric manifolds, under the assumption that the Reeb flow is ergodic. To do so, we rely on a semiclassical pseudodifferential calculus developed for general filtered manifolds that we specialize in the setting of contact manifolds. Our strategy is then reminiscent of an implementation of the Born-Oppenheimer approximation as we rely on the construction of microlocal projectors in our calculus that commute with the subLaplacian, called \emph{Landau projectors}. The subLaplacian is then shown to act effectively on the range of each Landau projector as the Reeb vector field does. The remainder of the proof follows the classical path towards QE, once microlocal Weyl laws have been established.
\end{abstract}

\maketitle

\tableofcontents

\section{Introduction}

\subsection{Motivation}

Quantum Ergodicity (QE) theorems are concerned with the delocalization (or concentration) properties of stationary states of quantum systems in the high-energy limit. Mathematically, a version of QE can be expressed as follows: let $M$ be a compact metric space, endowed with a measure $\mu$, and let $P$ denote a self-adjoint operator on $L^2(M,\mu)$, bounded by below and having a compact resolvent (and hence a discrete spectrum). Let $(\varphi_k)_{k\in\N}$ be an orthonormal Hilbert basis consisting of eigenfunctions of $P$, and associated with the sequence of eigenvalues $\delta_0\leq\dots\leq\delta_k\leq\dots\rightarrow+\infty$. We say that QE holds for this eigenbasis if there exist a probability measure $\nu$ on $M$ and a density-one sequence $(k_j)_{j\in\N}$ of integers such that 
\begin{equation*}
    |\varphi_{k_j}|^2d\mu \rightharpoonup \nu\quad\mbox{when}\ j\rightarrow+\infty.
\end{equation*}

QE theorems have their origin in the seminal work of A. Schnirelman \cite{Sch}, and have been mainly stated in terms of the usual pseudodifferential calculus with $M$ being a compact Riemannian manifold and $P$ the Laplace-Beltrami operator. QE is then established under the assumption of ergodicity of the geodesic flow on the unit cotangent bundle for the Liouville measure, see \cite{CdV, Zel}. This result has been extended to the setting of manifolds with boundary in \cite{GL, ZZ}, to the semiclassical regime in \cite{HMR}, to the case of discontinuous metrics \cite{JSS}, as well as for Laplace and Dirac operators acting on bundles in \cite{JS}. However, the operators studied in the vast majority of these papers all share the property of being \emph{elliptic}. 

To our knowledge, the only QE theorem proved in the literature for non-elliptic operators is given by \cite[Theorem A]{CdVHT18}, which is stated on a closed 3D subRiemannian (sR) contact manifold and where the operator $P$ is given by its associated subLaplacian. Here, the ergodicity assumption falls onto the Reeb flow induced by the natural contact structure of such a sR manifold. In the broader question of finding the invariance properties of quantum limits for such a subLaplacian, we also mention the work \cite{AR}, where the addition of a semiclassical parameter allows for a precise multiscale asymptotic analysis and where the Reeb flow is also shown to play a predominant role.

\vspace{0.25cm}

The present paper is concerned with an extension of the result presented in \cite{CdVHT18} to any odd-dimensional closed contact manifold $(M^{2d+1},\eta)$ endowed with a contact metric structure $(\phi,g)$, i.e.,~$g$ is an associated Riemannian metric and $\phi$ is an almost complex structure on the contact distribution $$\cD = \ker \eta,$$ all compatible with each other in the sense that they satisfy
\begin{equation*}
    \forall X,Y\in\Gamma(\cD), \quad g(X,\phi Y) = d\eta(X,Y),
\end{equation*}
see Section~\ref{subsect:associatedmetrics}.

For such manifolds, the Riemannian volume form is in fact independent of the associated metric $g$ and is a multiplicative constant of the contact volume form $\eta\wedge (d\eta)^d$. We denote the associated normalized volume form by $\nu$.

The subLaplacian associated with the subRiemannian contact structure $(M^{2d+1},\eta,g)$ and with the measure $\nu$ (see Section~\ref{subsect:sublaplacian} for a precise definition), which we denote by $\Delta_{sR}$, is a differential operator which is essentially self-adjoint, hypoelliptic and which has discrete point spectrum given by the sequence 
\begin{equation*}
    0\leq \delta_0\leq \delta_1\leq \dots\leq\delta_k\leq\dots\rightarrow +\infty,
\end{equation*} 
with repetitions according to multiplicity. Moreover, there exists a Hilbert basis $(\varphi_k)_{k\in\N}$, such that for every $k\in \N$,
\begin{equation*}
    -\Delta_{sR}\varphi_k = \delta_k \varphi_k \quad\mbox{and}\quad \|\varphi_k\|_{L^2(M,\nu)} = 1.
\end{equation*}

Then our main result is the following one.

\begin{theorem}
    \label{thm:main}
    We assume that the Reeb flow is ergodic on $(M,\nu)$. Then, for any orthonormal Hilbert basis $(\varphi_k)_{k\in\N}$ consisting of eigenfunctions of the subLaplacian $-\Delta_{sR}$ associated with the eigenvalues $(\delta_k)_{k\in\N}$, there exists a density-one sequence $(k_j)_{j\in\N}$ of integers such that
    \begin{equation*}
        \int_{M}a(x)|\varphi_{k_j}(x)|^2\,d\nu(x) \Tend{j}{+\infty} \int_{M}a(x)\,d\nu(x),
    \end{equation*}
    for any continuous function $a$ on $M$.
\end{theorem}

The principal source of examples of contact metric structures where the Reeb vector field generates an ergodic flow is given by \emph{unit tangent bundles of Riemannian manifolds of negative sectional curvature}. Indeed, if $(M^{d+1},g)$ denotes such a manifold, its unit tangent bundle
\begin{equation*}
    SM = \{(x,v)\in TM\,:\,\|v\|_{g_x}=1\},
\end{equation*}
is a contact $(2d+1)$-manifold for the natural Liouville 1-form; moreover, as detailed in \cite[Section 9.2]{Blair}, $SM$ can be endowed with an associated metric, in the sense of Definition~\ref{def:contactmetricstruct}, called the \emph{Sasaki metric}. In this setting, the Reeb flow on $SM$ then coincides with the \emph{geodesic flow} and the assumption on the sectional curvature of $M$ readily implies the desired ergodicity property.

We also emphasize that by Section~\ref{subsubsect:3DsR}, any three dimensional closed subRiemannian contact manifold defines a contact metric structure, and we recover in particular the result of \cite{CdVHT18}.

\begin{remark}
    We mention, however, that \cite[Theorem A]{CdVHT18} has a more microlocal flavor than ours in the 3 dimensional case, since it is formulated in terms of classical pseudodifferential calculus. Up to cumbersome additions, Theorem~\ref{thm:main} could also be stated microlocally, but we would ultimately not be able to compare it directly with \cite[Theorem A]{CdVHT18} since our pseudodifferential calculi differ. More importantly, as soon as we consider contact metric manifolds of dimension higher than 5, the classical dynamic underlying the high-energy limit of eigenfunctions is given by a collection of partial connections lifting the Reeb flow to the so-called \emph{Landau bundles}, see Section~\ref{subsubsect:partialconnectLandau}. Thus, identifying the semiclassical measure (as defined in \cite{BFF25}) of a density-one subsequence of eigenfunctions would require a stronger assumption than the ergodicity of the Reeb flow, as in the spirit of \cite{JS}.
\end{remark}

\subsection{Strategy of the proof}

Similarly to any QE result, the first step towards relating the high-energy limit of eigenfunctions with a certain classical flow is to lift the density measures $|\varphi_k|^2\,d\nu$, $k\geq 0$, on $M$ to a certain phase space. Classically, the phase space of a smooth manifold $M$ is taken to be the cotangent bundle $T^*M$. However, for a contact manifold, seen as a filtered manifold, another choice can be made that is better suited to the anisotropic nature of the subLaplacian $\Delta_{sR}$. 

Following the presentation of \cite{FFF24}, we can associate with $M^{2d+1}$ its \emph{osculating Lie algebra bundle} $\mathfrak{g}M$, which is a vector bundle locally modeled on the Heisenberg Lie algebra $\mathfrak{h}^d$. Using the exponential map of nilpotent Lie groups, we can also consider its \emph{tangent group bundle} $GM$, where $G_xM = {\rm Exp}(\mathfrak{g}_xM)$, which will replace to all effects the usual tangent bundle $TM$. Our choice of phase space is then given by the \emph{dual tangent bundle} defined by
\begin{equation*}
    \Ghat M := \{(x,\pi)\,:\,x\in M,\,\pi\in\widehat{G_x M}\},
\end{equation*}
where $\widehat{G_x M}$ denotes the topological set of strongly continuous irreducible unitary representations of $G_xM$, up to equivalence.

All these bundles admit nice local trivializations described in Section~\ref{sect:tanggroupbundle}, thanks to the existence of well-suited local frames of $TM$, called $\phi$-frames, see Section~\ref{subsect:associatedmetrics}. In particular, we identify a dense open subset of the dual tangent bundle, denoted by $\Ghat_\infty M$, that is homeomorphic to the \emph{characteristic manifold} of $M$, a smooth and symplectic manifold.

\vspace{0.25cm}

These bundles lead to the following definition of symbols. Indeed, the symbols that we introduce are fields of operators
\begin{equation*}
    \sigma = \{\sigma(x,\pi):\mathcal{H}_\pi^\infty\rightarrow\mathcal{H}_\pi^\infty\,:\,(x,\pi)\in\Ghat M\},
\end{equation*}
where $\mathcal{H}_\pi^\infty$ denotes the subspace of smooth vectors of the space of representation $\mathcal{H}_\pi$ of $\pi$.

In fact, we can define symbol classes $S^m(\Ghat M)$, $m\in\R\cup\{-\infty\}$. Introducing a semiclassical parameter $\hbar>0$, we also define global semiclassical quantization procedures for such symbols, see Section~\ref{subsubsect:semiquantization}. The resulting semiclassical pseudodifferential calculus is denoted by 
\begin{equation*}
    \Psi_\hbar^m(M),\ m\in \R\cup\{-\infty\}.
\end{equation*}
In particular, we have
\begin{equation*}
    -\hbar^2\Delta_{sR} \in \Psi_\hbar^2(M),
\end{equation*}
and we denote by $H\in S^2(\Ghat M)$ its principal symbol. Using local trivializations of $\Ghat M$ given by $\phi$-frames, the symbol $H$ can be shown to be described by a field of harmonic oscillators, more precisely by a sum of $d$ one-dimensional harmonic oscillators with pondering coefficients set to 1, i.e., in a completely resonant state. This fact allows us to distinguish the different Landau levels of the symbol $H$ and to define the associated projector $\Pi_n$, $n\in\N$, as homogeneous symbols in $\dot{S}^0(\Ghat M)$, see Section~\ref{subsect:Landauproj}.

Fixing a global quantization procedure $\Op_\hbar$, we observe that the symbolic calculus gives us for every $n\in\N$,
\begin{equation*}
    [-\hbar^2 \Delta_{\rm sR},\Op_\hbar(\Pi_n)] \in \hbar \Psi_\hbar^{1}(M)\quad\mbox{and}\quad\Op_\hbar(\Pi_n)\circ \Op_\hbar(\Pi_n)=\Op_\hbar(\Pi_n) + \hbar\Psi_\hbar^{-1}(M).
\end{equation*} 

Thus, up to a remainder of order $\O(\hbar)$, the quantization of the symbol $\Pi_n$ gives rise to a projector that commutes with the subLaplacian. Section~\ref{subsect:Landauproj} is then devoted to show that, in the spirit of the implementation of a Born-Oppenheimer approximation, by following the ideas of \cite{BFKRV}, this remainder can be improved to any power of $\hbar$, i.e.,~up to $\O(\hbar^\infty)$, up to correcting the symbols $\Pi_n$, $n\in\N$, with lower order symbols. The resulting operators, which we denote here by $\hat{\Pi}_n^\hbar$, are called \emph{Landau projectors}.

\vspace{0.25cm}

The crucial observation that we make is that the subLaplacian acts on the range of the Landau projector $\hat{\Pi}_n^\hbar$, $n\in\N$, as does the operator
\begin{equation*}
    \hbar^2(2n+d)|R| + \O(\hbar^2),
\end{equation*}
see Proposition~\ref{prop:opeffectif}. From this result, we are able to prove an Egorov theorem for observables of the form $\hat{\Pi}_n^\hbar A \hat{\Pi}_n^\hbar$, with $A\in\Psi_\hbar^{-\infty}(M)$, which we refer to as \emph{Toeplitz operators}, see Theorem~\ref{thm:Egorov}.

\vspace{0.25cm}

In order to apply the now classical path towards QE, as presented, for instance, in \cite{An}, we are left with computing microlocal Weyl laws for observables in $\Psi_\hbar^{-\infty}(M)$. We do so using the functional calculus of the subLaplacian, which is shown by Theorem~\ref{thm:functionalcalc} to be part of the semiclassical pseudodifferential calculus $\Psi_\hbar^{-\infty}(M)$ for smooth functions with compact support of $-\hbar^2\Delta_{sR}$. The final arguments for the proof of Theorem~\ref{thm:main} are given in Section~\ref{subsect:proofmain}.
 
\begin{remark}
    We mention here that the crucial argument in both \cite{AR} and \cite{CdVHT18} is a Birkhoff normal form for the cometric $g_\cD^*$ close to its zero locus, i.e.~the characteristic manifold. The proof of this normal form, as presented in \cite[Appendix C]{CdVHT18}, is divided into two parts: if the first part could certainly be extended to the setting of contact metric manifolds due to the completely resonant state of the harmonic oscillators discussed above, the second part, referred to there as \emph{Melrose normal form}, is, however, specific to dimension 3 due to the same presence of resonances.
\end{remark}

We point out that this strategy towards QE on contact metric manifolds should be robust enough to be adapted to the setting of quasi-contact manifolds of dimension 4, as well as to Engel manifolds, where the high-energy limit is expected to exhibit geometric flows in relation with \emph{abnormal geodesics}.

\subsection{Organization of the paper}

In Section~\ref{sect:contactmanifolds}, we recall the main definitions and properties concerning contact manifolds and their associated compatible metrics. We also define the subLaplacian we wish to investigate and present a list of examples of contact metric manifolds which can be of particular interest.

Section~\ref{sect:tanggroupbundle} describes the tangent group bundle of contact metric manifolds seen as filtered manifolds. We recall there the harmonic analysis of the Heisenberg group, as well as make precise the relation between the dual tangent bundle and the characteristic manifold of the contact manifold.
In Section~\ref{sect:semipseudocalc}, we use these bundles to give a precise geometric description of the semiclassical pseudodifferential calculus for filtered manifolds (see \cite{BFF25}) in the case of contact metric manifolds. We conclude the section with the construction of the Landau projectors.

Section~\ref{sect:egorov} is dedicated to the proof of the Egorov theorem for Toeplitz operators, while Section~\ref{sect:quantumergodicity} focuses on microlocal Weyl laws and on the proof of Theorem~\ref{thm:main}.

Finally, Appendix~\ref{sect:functionalcalc} presents the semiclassical functional calculus for the subLaplacians.

\subsection{Aknowledgment} I warmly thank Clément Cren, Patrick Gérard, Emmanuel Giroux, Luc Hillairet, Véronique Fischer, Steven Flynn and Søren Mikkelsen for discussions related to various aspects of this paper. I also express my gratitude, in particular, to Clotilde Fermanian Kammerer for her constant interest and support during the realization of this work. This project was realized with the support of the Région Pays de la Loire via the Connect Talent Project HiFrAn 2022 07750, the grant ANR-23-CE40-0016 (OpART), and the France 2030 program, Centre Henri Lebesgue ANR-11-LABX-0020-01.

\section{Contact metric structures}
\label{sect:contactmanifolds}

This section is devoted to the definitions and main properties satisfied by contact manifolds and contact metric structures, and is largely inspired by \cite{Blair} where the reader will find the proofs of the statement discussed below.

\subsection{Contact manifolds} 

We begin with fundamental definitions and properties of contact manifolds.

\subsubsection{Generalities}
\label{subsubsect:generalitiescontact}

A \emph{contact manifold} is the data of a smooth manifold $M^{2d+1}$ together with a 1-form $\eta$ satisfying
\begin{equation*}
    \eta \wedge (d \eta)^d \neq 0.
\end{equation*}

The 1-form $\eta$ is called the \emph{contact form} and, in particular, the $(2d+1)$-form $\eta \wedge (d \eta)^d$ gives $M$ an orientation by defining a volume element. 
In what follows, we will denote by $\cD$ the \emph{contact distribution}, defined by 
\begin{equation*}
 \forall x\in M,\quad \cD_x := \{V \in T_x M\,:\,\eta_x(V)=0\}.
\end{equation*}

\begin{proposition}
 Let $(M^{2d+1},\eta)$ be a contact manifold. There exists a unique global vector field $R$, called the \emph{Reeb vector field}
 of the contact structure $\eta$, satisfying the following relations:
 \begin{equation*}
 \eta(R) = 1\quad\mbox{and}\quad d\eta(R,\cdot) = 0.
 \end{equation*}
\end{proposition}

A diffeomorphism $f$ between two contact manifolds $(M_1,\eta_1)$ and $(M_2,\eta_2)$ is called a \emph{contact transformation} if we have 
$f^* \eta_2 = \tau \eta_1$ for some non-vanishing function $\tau$. If the function $\tau$ is equal to 1, then one talks about a \emph{strict contact transformation}.
The next theorem indicates that any two contact structures can be sent on one another by a strict contact transformation locally (see, e.g., \cite[Section 6.2]{Mo02}).

\begin{theorem}[Darboux coordinates]
 \label{thm:darboux}
 In a neighborhood of any point of a contact manifold $(M^{2d+1},\eta)$, there exists a system of local coordinates $(x_1,\dots,x_d,y_1,\dots,y_d,r)$ with 
 respect to which we have 
 \begin{equation*}
 \eta = dr + \frac{1}{2}\sum_{i=1}^d (y_i dx_i - x_i dy_i).
 \end{equation*}
\end{theorem}

In such a Darboux coordinate system, the Reeb vector field is simply given by 
\begin{equation*}
 R = \partial_r.
\end{equation*}

An important geometric object associated with a contact manifold is its symplectification $\Sigma^{2d+2}$, also called its \emph{characteristic manifold}.

\begin{definition}
    \label{def:characteristic}
    Let $(M^{2d+1},\eta)$ be a contact manifold. We denote by $\Sigma\subset T^* M$ the annihilator $\cD^\perp\setminus o$ of the contact distribution $\cD = \ker \eta$, where $o$ denotes the zero section. It coincides with the cone of non-zero contact forms that defines $\cD$:
    \begin{equation*}
        \Sigma = \{(x,\lambda \eta_x)\in T^* M\,:\,x\in M,\,\lambda\in\R\setminus\{0\}\}.
    \end{equation*}
\end{definition}

The last expression of $\Sigma$ allows us to see it as a foliation parametrized by $\lambda$, in the sense that 
\begin{equation*}
    \Sigma = \bigcup_{\lambda\in\R\setminus\{0\}} \Sigma_\lambda,
\end{equation*}
where $\Sigma_\lambda = \{(x,\lambda\eta_x)\in T^* M\,:\,x\in M\}$. We have the following important result for the characteristic manifold. In what follows, we denote by $\omega$ the symplectic 2-form on $T^* M$.

\begin{proposition}[\cite{CdVHT18}]
    \label{prop:hamliftReeb}
    The restriction $\omega_{|\Sigma}$ defines a symplectic 2-form on $\Sigma$. Moreover, the Hamiltonian vector field $\Vec{\lambda}$ associated with the Hamiltonian $$ (x,\lambda\eta_x)\in\Sigma \mapsto \lambda\in\R,$$ leaves $\Sigma_\lambda$ invariant for $\lambda\in\R\setminus\{0\}$ and projects onto the Reeb vector field $R$.
\end{proposition}

Using the parameterization $j:(x,\lambda)\in M\times \R\setminus\{0\}\mapsto (x,\lambda\eta_x)\in\Sigma$, the restricted symplectic form $\omega_{|\Sigma}$ is given by 
\begin{equation*}
    \omega_{|\Sigma} = j^*(\omega) = -(d\lambda \wedge \eta + \lambda d\eta).
\end{equation*}
In particular, the associated symplectic volume form is given by 
\begin{equation}
    \label{eq:symplecticvolume}
    \omega_{|\Sigma}^{d+1} = (-1)^{d+1} \lambda^{d}\,d\lambda\wedge \eta \wedge (d\eta)^d.
\end{equation}

\subsubsection{Heisenberg group}
\label{subsubsect:heisenberg}

The contact structure on $\R^{2d+1}$ given by the 1-form described in Theorem~\ref{thm:darboux} will be identified with that of the Heisenberg group of dimension $2d+1$.
Indeed, we introduce the Lie algebra $\fh^d$, called the \emph{Heisenberg Lie algebra}, generated by the vector fields 
\begin{equation*}
 \tilde{X}_1,\dots,\tilde{X}_d,\tilde{Y}_1,\dots,\tilde{Y}_d,R,
\end{equation*}
satisfying the following commutator relations: for all $i,j\in\{1,\dots,d\}$, 
\begin{equation}
\label{eq:heisenbergrelations}
 [\tilde{X}_i,\tilde{Y}_j] = \delta_{i,j}R\quad\mbox{and}\quad [\tilde{X}_i,\tilde{X}_j] = [\tilde{Y}_i,\tilde{Y}_j] = 0,
\end{equation}
as well as any commutator with $R$ set to 0. We introduce the notation 
\begin{equation*}
 \fv = {\rm span}(\tilde{X}_1,\dots,\tilde{X}_d,\tilde{Y}_1,\dots,\tilde{Y}_d)\quad\mbox{and}\quad \fr = {\rm span}(R),
\end{equation*}
as well as $\mathfrak{p} = {\rm span}(\tilde{X}_1,\dots,\tilde{X}_d)$ and $\mathfrak{q} = {\rm span}(\tilde{Y}_1,\dots,\tilde{Y}_d)$. Then, we have the following gradation of $\fh^d$, adapted to the Lie bracket,
\begin{equation}
\label{eq:stratificationHeis}
 \fh^d = \fv \oplus \fr.
\end{equation}

We denote by $\Heis^d$ the following realization on $\R^{2d+1}$ of the \emph{Heisenberg group}, i.e.~the unique (up to isomorphism) connected simply connected nilpotent Lie group
with Lie algebra given by $\fh^d$:
\begin{equation}
 \label{eq:vectorfields}
 \tilde{X}_i = \partial_{x_i} - \frac{y_i}{2} \partial_r,\quad \tilde{Y}_i = \partial_{y_i} + \frac{x_i}{2}\partial_r\quad\mbox{and}\quad R =\partial_r.
\end{equation}
This corresponds to realizing the Heisenberg group on the Heisenberg Lie algebra $\fh$ through the exponential map 
\begin{equation*}
 \Exp_{\Heis^d} : \fh^d \rightarrow \Heis^d.
\end{equation*}
Indeed, for $h_1,h_2 \in \Heis^d$, writing $h_1 = \Exp(V_1 + R_1)$ and $h_2 = \Exp(V_2 + R_2)$, with $V_i\in\fv$ and $R_i\in\fr$, $i=1,2$, we have by the Baker-Campbell-Hausdorff formula
\begin{equation*}
 h_1 * h_2 = \Exp_{\Heis^d}(V + R),\quad\mbox{with}\ V = V_1 + V_2\in\fv\ \mbox{and}\ R = R_1 + R_2 + \frac{1}{2}[V_1,V_2]\in\fr,
\end{equation*}
as well as $h^{-1} = \Exp_{\Heis^d}(-V-R)$. 

In anticipation of the next sections, we endow the Heisenberg group with the Riemannian metric $g_\Heis$ defined by setting the left-invariant vector fields $(\tilde{X}_1,\dots,\tilde{X}_d,\tilde{Y}_1,\dots,\tilde{Y}_d,R)$ as a global orthonormal
frame. The corresponding \emph{Haar measure} is then simply the Lebesgue measure in the coordinates $(x_i,y_i,r)$, and we denote it by $\mu_\Heis$.

\subsubsection{Principal circle bundles}
\label{subsubsect:princcirclebundles}

A particular family of contact manifolds that will interest us in the applications is given by the principal circle bundles. We recall here elementary facts about them.

Let $P$ and $M$ be smooth manifolds, and $\pr: P\rightarrow M$ a smooth map of $P$ onto $M$. Then $(P,M)$ is called a \emph{principal circle bundle} if
\begin{itemize}
    \item[(i)] $\bS$ has a free right-action on $P$;
    \item[(ii)] for all $p_1,p_2\in P$, we have $\pr(p_1) = \pr(p_2)$ if and only if there exists $\theta \in \bS$ such that $p_1 = p_2 \cdot \theta$;
    \item[(iii)] $P$ is locally trivial over $M$, with local trivializations that commute with the action of $\bS$.
\end{itemize}

Thus, on any principal circle bundle $(P,M)$, one can consider the vector field $\partial_\theta$ given by the action of $\bS$ on $P$:
\begin{equation*}
    \forall p \in P,\quad \partial_\theta(p) = \frac{d}{dt}_{|t=0}(p\cdot e^{it}).
\end{equation*}

We will denote by $\mathcal{P}(M,\bS)$ the set of all principal circle bundles over $M$. It is a fundamental result of the theory of principal circle bundles that this set can be given the structure of an abelian group with the existence of an isomorphism
\begin{equation*}
    \mathcal{P}(M,\bS) \cong H^2(M,\Z),
\end{equation*}
i.e.~with the second group of cohomology of $M$ with integer coefficients, which we discuss now.

A \emph{connection} on a principal circle bundle $(P,M)$ is the datum of a smooth distribution $H\subset TP$ invariant by the action of $\bS$ and which gives the following decomposition of the tangent space:
\begin{equation*}
    \forall p\in P,\quad T_p P = H_p\oplus V_p,\quad\mbox{where}\ V = \ker \pr_{*},
\end{equation*}
and $V$ is called the \emph{vertical bundle}. One can then define in a unique way a non-zero 1-form $\eta$ whose kernel is given by the distribution $H$ and by $\eta(\partial_\theta)=1$, called the \emph{connection form}.

One can then consider the 2-form $\Phi = d\eta$, called the \emph{curvature form}, which is invariant by the action of $\bS$. Therefore, there exists a unique 2-form $\Omega$ on $M$ such that 
\begin{equation*}
    \Phi = \pr^* \Omega,
\end{equation*}
where $\pr^*$ denotes the pullback by the projection map $\pr : P\rightarrow M$.
One can show that the cohomology class of $\Omega$ is independent of the choice of connection on $P$, called the \emph{characteristic class} of $P$.

A fundamental result regarding principal circle bundles is the following one (see \cite[Theorem 2.5]{Blair}).

\begin{theorem}
\label{thm:existenceprincipalbundle}
    Let $\Omega$ be a 2-form on $M$ representing an element of $H^2(M,\Z)$. Then there exists a principal bundle $P$ and a connection form $\eta$ on $P$ such that $d\eta = \pr^* \Omega$.
\end{theorem}

Moreover, one observes that if $\Omega$ is symplectic, then $(P,\eta)$ is a contact manifold. This type of contact manifold is sometimes called a \emph{Boothby-Wang fibration}.

\subsection{Associated metrics}
\label{subsect:associatedmetrics}

We now present here the class of metrics that we will consider.

\subsubsection{Almost contact structures and compatible metrics}
\label{subsubsect:almostcontactstruct}

In order to do so, we first begin by enlarging our geometric setting by defining almost complex structures on contact distributions.

\begin{definition}
 \label{def:almostcontactstruct}
 An odd-dimensional smooth manifold $M^{2d+1}$ is said to have an \emph{almost contact structure} if it admits a tensor 
 field $\phi$ of type $(1,1)$, a vector field $R$, and a 1-form $\eta$ satisfying
 \begin{equation*}
 \phi^2 = -I+ \eta\otimes R\quad\mbox{and}\quad \eta(R)=1.
 \end{equation*}
\end{definition}

We now begin to define the class of Riemannian metrics that we will work with in the remainder of this paper.

\begin{definition}
 \label{def:almostcontactmetricstruct}
 Let $(M,\phi,R,\eta)$ be an almost contact structure. A Riemannian metric $g$ on $M$ is said to 
 be \emph{compatible} if it satisfies the relation 
 \begin{equation}
 \label{eq:compatiblemetric}
 \forall X,Y\in \Gamma(T M),\quad g(\phi X,\phi Y) = g(X,Y)-\eta(X)\eta(Y).
 \end{equation}
 The manifold $M$ is then said to be endowed with an \emph{almost contact metric structure}.
\end{definition}

\begin{remark}
We observe that for such a manifold, we immediately have
\begin{equation*}
 \eta(X) = g(X,R),
\end{equation*}
and thus $R$ is a unit vector field.
\end{remark}

In any (sufficiently small) open subset in a manifold $M^{2d+1}$ with an almost contact metric structure $(\phi,R,\eta,g)$, we can
exhibit useful orthonormal frames that we now describe. Let $U\subset M$ be a coordinate neighborhood and let $X_1$ be any unit (for the metric $g$) vector field on $U$
orthogonal to $R$. We then define the vector field $Y_1$ on $U$ by 
\begin{equation*}
 Y_1 = \phi (X_1).
\end{equation*}
By the compatibility relation given by Equation~\eqref{eq:compatiblemetric}, $Y_1$ is a unit vector field orthogonal to both $X_1$ and $R$.
We now choose a unit vector field $X_2$ orthogonal to $R$, $X_1$ and $Y_1$. Similarly to before, we set $Y_2 = \phi (X_2)$, which is a unit vector field orthogonal to $R$, $X_1$, $Y_1$ and $X_2$. Proceeding recursively, we obtain a local orthonormal frame 
\begin{equation*}
 \bX = (X_1,\dots,X_d,Y_1,\dots,Y_d,R),
\end{equation*}
which we call a \emph{$\phi$-frame}.

Using a $\phi$-frame to trivialize the tangent bundle, the $(1,1)$-tensor field $\phi$ is then given by the matrix 
\begin{equation*}
 \begin{pmatrix}
 0 & -I_d & 0\\
 I_d & 0 & 0\\
 0 & 0 & 0
 \end{pmatrix}.
\end{equation*}

\subsubsection{Contact metric structures}

Considering a contact manifold $(M,\eta)$ with the Reeb vector field denoted by $R$, one can ask if it is possible to find a $(1,1)$-tensor field $\phi$ and a metric $g$ such that $(M,\phi,R,\eta,g)$ defines an almost contact metric structure. The answer is given by the following important result, which is a restatement of \cite[Theorem 4.4]{Blair}.

\begin{theorem}
 \label{thm:existencecontactmetric}
 Let $(M^{2d+1},\eta)$ be a contact manifold and $R$ its associated Reeb vector field. Then there exist a $(1,1)$-tensor field $\phi$ and a metric $g$ 
 such that $(M,\phi,R,\eta,g)$ is an almost contact metric structure and, moreover, satisfying  the relation 
 \begin{equation*}
 \forall X,Y \in \Gamma(TM),\quad g(X,\phi Y) = d\eta(X,Y).
 \end{equation*}
\end{theorem}

The previous theorem tells us that any contact manifold can be endowed with a contact metric structure in the sense of the following definition.

\begin{definition}
 \label{def:contactmetricstruct}
 If $(M^{2d+1},\phi,R,\eta,g)$ is an almost contact structure for which the relation 
 \begin{equation}
 \label{eq:associatedmetric}
 \forall X,Y \in \Gamma(TM),\quad g(X,\phi Y) = d\eta(X,Y),
 \end{equation}
 holds, then the metric $g$ is said to be an \emph{associated metric} and the manifold $M$ is said to be endowed with a \emph{contact metric structure}.
\end{definition}

Equivalently, given a contact manifold $(M^{2d+1},\eta)$ with Reeb vector field $R$, a Riemannian metric $g$ is an associated metric if, firstly,
\begin{equation*}
 \eta(X) = g(X,R),
\end{equation*}
and secondly, there exists a $(1,1)$-tensor field $\phi$ such that 
\begin{equation}
 \label{eq:relassociatedmetric}
 \phi^2 = -I+\eta\otimes R\quad\mbox{and}\quad d\eta(X,Y) = g(X,\phi Y).
\end{equation}
Since we are more interested in the metric $g$ than in the tensor field $\phi$, we will often simply write $(M^{2d+1},\eta,g)$ for a contact metric structure,
the existence of $\phi$ satisfying Equation~\eqref{eq:relassociatedmetric} being implicit.

Finally, we note that the volume form of any associated metric is determined by the contact form $\eta$, as stated below (see also \cite[Theorem 4.6]{Blair}).

\begin{proposition}
 Let $(M^{2d+1},\eta)$ be a contact manifold and $g$ be an associated metric. Then
 \begin{equation*}
 d{\rm vol}_g = \frac{(-1)^d}{2^d d!} \eta \wedge (d\eta)^d.
 \end{equation*}
\end{proposition}

\begin{remark}
    The last proposition also means that the volume form coincides, up to a multiplicative constant, with the Haar measure on $\Heis^d$ in any Darboux coordinates. This observation then makes it obvious that this volume form is invariant by the Reeb flow, i.e., the flow generated by the vector field $R$.
\end{remark}

In what follows, we shall drop the index $g$ when talking about the volume form and simply write ${\rm vol}$.

\subsubsection{Symplectic normal bundle}
\label{subsubsect:symnormbundle}

Another crucial geometric object associated with the characteristic manifold of a contact manifold (see Definition~\ref{def:characteristic}) is its symplectic orthogonal bundle. Indeed, since $ \Sigma$  is symplectic, for every $\sigma\in\Sigma$, we have the symplectic orthogonal decomposition
\begin{equation}\label{symplecticorthogonalsplitting}
T_\sigma(T^* M) = T_\sigma \Sigma \oplus \mathrm{orth}_\omega(T_\sigma \Sigma) ,
\end{equation}
where the notation $\mathrm{orth}_\omega$ stands for the symplectic $\omega$-orthogonal complement.
We define the symplectic normal bundle $N\Sigma$ of $\Sigma$ by
$$
N\Sigma = \{ (\sigma,w)\,:\, \sigma\in\Sigma,\ w\in
N_\sigma\Sigma \} \,\subset\, T(T^* M),
$$
where
$N_\sigma\Sigma = \mathrm{orth}_\omega(T_\sigma \Sigma) $.
The set $N\Sigma$ is a vector bundle over $\Sigma$ and the manifold $\Sigma$ is canonically embedded in $N\Sigma$, by the zero section.  

Considering a local $\phi$-frame $\bX = (X_1,\dots,X_d,Y_1,\dots,Y_d,R)$ in $U$, and denoting by $h_V$ the hamiltonian on $T^*M$ associated to any vector field $V\in \Gamma(M)$, we have
\begin{multline*}
T\Sigma_U =\bigcap_{1\leq i \leq d} \ker dh_{X_i}\cap \ker dh_{Y_i} 
= \{ w\in T(T^*_U M)\,:\,\forall i \in\{1,\dots,d\},\, dh_{X_i}.w = dh_{Y_i}.w = 0 \}  \\
= \{ w\in T(T^*_U M) \,:\,\forall i \in\{1,\dots,d\},\, \omega(\vec h_{X_i},w) = \omega(\vec h_{Y_i},w) = 0 \},
\end{multline*}
and thus
\begin{equation*}
N_\sigma\Sigma = {\rm span}\left( \vec h_{X_1}(\sigma),\dots,\vec h_{X_d}(\sigma), \vec h_{Y_1}(\sigma),\dots,\vec h_{Y_d}(\sigma)\right),
\end{equation*}
along $\Sigma$.

Denoting by $\pi:T^* M\rightarrow M$ the canonical projection, we have, along $\Sigma$, $d\pi(\vec h_{X_i}) = X_i\circ\pi$ and $d\pi(\vec h_{Y_i}) = Y_i\circ\pi$ for $1\leq i\leq d$, and
thus $d\pi_\sigma(N_\sigma\Sigma)  = \cD_{\pi(\sigma)}$. In other words, for every $\sigma\in\Sigma$, the mapping
\begin{equation}
\label{eq:defXi}
\Xi_\sigma := (d\pi_\sigma)_{\vert N_\sigma\Sigma}^{\vert \cD_{\pi(\sigma)}} : N_\sigma\Sigma\rightarrow \cD_{\pi(\sigma)},
\end{equation}
is an isomorphism. 


We now define $\omega_{N\Sigma}$ on the bundle $N\Sigma$ as follows: given any point $\sigma\in\Sigma$, we set 
$$
\omega_\sigma^{N\Sigma} = (\omega_\sigma)_{\vert N_\sigma\Sigma} .
$$
Since $\omega_\sigma^{N\Sigma}(\vec h_{X_i}(\sigma),\vec h_{Y_j}(\sigma)) = \omega_\sigma(\vec h_{X_i}(\sigma),\vec h_{Y_j}(\sigma))=\{h_{X_i},h_{Y_j}\}(\sigma)=\delta_{i,j}h_R(\sigma)=\lambda(\sigma)$, as well as 
$\omega_\sigma^{N\Sigma}(\vec h_{X_i}(\sigma),\vec h_{X_j}(\sigma)) = \omega_\sigma^{N\Sigma}(\vec h_{Y_i}(\sigma),\vec h_{Y_j}(\sigma)) =0$ for any indices $1\leq i,j\leq d$, it follows that
$$
\omega^{N\Sigma} = \lambda\, \Xi^* (-d\eta_{\vert \cD}).
$$
The symplectic normal bundle $N\Sigma$ inherits a linear symplectic structure on its fibers $N_\sigma\Sigma$ endowed with the symplectic form $\omega_\sigma^{N\Sigma}$.

\subsection{SubLaplacians}
\label{subsect:sublaplacian}

Let $(M^{2d+1},\eta,g)$ be a compact contact metric structure. We denote by $g_\cD$ the restriction of the metric to the contact distribution $\cD = \ker \eta$. We also introduce the (degenerate) dual metric $g_{\cD}^*$ on $T^* M$ associated with $g_\cD$, which is defined as follows: for $x\in M$ and any two $\xi_1, \xi_2\in T_x^* M$, we set
\begin{equation*}
    g_{\cD,x}^*(\xi_1,\xi_2)=g_{\cD,x}(v_1,v_2),
\end{equation*}
where $v_i$, $i=1,2$, are vectors in $\cD_x$ satisfying
\begin{equation*}
    \xi_i(\cdot) = g_{\cD,x}(v_i,\cdot).
\end{equation*}

Finally, we introduce the notation $\nu$ for the probability measure given by the density
\begin{equation*}
 \nu = \frac{1}{{\rm vol}(M)}|{\rm vol}|.
\end{equation*}

Let $L^2(M,\nu)$ be the Hilbert space of complex-valued functions $\psi$ such that $|\psi|^2$ is $\nu$-integrable over $M$. 
Then, we consider the following differential operator.

\begin{definition}
 \label{def:subLaplacian}
 We define the operator $-\Delta_{\rm sR}$, called the \emph{subLaplacian}, as the nonnegative self-adjoint operator on $L^2(M,\nu)$
 obtained by Friedrichs extension of the Dirichlet integral 
 \begin{equation*}
 \forall \psi \in C^\infty(M),\quad Q(\psi) = \int_M \|d\psi\|_{g_{\cD}^*}^2 d\nu,
 \end{equation*}
 where the norm of the 1-form $d\psi$ is computed with respect to the dual metric $g_{\cD}^*$. We denote by ${\rm Dom}(-\Delta_{sR})\subset L^2(M,\nu)$ the domain of the subLaplacian.
\end{definition}

\begin{remark}
    Actually, since $M$ is compact, the subLaplacian is essentially self-adjoint on $C^\infty(M)$, see \cite{Str}.
\end{remark}

If we are given a local $g$-orthonormal frame $(V_1,\dots,V_{2d})$ of $\cD$, we intoduce the notation 
\begin{equation*}
 \nabla_{sR} \psi = \sum_{i=1}^{2d} (V_i \psi)V_i,
\end{equation*}
called the \emph{subRiemannian gradient}. Indeed, it is independent of the choice of local $g$-orthonormal frame and is thus defined globally. Then we have 
\begin{equation*}
 Q(\psi) = \int_M \|\nabla_{sR} \psi\|_{g_\cD}^2 d\nu.
\end{equation*}
It follows that 
\begin{equation*}
 \Delta_{sR} = - \sum_{i=1}^{2d} V_{i}^* V_i = \sum_{i=1}^{2d} \left(V_{i}^2 + \div_\nu(V_i) V_i\right),
\end{equation*}
where the adjoints are taken in $L^2(M,\nu)$ and where the divergence $\div_\nu(X)$ of a vector field $X$ is given by 
$\mathcal{L}_X \nu = \div_\nu(X)\nu$.

\begin{remark}
    On the Heisenberg group $\Heis^d$ endowed with its left-invariant metric $g_{\Heis^d}$, the subLaplacian is then simply given by
    \begin{equation*}
        \Delta_{\Heis^d} = \sum_{i=1}^d (\tilde{X}_i^2+\tilde{Y}_i^2).
    \end{equation*}
\end{remark}

\subsection{Examples}

We present here various examples of contact manifolds which carry natural contact metric structure and their associated subLaplacian.

\subsubsection{3D subRiemannian contact manifolds} 
\label{subsubsect:3DsR}
The particular case of the 3-dimensional subRiemannian contact manifolds has been the main geometric setting studied for the analysis of subLaplacians associated with an equiregular distribution. We recall here the point of view adopted by \cite{CdVHT18} and show that it is included in the setting of contact metric structures.

\begin{definition}
 \label{def:3dsubRcontactmanifolds}
 A \emph{3-dimensional subRiemannian contact manifold} is the data of a smooth manifold $M^3$ together with a non-integrable and orientable distribution $\cD\subset TM$ of rank 2, as well as with 
 a smooth metric $g_\cD$ on the distribution $\cD$.
\end{definition}

Since $\cD$ is orientable, and as we are in dimension 3, it is co-orientable as well, and hence it can be seen as the kernel of a global 1-form $\eta_g$. Moreover, $\eta_g$ 
is chosen such that $d\eta_g$ coincides on $\cD$ with the volume form induced by $g_\cD$ on $\cD$:
\begin{equation}
    \label{eq:normalizationeta}
    d\eta_g = d{\rm vol}_{g_\cD}.
\end{equation}
The non-integrability condition on $\cD$ implies that $\eta_g$ is a contact form.

Denoting by $R$ the Reeb vector field associated with this contact form, we extend the metric $g_\cD$ to the whole tangent bundle by asking the vector field $R$ to be orthogonal to $\cD$ and of norm 1.
We denote by $g$ the resulting metric. 

Then $g$ is an associated metric with the contact manifold $(M^3,\eta_g)$. Indeed, considering the $(1,1)$-tensor field $\phi$ defined by the relation
\begin{equation*}
    \forall X,Y \in \Gamma(\cD), \quad d\eta(X,Y) = g(X,\phi Y),
\end{equation*}
and by $\phi(R) = 0$, it follows from the normalization condition \eqref{eq:normalizationeta} that the relations given by~\eqref{eq:relassociatedmetric} are satisfied.
The density given by the volume form ${\rm vol}_g$ on $M$ is then called in this context the \emph{Popp measure}.

\subsubsection{Strongly pseudoconvex CR-manifolds}

We present here an example of contact metric manifolds, which in particular encompasses the previous one.

\begin{definition}[\cite{Blair,DT}]
    A CR-manifold $M^{2d+1}$ is an odd-dimensional manifold with a subbundle
    \begin{equation*}
        T^{0,1}M\subset TM\otimes\C,
    \end{equation*}
    with the following properties:
    \begin{enumerate}
        \item[(i)] ${\rm dim}_\C T^{0,1}M = d$;
        \item[(ii)] $T^{1,0}M\cap T^{0,1}M = \{0\}$, where we set $T^{1,0}M = \overline{T^{0,1}M}$;
        \item[(iii)] $T^{0,1}M$ is involutive, meaning for two section $Z_1,Z_2$ of $T^{0,1}M$, the Lie bracket $[Z_1,Z_2]$ is still a section of $T^{0,1}M$.
    \end{enumerate}
\end{definition}

For any CR-manifold $(M^{2d+1},T^{0,1}M)$, following \cite[Section 6.4]{Blair}, there exists a subbundle $\cD\subset TM$ such that
\begin{equation*}
    \cD\otimes \C = T^{1,0}M\oplus T^{0,1}M,
\end{equation*}
as well as a unique bundle map $\phi:\cD\rightarrow\cD$ such that $\phi^2 = -{\rm Id}$ and
\begin{equation*}
    T^{0,1}M = \{X - i\phi X\,:\,X\in\Gamma(\cD)\}.
\end{equation*}

If $M$ is orientable, which we now assume, the distribution $\cD$ can be described as the kernel of a global 1-form $\eta$, called in this context a \emph{pseudo-Hermitian structure}. Associated with this structure is a bilinear form on $\cD$ defined by
\begin{equation*}
    L_\eta(X,Y) := -d\eta(X,\phi Y),\quad X,Y\in\Gamma(\cD),
\end{equation*}
and called the \emph{Levi form}. The pseudo-Hermitian manifold $(M^{2d+1},T^{0,1}M,\eta)$ is then said to be \emph{strongly pseudoconvex} if the Levi form is positive everywhere. 

It follows that for such a manifold, $\eta$ is a contact form. Extending both the Levi form and the almost complex structure $\phi$ to the Reeb vector field $R$ associated with $\eta$ in the natural way, it follows that any strongly pseudoconvex CR-manifold carries a contact metric structure.
In this context, the subLaplacian, as defined by Definition~\ref{def:subLaplacian}, is often denoted by $\Delta_b$.

\vspace{0.25cm}

The interest for CR-manifolds comes from the problem of restriction of holomorphic functions to hypersurfaces, on which the restriction of the Cauchy-Riemann equations gives rise to a particular differential operator. On a CR-manifold $(M^{2d+1},T^{0,1}M)$, this operator is defined by
\begin{equation*}
    \forall f\in C^\infty(M),\quad \overline{\partial}_bf = (df)_{|T^{0,1}M}\in\Gamma(T^{0,1}M^*),
\end{equation*}
called the \emph{tangential Cauchy-Riemann operator} and whose kernel thus represents in this abstract setting the restrictions of holomorphic functions. In fact, this operator can be defined for sections of $q$-forms on the bundle $T^{0,1}M^*$, $q\geq 0$, and gives rise to a cochain complex. 

This complex has been studied by many authors (see \cite{DT,FolKohn} for further references). Its study naturally leads to the introduction of a second-order differential operator, defined in particular on functions by
\begin{equation*}
    \square_b = \overline{\partial}_b^*\overline{\partial}_b,
\end{equation*}
and often referred to as the \emph{Kohn Laplacian}.

On functions, the Kohn Laplacian and the subLaplacian are related as follows (see e.g.\,\cite{Ep}):
\begin{equation*}
    \square_b = -\frac{1}{4}\left(\Delta_b -idR \right) + \mbox{lower order terms}.
\end{equation*}

This relation motivates the study of the subLaplacian on strongly pseudoconvex manifolds.

\subsubsection{Principal circle bundles} 

Suppose that we are given a symplectic manifold $(M^{2d},\Omega)$, with $\Omega$ of integral class and consider a corresponding Boothby-Wang fibration $P^{2d+1}$ with contact form $\eta$ and projection $\pr: P\rightarrow M$, as defined in Section~\ref{subsubsect:princcirclebundles}.

Suppose furthermore that $(M,\Omega)$ is endowed with an almost Kählerian structure, i.e.~there exists a $(1,1)$-tensor field $J$ such that $J^2 = -I_{d}$ and a Riemannian metric $G$ satisfying for all vector fields $X$ and $Y$,
\begin{equation*}
    \Omega(X,Y) = G(X,JY).
\end{equation*}
From this almost Kählerian structure, we can construct a contact metric structure on $P^{2d+1}$ as follows: for $X\in TM$, define $\phi(X)$ as the horizontal lift of $J(\pr_{*}(X))$. Then, since the Reeb vector field $R$ is vertical we get 
\begin{equation*}
    \phi^2 = -I + \eta \otimes R.
\end{equation*}
Defining now a Riemannian metric on $P$ by $g = \pr^{*} G + \eta \otimes \eta$, one can check that we have 
\begin{equation*}
    d\eta(X,Y) = g(X,\phi(Y)).
\end{equation*}
It is not hard to see that the Reeb vector field $R$ is Killing for the metric $g$, meaning we defined a K-contact metric structure $(\phi,R,\eta,g)$ on the principal circle bundle $P^{2d+1}$.

\subsubsection{Magnetic Schrödinger operators} Let $(X^2,g)$ be a compact oriented Riemannian surface. Suppose we are given a non zero 2-form $B$ of integral class. One can write 
\begin{equation*}
    B = b\, {\rm vol}_g,
\end{equation*}
with $b$ a smooth non-vanishing function on $M$ and ${\rm vol}_g$ the volume form associated with $g$. The function $b$ is called the \emph{magnetic field}, and we will assume it to be positive.

Applying Theorem~\ref{thm:existenceprincipalbundle}, one can consider a principal circle bundle $(M^3,X^2)$ with connection form $\eta$ satisfying 
\begin{equation*}
    d\eta = \pr^* B,
\end{equation*}
where $\pr:M\rightarrow X$ is the projection map. Since $B$ is non-degenerate, the 1-form $\eta$ is a contact form. Considering local coordinates $(x_1,x_2)$ on $X$ and parameterizing the fibers locally by the variable $\theta$, there exists a 1-form $A(x_1,x_2)$  such that we have
\begin{equation*}
    \eta = d\theta + A.
\end{equation*}

Denoting by $\cD = \ker\eta$ the horizontal distribution, since the map $\pr_{*}$ is one-to-one from $\cD$ to $TX$, one can lift the metric $g$ into a subRiemannian metric $g_\cD$ on $M$ defined on the distribution $\cD$. We find ourselves back in the setting of 3-dimensional subRiemannian contact manifolds or of principal circle bundles, and we can extend accordingly this subRiemannian structure into a contact metric structure.
Observe that the contact form $\eta_g \in \R \eta$ rescaled so as to have 
\begin{equation*}
    (d\eta_{g})_{|\cD} = d{\rm vol}_{g_\cD}\ \mbox{on}\ \cD,
\end{equation*}
is given by
\begin{equation*}
    \eta_g = b^{-1}\eta.
\end{equation*}
Using the coordinates $(x_1,x_2,\theta)$, $\eta_g$ is locally given by $b^{-1}(d\theta+A)$. One can then find the following expression of the Reeb vector field:
\begin{equation*}
    R = b\partial_\theta - \Vec{b},
\end{equation*}
where $\Vec{b}$ is the horizontal lift of the Hamiltonian vector field of $b$ computed with respect the symplectic form $B$.

Denoting by $(g_{i,j}(x_1,x_2))_{1\leq i,j\leq 2}$ the matrix coefficients of the metric $g$ in the coordinates $(x_1,x_2)$, we have the following expression for the subLaplacian in coordinates:
\begin{equation}
    \label{eq:sublaplaciancoord}
    \Delta_{sR} = \frac{1}{\sqrt{|g|}}\sum_{i,j} (\partial_i - A\partial_\theta)[\sqrt{|g|}g^{i,j}(\partial_j - A\partial_\theta)],
\end{equation}
where we set $|g| = {\rm det}(g_{i,j})$ and $(g^{i,j})_{1\leq i,j\leq 2}$ is the inverse matrix of $(g_{i,j})_{1\leq i,j\leq 2}$.

Observe that the subLaplacian is invariant by action of the circle $\bS$ on $M$, since the 1-form $\eta_g$ and the choice of metric is invariant. This can also be seen locally in the coordinates $(x,y,\theta)$ by Equation~\eqref{eq:sublaplaciancoord}.
Performing a Fourier series expansion of the functions in $L^2(M)$ with respect to the $\bS$-action, one obtain the decomposition
\begin{equation*}
    L^2(M) := \bigoplus_{m\in\Z} L_{m}^2(X),
\end{equation*}
where each subspace $L_{m}^2(X)$, $m\in\Z$, is stable under the action of $-\Delta_{sR}$. Introducing the notation $C_{m}^\infty(X)\subset L_{m}^2(X)$, $m\in\Z$, for the smooth functions corresponding to $m$-th mode of the Fourier series, we write $\Delta_m$ for the restriction of $\Delta_{sR}$ to $C_{m}^\infty(X)$. In the coordinates $(x,y)$ one gets
\begin{equation*}
    \Delta_m = \frac{1}{\sqrt{|g|}}\sum_{i,j} (\partial_i - m A)[\sqrt{|g|}g^{i,j}(\partial_j - m A)],
\end{equation*}
which corresponds to the \emph{magnetic Schrödinger operator} on $(X,g)$ with magnetic field given by the 2-form $m B$.

\section{Tangent group bundle}
\label{sect:tanggroupbundle}

In this section, we suppose given $(M^{2d+1},\eta,g)$, a compact contact metric structure. Most of the notions and notations are taken from \cite{FFF24}.

\subsection{Osculating Lie algebras}

Recall that we denote by $\cD$ the contact distribution associated with the 1-form $\eta$. This subbundle defines a (trivial) filtration of the tangent space:
\begin{equation*}
\{0\}\subset \cD\subset TM.
\end{equation*}
For each point $x\in M$, we define the graded vector space
\begin{equation*}
 {\rm gr}(T_x M) = \cD_x \oplus T_x M/ \cD_x.
\end{equation*}
Following the presentation of \cite[Section 3.1]{FFF24}, each of these fibers is naturally endowed with a Lie bracket, which we denote by $[\cdot,\cdot]_{\fg_x M}$. Indeed, for 
two vector fields $X,Y \in \cD$, we can set
\begin{equation*}
 [X_x,Y_x]_{\fg_x M} := [X,Y](x) \mod \cD_x,
\end{equation*}
as we observe that the Lie bracket of vector fields of $\Gamma(\cD)$ projected on $TM/\cD$ is $C^\infty(M)$-linear and thus makes sense punctually.
To complete the definition of the structure of Lie algebra, the bracket with any element of $T_x M/\cD_x$ is set to 0. 
We denote the resulting Lie algebra by 
\begin{equation*}
 \fg_x M = ({\rm gr}(T_x M),[\cdot,\cdot]_{\fg_x M}).
\end{equation*}
The collection of these fibers, denoted by $\fg M$, is then a smooth bundle of step-2 stratified nilpotent Lie algebras, called the \emph{bundle of osculating Lie algebras} of $(M,\eta)$. 

In order to obtain local trivializations of this vector bundle, we can consider adapted frames.

\begin{definition}
    A local adapted frame $(\bX,U)$ is a frame $\bX = (V_1,\dots,V_{2d+1})$ on the open subset $U\subset M$ such that $(V_1,\dots,V_{2d})$ is a frame for the contact distribution $\cD_{|U}$. We then denote by $\langle \bX\rangle := (\langle V_1\rangle,\dots,\langle V_{2d+1}\rangle)$ its natural projection on $\mathfrak{g}_{|U}M$.
\end{definition}

For $x\in M$, we introduce the notation $\fv_x M$ for the projection of $\cD_x$ in 
$\fg_x M$ as well as $\fr_x M$ for the projection of $T_x M/ \cD_x$, so as to be able to write
\begin{equation*}
 \fg M = \fv M \oplus \fr M.
\end{equation*}

As graded spaces, the fibers of $\fg M$ come with natural dilations.

\begin{definition}
    \label{def:dilations}
    For $x\in M$, we define on the Lie algebra $\fg_x M$ the family $(\delta_t)_{t>0}$ of Lie algebra automorphisms, called \emph{dilations}, by 
    \begin{equation*}
        \forall t>0,\quad \delta_t(V_x+R_x) = t V_x + t^2 R_x,\quad\mbox{where}\ 
        V_x \in \fv_x M,\, R_x \in \fr_x M.
    \end{equation*}
\end{definition}

The Lie algebra bundle gives rise to a Lie group bundle as follows. Each fiber $\fg_x M$, $x\in M$ is a nilpotent Lie algebra and, as such, is associated with a unique (up to isomorphism) nilpotent Lie group, denoted by $G_x M$, which is connected and simply connected. By nilpotency, the exponential mapping 
\begin{equation*}
    \Exp_{G_x M}: \fg_x M \rightarrow G_x M,
\end{equation*}
is a diffeomorphism between the Lie algebra and the associated Lie group.
We can thus realize the Lie group directly as $\fg_x M$, endowed with the group law given by the Baker-Campbell Hausdorff formula.

In this way, the bundle $GM = \bigcup_{x\in M} G_x M$ coincides as a manifold with $\fg M$ and is a smooth bundle of step 2 nilpotent Lie groups. It is called the \emph{tangent group bundle} of $(M,\eta)$.

\subsubsection{Local $\phi$-frames}
\label{subsubsect:phiframe}

We investigate how local $\phi$-frames introduced in Section~\ref{subsubsect:almostcontactstruct}, which are adapted frames, give nice trivializations of the tangent group bundle $GM$.
Indeed, recall that a local $\phi$-frame $(\bX,U)$ is a $g$-orthonormal frame on $U$, $$\bX = (X_1,\dots,X_d,Y_1,\dots,Y_d,R),$$  where the vector fields $(X_i)_{1\leq i\leq d}$ and $(Y_i)_{1\leq i\leq d}$ 
are tangent to the contact distribution $\cD$, and satisfy moreover 
\begin{equation*}
 \forall i,j \in \{1,\dots,d\},\quad d\eta(X_i,Y_j) = -\delta_{i,j}\quad\mbox{and}\quad d\eta(X_i,X_j) = d\eta(Y_i,Y_j) = 0.
\end{equation*}
Since for two vector fields $X$ and $Y$ tangent to $\cD$, we can write
\begin{equation*}
 d\eta(X,Y) = -\eta([X,Y]),
\end{equation*}
the previous equation gives 
\begin{equation*}
 \forall i,j \in \{1,\dots,d\},\quad [X_i,Y_j] = \delta_{i,j} R \mod \cD\quad\mbox{and}\quad [X_i,X_j],\, [Y_i,Y_j]\in\cD.
\end{equation*}

Then, the projection $\langle \bX\rangle$ of this adapted frame on $\fg_x M$, which we recall is given by 
\begin{equation*}
 \forall i \in \{1,\dots,d\},\quad \langle X_i \rangle_x = X_{i,x}\quad\mbox{and}\quad\langle Y_i \rangle_x = Y_{i,x},
\end{equation*}
seen as elements of $\fv_x M$, and by 
\begin{equation*}
 \langle R\rangle_x = R_x \mod \cD_x,
\end{equation*}
defines a frame for the osculating Lie algebras $\fg M$ over $U$, which satisfies moreover the following commutator relations: for all $i,j\in\{1,\dots,d\}$,
\begin{equation*}
 \forall x \in U,\quad [\langle X_i\rangle_x,\langle Y_j\rangle_x]_{\fg_x M} = \delta_{i,j} \langle R\rangle_x\quad\mbox{and}\quad [\langle X_i\rangle_x,\langle X_j\rangle_x]_{\fg_x M}= [\langle Y_i\rangle_x,\langle Y_j\rangle_x]_{\fg_x M}=0.
\end{equation*}

Adapted frames for which the projection on $\mathfrak{g}M$ satisfies exactly the Heisenberg Lie algebra relations~\eqref{eq:heisenbergrelations} are called \emph{Heisenberg frames}. Thus, local $\phi$-frames, similarly to local frames coming from Darboux coordinates, are Heisenberg frames.

\vspace{0.25cm}

We deduce in particular that the osculating Lie algebras bundle is a locally trivial Lie algebra bundle modeled on the Heisenberg Lie algebra $\fh^d$, i.e.~the map 
\begin{equation*}
 \varphi_{U}^{\bX}:\,U\times\fh^d \longrightarrow \fg M_{|U},
\end{equation*}
defined by
\begin{equation*}
 \varphi_{U}^{\bX}(x,v) = \left(x,\sum_{i=1}^d (v_i \langle X_i\rangle_x + v_{d+i}\langle Y_i\rangle_x)+v_{2d+1}\langle R\rangle_x\right),
\end{equation*}
is a diffeomorphism and a Lie algebra isomorphism between each fibers. 

Now, using the group exponential on each fiber $G_x M$, $x\in M$, we can consider the parameterization of $G_x M$ by the Heisenberg Lie algebra $\fh^d$ given by 
\begin{equation*}
    \forall v\in\fh^d,\quad \Exp_{x}^\bX(v) := \Exp_{G_x M}(\varphi_{U}^\bX(x,v)),
\end{equation*}
and we define the smooth diffeomorphism $\Exp^\bX: U\times\fh^d \rightarrow G_{|U}M$ by 
\begin{equation*}
    \forall (x,v)\in U\times\fh^d,\quad \Exp^\bX(x,v) := \Exp_{x}^\bX(v).
\end{equation*}

If we are given two such trivializations $\varphi_{U_1}^{\mathbb{X_1}}$ and $\varphi_{U_2}^{\mathbb{X_2}}$ associated with two local $\phi$-frames $(U_1,\mathbb{X_1})$ and $(U_2,\mathbb{X_2})$,
then, as discussed in Section~\ref{subsubsect:almostcontactstruct}, the transition between the two frames on $U_1\cap U_2$ is given by a family of matrices $(A(x))_{x\in U_1\cap U_2}$ in 
$U(d)\times 1$:
\begin{equation}
    \label{eq:transitionmaps}
    \forall x\in U_1\cap U_2,\quad \bX_1(x) = A(x)\bX_2(x).
\end{equation}
Because of their block-like structure, these transition matrices directly give the transition matrices between the adapted frames $\langle\bX_1\rangle$ and $\langle\bX_1\rangle$ in $\fg M_{|U_1\cap U_2}$.
We obtain 
\begin{equation*}
 \forall (x,v)\in (U_1\cap U_2,\fh^d),\quad \varphi_{U_1}^{\bX_1,-1}\circ \varphi_{U_2}^{\mathbb{X_2}}(x,v) = (x,A(x)v).
\end{equation*}

\subsubsection{Smooth Haar system}
\label{subsubsect:smoothhaarsys}

Observe that, since the transition matrices between two $\phi$-frames as defined by Equation~\eqref{eq:transitionmaps}, are in $U(d)\times 1$, the Riemannian structure on the tangent space $TM$ descends on the osculating Lie algebra bundle $\fg M$.

This remark allows us to say that the tangent group bundle $GM$ of a contact metric structure carries a natural \emph{smooth Haar system} $\mu_{GM} = \{\mu_x\}_{x\in M}$. Recall that a smooth Haar system (defined by \cite[Definition 5.2]{FFF24}) is a family of Haar measures $\mu_x$  on each fiber $G_x M$ such that the map $x\in M \mapsto \int_{M} f(x,h)d\mu_x(h)$ is smooth for any $f\in C_{c}^\infty(GM)$. 

\vspace{0.25cm}

Indeed, on the contact manifold $(M,\eta)$, choosing a local $\phi$-frame $(\bX,U)$, the pushforward by $\varphi_{U}^\bX$ of the constant measure system $\{{\rm Leb}_{\fh^d}^{\bX}\}_{x\in U}$, where ${\rm Leb}_{\fh^d}^{\bX}$ is the Lebesgue measure on $\fh^d$ in the basis $\langle\bX\rangle_x$, defines indeed a smooth measure system on $\fg_{|U}M$. Moreover, the previous discussion ensures that the resulting measure system is independent of the choice of local $\phi$-frames. Finally, using the group exponential $\Exp_{G_x M}$, $x\in M$, pushing forward this measure system defines uniquely a smooth Haar system, denoted simply by $\{\mu_x\}_{x\in M}$.

\subsubsection{Stratified spaces and higher order terms}

We recall here the notions of stratified vector space and their dilations.

\begin{definition}
    \label{def:stratspace}
    A finite-dimensional vector space $\cV$ is said to be \emph{stratified} with respect to the sequence $(\cV_{i})_{1\leq j\leq p}$ of subspaces, if it admits the decomposition
    \begin{equation*}
        \cV = \bigoplus_{1\leq j\leq p} \cV_j,
    \end{equation*}
    and $\cV_j \neq \{0\}$ for $1\leq j \leq p$.
\end{definition}

For a stratified vector space $\cV$, similarly as the dilations introduced by Definition~\ref{def:dilations}, we consider the family of automorphisms $(\delta_t)_{t>0}$ of $\cV$ defined by 
\begin{equation*}
    \delta_r(v) = r^j v,\quad v\in\cV_j.
\end{equation*}
The \emph{homogeneous dimension} of $\cV$ for this stratification is denoted by $Q$ and given by
\begin{equation*}
    Q := \sum_{j=1}^p j\, {\rm dim}(\cV_j). 
\end{equation*}

The dilations lead to the notion of homogeneous function: for $m\in\R$, a function $f:\cV \rightarrow \C$ is \emph{m-homogeneous} if $f(\delta_r v)= r^{m} f(v)$ for all $v\in\cV$ and $r>0$. This extends to distributions on $\cV$. A linear operator $T:\cS(\cV)\rightarrow \cS'(\cV)$ is also said to be $m$-homogeneous with respect to $(\delta_t)_{t>0}$ if for any $f\in\cS(\cV)$, we have $T(f\circ\delta_r) = r^m (Tf)\circ\delta_r$.

\begin{definition}
    \label{def:quasinorm}
    A \emph{quasinorm} on $\cV$ is a continuous map $|\cdot|:\cV\rightarrow \R$ which is symmetric, definite and 1-homogeneous with respect to $(\delta_t)_{t>0}$, i.e.~satisfying
    for all $v\in\cV$ and $r>0$,
    \begin{equation*}
        |-v|=|v|,\quad |v|=0 \iff v=0,\quad \mbox{and}\quad |\delta_r v| = r|v|.
    \end{equation*}
\end{definition}

As our primary example, we can take the Heisenberg Lie algebra $\fh^d$ and its stratification given by Equation~\eqref{eq:stratificationHeis}:
\begin{equation*}
    \fh^d = \fv\oplus\fr.
\end{equation*}
Then one can consider the \emph{Kor\'anyi quasinorm} given in the coordinates $(x,y,z)$ by 
\begin{equation*}
    |(x,y,z)| = \left((x^2+y^2)^2+z^2\right)^{1/4}.
\end{equation*}
Remark that this quasinorm makes sense on the fibers of $\fg M$ independently of the choice of local trivialization given by a local $\phi$-frame.

In the case of the Heisenberg Lie algebra, keeping the same notation for it, we extend the quasinorm to the Heisenberg group using the exponential map. Then, we have the following triangle inequality:
\begin{equation*}
    \exists C>0, \forall h_1,h_2\in\Heis,\quad |h_1 * h_2|\leq C(|h_1|+|h_2|).
\end{equation*}

We now define the notion of \emph{higher order map}. Let $\cV_1 = \bigoplus_{1\leq j\leq p_1}\cV_{1,j}$ and $\cV_2= \bigoplus_{1\leq j\leq p_2}\cV_{2,j}$ be two stratified vector spaces as in Definition~\ref{def:stratspace}, and denote by $(\delta_{1,r})_{r>0}$ and $(\delta_{2,r})_{r>0}$ their respective dilations. We suppose moreover that we fixed adapted bases $(V_1,\dots,V_{{\rm dim}(\cV_1)})$ and $(W_1,\dots,W_{{\rm dim}(\cV_2)})$ of $\cV_1$ and $\cV_2$ respectively, i.e.~that are concatenations of bases of the subspaces $\cV_{i,j}$, $i\in\{1,2\}$, $1\leq j\leq p_i$.

\begin{definition}
    \label{def:higherorder}
    A polynomial map $r:\cV_1\rightarrow\cV_2$ is higher order (with respect to the stratification) when the degrees of homogeneity (for the dilations on $\cV_1$) of the monomials occurring in its $j$-th component $r_j$, $1\leq j\leq p_2$, are $> j$. The notion of higher order does not depend on the choice of adapted bases.
\end{definition}

We extend the notion of higher order to smooth maps.

\begin{definition}
    \label{def:higherordersmooth}
    A smooth function $r$ defined on a neighborhood of $0$ in $\cV_1$ and valued in $\cV_2$ is also said to be higher order when its Taylor serie at $0$ is a serie of polynomial maps of higher order.
\end{definition}

We recall some characterizations of higher order maps given by \cite[Proposition 4.5]{FFF24}.

\begin{proposition}
    Let $r:\cV_1\rightarrow\cV_2$ a smooth map. The following properties are equivalent:
    \begin{itemize}
        \item[(i)] The map $r$ is higher order.
        \item[(ii)] For one (and then any) neighborhood $\cV_0$ of $0$, we have for any $v\in V_0$,
        \begin{equation*}
            \lim_{\eps\rightarrow 0} \delta_{2,\eps^{-1}}r(\delta_{1,\eps} v) = 0.
        \end{equation*}
        \item[(iii)] There exists a unique smooth map $R: \R_{+}\times\cV_1 \rightarrow \cV_2$ satisfying $\delta_{2,\eps^{-1}}(r(\delta_{1,\eps} v)) = \eps R(\eps,v)$, for any $v\in\cV_1$ and $\eps>0$.
    \end{itemize}
\end{proposition}

\subsubsection{Geometric exponentiation}

We present here the notion of geometric exponentiation for a local Heisenberg frame on $M$, as well as some of its properties. 

First, if $X$ denotes a vector field on $M$, we write $\exp_x(X)$ for the time-one flow, when defined, of $X$ starting from the point $x\in M$. More precisely, it corresponds to solution at time 1 of ordinary differential equation given by 
\begin{equation*}
    \Dot{\gamma}(t) = X(\gamma(t))\quad\mbox{and}\quad \gamma(0) = x,
\end{equation*}
i.e.~$\exp_x(X) := \gamma(1)$.

Given a local Heisenberg frame $(\bX,U)$ on $M$, with $\bX = (V_1,\dots,V_{2d+1})$, we set
for $x\in U$ and $v\in\R^{2d+1}$,
\begin{equation*}
    \exp_{x}^\bX(v) = \exp_x\left(\sum_{i=1}^{2d+1} v_i V_i\right),
\end{equation*}
when this makes sense. In what follows, we shall identify $\R^{2d+1}$ with the Heisenberg Lie algebra $\fh^d$ through the basis $(X_1,\dots,X_d,Y_1,\dots,Y_d,R)$ introduced in Equation~\eqref{eq:vectorfields}.

In fact, there exists an open neighborhood $U_\bX$ in $U\times \fh^d$ of the $(x,0)$ on which the following map, called the \emph{geometric exponential map},
\begin{equation*}
    \begin{aligned}
        \exp^\bX : U_\bX & \longrightarrow U \times U \\
        (x,v) & \longmapsto \exp_{x}^\bX(v),
    \end{aligned}
\end{equation*}
is a smooth diffeomorphism onto its image in $U\times U$. On suitable subsets of $U\times U$ and $U_\bX$, we introduce
\begin{equation*}
    \ln^\bX := (\exp^\bX)^{-1}\quad\mbox{and}\quad \ln_{x}^\bX := (\exp_{x}^\bX)^{-1}.
\end{equation*}

Without restriction, we will assume that the open set $U_\bX$ is such that for $(x,v)\in U_\bX$ and $t\in (0,1)$, $(x,tv)$ is also in $U_\bX$.

\vspace{0.25cm}

The following theorem, which is a special case of \cite[Theorem 4.13]{FFF24}, is the most important statement of this section, as it expresses the link between the geometric maps associated with local Heisenberg frames and the tangent group bundle.

\begin{theorem}
\label{thm:composgeomexp}
    Let $(\bX,U)$ be a local Heisenberg frame and $x\in U$. Then, on a neighborhood of $\{x\}\times\{0\}\times\{0\}$ in $U\times\fh^d\times\fh^d$, we have
    \begin{equation*}
        \ln_{x}^\bX\left(\exp_{\exp_{x}^\bX(w)}^\bX(v)\right) = w *_{\fh^d} v + r(x;w,v),
    \end{equation*}
    where $*_{\fh}$ denotes the Heisenberg group law lifted on $\fh^d$, and $r$ is a smooth map of higher order in $w$ and $v$ for the canonical stratification of $\fh^d$, valued in $\fh^d$.
\end{theorem}

\subsection{Harmonic analysis of nilpotent Lie groups}

We give here a reminder of the classical results of harmonic analysis on nilpotent Lie groups. Let $G$ be such a simply connected group. The notations we use in this section are inspired by \cite{FR}.

\subsubsection{Unitary dual set}
\label{subsubsect:unitarydualset}

First, recall that a \emph{unitary representation} $(\pi,\cH_\pi)$ of the group $G$ is a pair consisting of a separable Hilbert space $\cH_\pi$ and a strongly continuous
morphism $\pi$ from $G$ to the set of unitary operators on $\cH_\pi$. A representation is said to be \emph{irreducible} if the only closed subspaces of $\cH_\pi$ stable under the action of
$\pi$ are $\{0\}$ and $\cH_\pi$ itself. Two representations $\pi_1$ and $\pi_2$ are equivalent if there exists a unitary transformation $\mathcal{I}: \cH_{\pi_1}\rightarrow \cH_{\pi_2}$, 
called an \emph{intertwining map}, such that 
\begin{equation*}
 \pi_1 = \mathcal{I}^{-1}\circ \pi_2 \circ \mathcal{I}.
\end{equation*}

We denote by $[\pi_1]=[\pi_2]$ the associated equivalence class. This leads to the following definition.

\begin{definition}
 \label{def:dualset}
 The set of all irreducible unitary representations, up to equivalence, is called the \emph{(unitary) dual set} and is denoted by $\Ghat$:
 \begin{equation*}
 \Ghat := \{[\pi]\,:\, \pi:G \rightarrow \mathcal{L}(\cH_\pi)\ \textrm{irreducible s.c.~unitary representation of}\ G\}.
 \end{equation*}
\end{definition}

We then have the following structure on the dual set of a nilpotent group.

\begin{proposition}
 \label{prop:dualsettopology}
 The unitary dual set $\Ghat$ is naturally endowed with a topology, called the \emph{hull-kernel topology} or the \emph{Fell topology}, as well as with a unique (up to a multiplicative constant) measure $\mu_{\Ghat}$ on the corresponding Borel $\sigma$-algebra, called the \emph{Plancherel measure}. Moreover, this defines a standard Borel structure on $\Ghat$.
\end{proposition}

In practical situations, the representations of a nilpotent Lie group can be explicitly computed using \emph{Kirillov's theory of coadjoint orbits}. Indeed, there exists an explicit homeomorphism between the dual set $\Ghat$ and the space of coadjoint orbits $\fg^*/ G$. In this parameterization, one can then get an expression of the Plancherel measure. 

\vspace{0.25cm}

In the case of the Heisenberg group $\Heis^d$, the following description of its dual is obtained. The coadjoint orbits are of two kinds. The first ones are given by the planes $\fv^* \times \{\lambda\}$ for $\lambda \in \fr^{*}\setminus\{0\}$, while the remaining ones are given by the singleton $\{\gamma\}$ for $\gamma\in\fv^*$.

From this picture, we can get explicit expressions of the representations. First, fix $\lambda \in \fr^{*}\setminus\{0\}$.
We introduce the irreducible unitary representation $\pi^\lambda$ of $\Heis$ defined on $\cH_{\pi^\lambda} = L^2(\mathfrak{p})$ by 
\begin{equation}
\label{eq:schrodingerrep}
 \pi^\lambda(h)f(\xi) = \exp\left(i \lambda r + \frac{i}{2}|\lambda| x\cdot y + i\sqrt{|\lambda|}\xi\cdot y\right)f\left(\xi+\sqrt{|\lambda|}x\right),
\end{equation}
where identified $\fr^*\setminus\{0\}$ with $\R\setminus\{0\}$ through the vector $R$.

The representations $\pi^\lambda$, $\lambda\in\fr^{*}\setminus\{0\}$, called the \emph{Schrödinger representations}, are infinite dimensional and are not equivalent to each other.
Moreover, according to the Stone-von Neumann theorem, each representation $\pi^\lambda$ is characterized by the action of the center $\Exp_\Heis(\R R)$ on $\cH_{\pi^\lambda}$. More precisely, there exists one (and only one) irreducible representation
$(\pi,\cH_\pi)$ (up to equivalence) such that the elements of the form $\Exp_\Heis(r R)$, $r\in\R$ are sent by $\pi$ on 
\begin{equation*}
 e^{i\lambda r}I_{\cH_\pi}.
\end{equation*}

We also introduce another realization of these representations, called the \emph{Fock-Bargmann representations}. To this end, we define a complex structure on $\fv$ by considering the identification
\begin{equation*}
    \fv \sim \mathfrak{p}\oplus i\mathfrak{q}.
\end{equation*}
We introduce the variable $z=x+iy$ and we denote by $h=(z,r)$ the elements of the Heisenberg group.
We then define the Fock space $\mathscr{F}^\lambda(\fv)$ to be the space of functions $F:\fv\rightarrow\C$ that are holomorphic if $\lambda>0$ and antiholomorphic otherwise, satisfying
\begin{equation*}
 \int_\fv |F(z)|^2 e^{-|\lambda| |z|^2}dz<\infty.
\end{equation*}
The Heisenberg group then acts on the Fock spaces as follows:
\begin{equation}
    \label{eq:repfockbargmann}
    \pi^\lambda_\mathscr{F}(h) F(w) = \exp\left(i \lambda r - \frac{\lambda}{4}|z|^2-\frac{\lambda}{2}w \overline{z}\right) F(w+z),\quad h=(z,r)\in\Heis^d.
\end{equation}

By the Stone-von Neumann theorem, we easily deduce that
\begin{equation*}
    [\pi^\lambda] = [\pi^\lambda_\mathscr{F}],\quad \lambda\in\R^*.
\end{equation*}
For an explicit formula of the unitary operator that intertwines these two representations, we refer to \cite{Fol1}[Chapter 1]. The merit of the Fock-Bargmann representations in our later analysis is that the variables of the first stratum are treated on an equal footing.

There also exist other unitary irreducible representations of $\Heis$, which are given by the characters of the first stratum $\fv$ in the following way: for $\gamma\in\fv^*$, we set 
\begin{equation*}
 \pi^{0,\gamma}(h) = e^{i\gamma(V)},\quad\mbox{where}\ h = \Exp_\Heis(V+rR),\ \mbox{with}\ V\in\fv, r\in\R.
\end{equation*}
With these two classes of representations one gets the full dual of $\Heis$:
\begin{equation*}
 \Hhat = \{[\pi^\lambda]\,:\,\lambda\in\fr^{*}\setminus\{0\}\}\sqcup\{[\pi^{0,\gamma}]\,:\,\gamma\in\fv^*\}.
\end{equation*}
We will denote by $\Hhat_{\infty}$ the subset of $\Hhat$ made of the Schrödinger representations:
\begin{equation*}
 \Hhat_\infty = \{[\pi^\lambda]\,:\,\lambda\in\fr^{*}\setminus\{0\}\},
\end{equation*}
also called the \emph{generic dual set}.
We can deduce from the coadjoint picture that $\Hhat_\infty$ is a dense open subset of $\Hhat$ which is homeomorphic to $\R^*$.
Moreover, the Plancherel measure $\mu_{\Hhat}$ is supported on $\Hhat_\infty$ and with the choice of Haar measure made we have 
\begin{equation*}
 d\mu_{\Hhat}(\pi^\lambda) = (2\pi)^{-(3d+1)}|\lambda|^d d\lambda,\quad \lambda\in\R^*.
\end{equation*}

\subsubsection{Infinitesimal representations} For a nilpotent Lie group $G$ with Lie algebra $\fg$ and a representation $\pi$ of $G$, we keep the same notation for the corresponding infinitesimal representation that acts on the enveloping algebra $\mathfrak{U}(\fg)$ of the Lie algebra. It is characterized by its action on $\fg$:
\begin{equation*}
    \pi(X) = \frac{d}{dt}\pi(\Exp_G(t X))_{|t=0},\ X\in\fg.
\end{equation*}
Through the infinitesimal representation, the enveloping algebra is seen acting on the space $\cH_{\pi}^\infty$ of smooth vectors, defined as the space of vectors $f\in\cH_\pi$ such that the map $g\in G\mapsto \pi(g)f\in\cH_\pi$ is smooth.

For instance, in the case of the Heisenberg group $\Heis$, for a Schrödinger representation $\pi^\lambda$ given by~\eqref{eq:schrodingerrep}, with $\lambda\in\fr^*\setminus\{0\}$, one has the following expressions for the action of the basis $(X_1,\dots,X_d,Y_1,\dots,Y_d,R)$: for $i\in\{1,\dots,d\}$,
\begin{equation*}
    \pi^\lambda(X_i) = \sqrt{|\lambda|} \partial_{\xi_i},\quad 
    \pi^\lambda(Y_i) = i\sqrt{|\lambda|} \xi_i\quad\mbox{and}\quad
    \pi^\lambda(R) = i\lambda,
\end{equation*}
while $\cH_{\pi^\lambda}^\infty$ identifies with $\cS(\R_\xi)$.
In particular, if we introduce the left-invariant differential operator 
\begin{equation*}
    \Delta_{\Heis^d} = \sum_{i=1}^d (X_{i}^2 + Y_{i}^2),
\end{equation*}
called the \emph{Heisenberg subLaplacian}, then, seen as an element of the $\mathfrak{U}(\fh^d)$, we get 
\begin{equation*}
    \pi^\lambda(-\Delta_{\Heis^d}) = |\lambda|\sum_{i=1}^d (-\partial_{\xi_i}^2 + \xi_{i}^2).
\end{equation*}

\subsubsection{Fourier transform} 

For a general nilpotent Lie group $G$, with Haar measure simply denoted by $dx$, we first define the \emph{group Fourier transform} for a function $\kappa$ in $L^1(G)$:
for an irreducible representation $\pi$ of $G$, we set 
\begin{equation*}
 \cF_G(\kappa)(\pi) = \hat{\kappa}(\pi) = \pi(\kappa) = \int_G \kappa(x) \pi(x)^* dx\in \mathcal{L}(\cH_\pi).
\end{equation*}

Note that if $\pi_1$ and $\pi_2$ are equivalent, i.e.~there exists an intertwining operator $\mathcal{I}$ such that $\mathcal{I}\circ \pi_1 = \pi_2 \circ \mathcal{I}$, 
then we have $\mathcal{I}\circ \cF_G(\kappa)(\pi_1) = \cF_G(\kappa)(\pi_2)\circ \mathcal{I}$. Therefore, we may consider the group Fourier transform of $\kappa$ at the 
equivalence class of the representation $\pi$, and $\cF_G(\kappa)$ can be seen as a measurable field of equivalence classes of bounded operators:
\begin{equation*}
 \cF_G(\kappa) = \{\hat{\kappa}(\pi)\in\mathcal{L}(\cH_\pi)\,:\,\pi\in\Ghat\}.
\end{equation*}
Note that for two functions $\kappa_1$ and $\kappa_2$ in $L^1(G)$, we have
\begin{equation*}
 \cF_G(\kappa_1) \circ \cF_G(\kappa_2) = \cF_G(\kappa_2 \star_G \kappa_1),\quad\mbox{where}\ \kappa_2 \star_G \kappa_1(x) = \int_G \kappa_2(y)\kappa_1(y^{-1}x)\,dy.
\end{equation*}
The Plancherel measure $\mu_{\Ghat}$ introduced in Proposition~\ref{prop:dualsettopology} is characterized by the \emph{Fourier inversion formula}:
\begin{equation}
    \label{eq:fourierinversion}
    \forall \kappa \in \cS(G),\,\forall x\in G,\quad \kappa(x) = \int_{\Ghat}{\rm Tr}_{\cH_\pi}\left(\pi(x)\hat{\kappa}(\pi)\right) d\mu_{\Ghat}(\pi).
\end{equation}

Equivalently, it is also characterized by the following property:
\begin{equation}
\label{eq:plancherelformula}
 \forall \kappa\in L^1(G)\cap L^2(G),\quad \int_G |\kappa(x)|^2\,dx = \int_{\Ghat} \|\hat{\kappa}(\pi)\|_{HS(\cH_\pi)}^2 d\mu_{\Ghat}(\pi),
\end{equation}
where $\|\cdot\|_{HS(\cH_\pi)}$ denotes the Hilbert-Schmidt norm on $\cH_\pi$. We denote by $L^2(\Ghat)$ the Hilbert space of measurable fields of Hilbert-Schmidt 
operators $$\sigma = \{\sigma(\pi)\in\mathcal{L}(\cH_\pi)\,:\,\|\sigma(\pi)\|_{HS(\cH_\pi)}<\infty,\  \pi\in\Ghat\},$$ (up to $\mu_{\Ghat}$-a.e.~equivalence) such that the following quantity is finite 
\begin{equation*}
 \|\sigma\|_{L^2(\Ghat)}^2 := \int_{\Ghat} \|\sigma(\pi)\|_{HS(\cH_\pi)}^2 d\mu_{\Ghat}(\pi).
\end{equation*}
Equation~\eqref{eq:plancherelformula} allows for a unitary extension of the group Fourier transform 
\begin{equation*}
 \cF_G: L^2(G)\rightarrow L^2(\Ghat).
\end{equation*}

However, it is still not a general enough space of definition of the group Fourier transform for our needs. To this end, we introduce the following space.

\begin{definition}
 \label{def:Linfty}
 The Banach space $L^\infty(\Ghat)$ is defined as the space of measurable field of bounded operators $\sigma= \{\sigma(\pi)\in\mathcal{L}(\cH_\pi)\,:\,\pi\in\Ghat\}$ 
 (up to $\mu_{\Ghat}$-a.e.~equivalence) on $\Ghat$ such that the following quantity is finite:
 \begin{equation*}
 \|\sigma\|_{L^\infty(\Ghat)} := \sup_{\pi\in\Ghat}\|\sigma\|_{\mathcal{L}(\cH_\pi)}.
 \end{equation*} 
\end{definition}

The supremum is understood as the essential supremum with respect to the Plancherel measure $\mu_{\Ghat}$. We check readily that $\cF_G(L^1(G)) \subset L^\infty(\Ghat)$, but this 
is not an equality. Indeed, it follows from Dixmier's full Plancherel theorem (see, e.g., \cite[Theorem B.2.32]{FR}) that $L^\infty(\Ghat)$ is a von Neumann algebra isomorphic to the subspace 
\begin{equation*}
 \mathcal{L}(L^2(G))^G,
\end{equation*}
of bounded operators on $L^2(G)$ commuting with left-translations. More precisely, for any $T\in \mathcal{L}(L^2(G))^G$, there exists a unique element $\hat{T}\in L^\infty(\Ghat)$ such
that 
\begin{equation*}
 \forall f\in L^2(G),\quad \cF_G(Tf) = \hat{T}\hat{f},
\end{equation*}
and moreover $\|T\|_{\mathcal{L}(L^2(G))} = \|\hat{T}\|_{L^\infty(\Ghat)}$.

From the Schwartz kernel theorem, it follows that for any $T\in\mathcal{L}(L^2(G))^G$, there exists a unique distribution $\kappa\in\mathcal{S}'(G)$ such that $Tf = f\star_G \kappa$ for $f\in\mathcal{S}(G)$.
The distribution $\kappa$ is called the (right) convolution kernel of $T$. By setting 
\begin{equation*}
 \cF_G(\kappa) := \hat{T},
\end{equation*}
we easily check that it extends the group Fourier transform to the space
\begin{equation*}
 \mathcal{K}(G) := \{\kappa\in\mathcal{S}'(G)\,:\,(f\mapsto f\star_G\kappa)\in\mathcal{L}(L^2(G))\},
\end{equation*}
and that $\cF_G:\mathcal{K}(G)\rightarrow L^\infty(\Ghat)$ is bijective.

\subsubsection{Further extension of the Fourier transform}
\label{subsubsect:extensionFourier}

We finish our presentation of the group Fourier transform in the particular case of the Heisenberg group.

First, for $s\in\R$, we introduce the notation $L_s^2(\Heis^d)$, the adapted Sobolev space on $\Heis^d$ as defined in \cite[Section 4.4]{FR}.

We now consider the following spaces of tempered distributions on $\Heis^d$: for $a,b \in\R$, we set
\begin{equation*}
    \cK_{a,b}(\Heis^d) := \{\kappa\in\cS'(\Heis^d)\,:\,(f\mapsto f\star \kappa)\in\mathcal{L}(L_{a}^2(\Heis^d),L_{b}^2(\Heis^d))\}.
\end{equation*}

The Schwartz kernel theorem readily gives that the space $\cK_{a,b}(\Heis^d)$ identifies with the Banach subspace of $\mathcal{L}(L_{a}^2(\Heis^d),L_{b}^2(\Heis^d))$ consisting of the left-invariant operators, i.e.~commuting with left-translations.

At the same time, we also define the space $L_{a,b}^\infty(\Hhat^d)$ as the space of measurable fields $\sigma = \{\sigma(\pi): \cH_{\pi}^{\infty}\rightarrow \cH_{\pi}^{\infty}\,:\,\pi\in\Hhat\}$ such that the measurable field
\begin{equation*}
    \left(1-\pi(\Delta_{\Heis^d})\right)^{b/2} \sigma \left(1-\pi(\Delta_{\Heis^d})\right)^{-a/2},
\end{equation*}
is an element of $L^\infty(\Hhat^d)$. The space $L_{a,b}^\infty(\Hhat^d)$ is a Banach space when endowed with the following norm
\begin{equation*}
    \|\sigma\|_{L_{a,b}^\infty(\Hhat^d)} = \| \left(1-\pi(\Delta_{\Heis^d})\right)^{b/2} \sigma \left(1-\pi(\Delta_{\Heis^d})\right)^{-a/2}\|_{L^\infty(\Hhat^d)}.
\end{equation*}

\begin{proposition}
    For $a,b\in\R$, the Banach spaces $\cK_{a,b}(\Heis^d)$ and $L_{a,b}^\infty(\Hhat^d)$ are isomorphic. More precisely, the Fourier transform $\cF_{\Heis^d}$ extends uniquely to an isomorphism $\cF_{\Heis^d}: \cK_{a,b}(\Heis^d) \rightarrow L_{a,b}^\infty(\Hhat^d)$ of Banach spaces, which also satisfies, for any $\kappa\in\cK_{a,b}(\Heis^d)$,
    \begin{equation*}
        \|\kappa\|_{\cK_{a,b}(\Heis^d)} = \|\cF_\Heis(\kappa)\|_{L_{a,b}^\infty(\Hhat^d)}.
    \end{equation*}
\end{proposition}

\subsubsection{Invariant classes of symbols}

We end this section on the harmonic analysis of the Heisenberg group by defining the notion of symbols that we will use in the following sections.

\begin{definition}
    \label{def:invsym}
    An \emph{invariant symbol} $\sigma$ on $\Hhat^d$ is a measurable field of operators $\sigma = \{\sigma(\pi):\cH_{\pi}^{+\infty}\rightarrow \cH_{\pi}^{-\infty}\,:\,\pi\in\Hhat^d\}$ over $\Hhat^d$, where we recall that $\cH_{\pi}^{+\infty}$ and $\cH_{\pi}^{-\infty}$ denote the spaces of smooth and distributional vectors on $\cH_\pi$ respectively.
\end{definition}

To define our classes, we introduce the notion of \emph{difference operators}.

\begin{definition}
    \label{def:diffop}
    Let $q\in C^\infty(\Heis^d)$ and $\sigma$ be an invariant symbol on $\Hhat^d$. We say that $\sigma$ is $\Delta_q$-differentiable when $\sigma$ is in $L_{a,b}^\infty(\Hhat^d)$ and $q \cF_{\Heis^d}^{-1}(\sigma)$ is in $L_{c,d}^\infty(\Hhat^d)$ for some $a,b,c,d \in \R$. We denote then 
    \begin{equation*}
        \Delta_q \sigma = \cF_{\Heis^d} \left(q \cF_{\Heis^d}^{-1}(\sigma)\right).
    \end{equation*}
 \end{definition}

Using the coordinates $(x,y,z)$ on $\Heis^d$ introduced in Section~\ref{subsubsect:heisenberg}, we consider the difference operators associated with the monomials $x,y$ and $z$, denoted respectively by $\Delta_x$, $\Delta_y$ and $\Delta_z$. For an index $\alpha \in \N^{2d+1}$, we define $\Delta^\alpha$ to be the difference operators associated with the smooth function $x_{1}^{\alpha_1}\cdots x_{d}^{\alpha_d}y_{1}^{\alpha_{d+1}}\cdots y_{d}^{\alpha_{2d}}z^{\alpha_{2d+1}}$.

We can now define our classes of symbols.

\begin{definition}
    \label{def:invsymbol}
    Let $m\in\R$. A invariant symbol $\sigma$ is said to be \emph{of order $m$}, or equivalently in the class $S^m(\Hhat^d)$, if for any index $\alpha\in \N^{2d+1}$,
    \begin{equation*}
        \Delta^\alpha \sigma \in L_{0,-m+[\alpha]}^\infty(\Hhat^d),
    \end{equation*}
    where $[\alpha] = \sum_{i=1}^{2d}\alpha_i + 2\alpha_{2d+1}$.
\end{definition}

The space $S^m(\Hhat^d)$ is a Fréchet space when endowed with the seminorms given by 
\begin{equation*}
    \|\sigma\|_{S^m(\Hhat^d),N} := \max_{[\alpha]\leq N} \|\Delta^\alpha \sigma\|_{L_{0,-m+[\alpha]}^\infty(\Hhat^d)}
    = \max_{[\alpha]\leq N}\left\|\left(1-\pi(\Delta_{\Heis^d})\right)^{\frac{m-[\alpha]}{2}}\Delta^\alpha \sigma\right\|_{L^\infty(\Hhat^d)}.
\end{equation*}

\begin{definition}
\label{def:smoothingsym}
    The class of \emph{smoothing symbols} is defined as the following intersection
    $$S^{-\infty}(\Hhat^d) = \bigcap_{m\in\R}S^m(\Hhat^d).$$ 
    It is endowed with the induced topology of projective limit.
\end{definition} 

Then, the results recalled above, or equivalently \cite[Lemma 5.4.13 and Theorem 5.4.9]{FR}, imply the following:
\begin{proposition}
    \label{prop:fourierschwartz}
    The group Fourier transform is an isomorphism between the topological vector spaces $\cS(\Heis)$ and $S^{-\infty}(\Hhat)$.
\end{proposition}

\subsection{Dual tangent bundle}

Following \cite{FFF24}, we introduce the dual tangent bundle and investigate its properties when associated with a compact manifold $(M^{2d+1},\eta,g)$ with a contact metric structure. Recall that this bundle is defined by
\begin{equation}
    \label{eq:dualtangentbundle}
 \Ghat M = \bigcup_{x\in M} (x,\widehat{G_x M}).
\end{equation} 
In the rest of the article, we will rather use the notation $\Ghat_x M$ for the fibers of $\Ghat M$. 

\subsubsection{Topology and local parameterization}

We can give this set a natural topology as follows. We introduce on $\fg^* M$, the dual bundle to the vector bundle $\fg M$, the equivalence relation $\sim_{Kir}$ defined by 
\begin{equation*}
    \forall (x_1,\ell_1),(x_2,\ell_2)\in \fg^* M,\quad
    (x_1,\ell_1)\sim_{Kir}(x_2,\ell_2) \iff x_1 = x_2\ \mbox{and}\ \ell_1 \in {\rm Ad}^*(G_{x_2} M)(\ell_2),
\end{equation*}
where ${\rm Ad}^*(G_{x_2} M)(\ell)$ for $\ell \in \fg_{x}^* M$, denotes the coadjoint orbits of $\ell$ in $\fg_{x}^* M$.
We redefine then the dual tangent bundle as the topological quotient space associated with this relation:
\begin{equation*}
    \Ghat M := \fg^* M /\sim_{Kir}.
\end{equation*}
This is coherent with Equation~\eqref{eq:dualtangentbundle}: denoting by $p:\fg^* M\rightarrow M$ the natural projection, this application descends on $\Ghat M$ as a continuous onto map
\begin{equation*}
    \hat{p}:\Ghat M \rightarrow M.
\end{equation*}
Moreover, for $x\in M$, we observe, thanks to the coadjoint picture of the unitary dual sets, that the fiber $\hat{p}^{-1}(x)$ is homeomorphic to $\Ghat_x M$, as desired.

\vspace{0.25cm}

Thanks to the existence of the local $\phi$-frames, it is possible to give parameterizations of $\Ghat M$.
Indeed, let $\bX=(X_1,\dots,X_d,Y_1,\dots,Y_d,R)$ be such a local $\phi$-frame on a neighborhood $U$, and $\varphi_{U}^\bX$ be the local trivialization of 
$\fg M$ associated with the adapted frame $\langle\bX\rangle$. 
Since for any $x\in U$, the frame $\langle\bX\rangle_x$ satisfies exactly the commutator relations of the Heisenberg Lie algebra, the representations of the associated 
nilpotent Lie group $G_x M$ are given by the one of $\Heis$. Explicitly, for any $\lambda \in \fr_{x}^* M\setminus o$ (where $\fr^* M$ stands for the dual bundle associated with $\fr M$ and $o$ denotes the null section), 
introducing the vector space
\begin{equation*}
 \mathfrak{p}^{\bX}_x = {\rm span}(\langle X_1\rangle_x,\dots,\langle X_d\rangle_x) \subset \fg_x M,
\end{equation*}
we define the irreducible unitary representation $\pi_{x}^\lambda$ on the Hilbert space $\cH_{\pi_{x}^\lambda} = L^2(\mathfrak{p}^{\bX}_x)$ by
\begin{equation}
\label{eq:repgrouptangent}
 \forall f\in\cH_{\pi_{x}^\lambda},\quad \pi_{x}^\lambda(h)f(\xi) = \exp\left(i \lambda z + \frac{i}{2}|\lambda| x\cdot y + i\sqrt{|\lambda|}\xi\cdot y\right)f\left(\xi+\sqrt{|\lambda|}x\right),
\end{equation}
where, for $h\in G_x M$, we wrote 
\begin{equation*}
 h = \Exp_{G_x M}(x\cdot\langle X\rangle_x + y\cdot \langle Y\rangle_x + z\langle R \rangle_x),\quad\mbox{with}\ x,y\in\R^d,\,z\in\R.
\end{equation*}
Similarly, for $\gamma\in \fv_{x}^* M$, we set 
\begin{equation*}
 \pi_{x}^{0,\gamma}(h) = e^{i\gamma(\langle V\rangle_x)},
\end{equation*}
where
\begin{equation*}
 h=\Exp_{G_x M}(\langle V\rangle_x + z\langle R\rangle_x),\ \mbox{with}\ \langle V\rangle_x\in\fv_x M, z\in\R.
\end{equation*}

We are thus led to introduce the map
\begin{equation*}
 \hat{\varphi}_{U}^\bX: U\times \Hhat \rightarrow \Ghat M_{|U},
\end{equation*}
defined as follows: for $x\in U$ and $\pi\in \Hhat$, if $[\pi] = [\pi^\lambda]$ for some $\lambda\in \fr^{*}\setminus\{0\}$, identifying $\fr^*$ and $\fr_{x}^* M$ through the dual basis of the vector $\langle R\rangle_x$, we set
\begin{equation*}
 \hat{\varphi}_{U}^\bX(x,\pi^\lambda) = (x,\pi_{x}^\lambda),
\end{equation*}
while if $[\pi] = [\pi^{0,\gamma}]$ for some $\gamma\in \fv^{*}$, identifying $\fv^*$ and $\fv_{x}^* M$ through the dual basis of $(\langle X_1\rangle_x,\dots,\langle X_d\rangle_x,\langle Y_1\rangle_x,\dots,\langle Y_d\rangle_x)$, we set
\begin{equation*}
 \hat{\varphi}_{U}^\bX(x,\pi^{0,\gamma}) = (x,\pi_{x}^{0,\gamma}).
\end{equation*}
Using the coadjoint orbit pictures of the unitary dual sets, we check readily that the map $\hat{\varphi}_{U}^\bX$ is an homeomorphism, making of $\Ghat M$ a locally trivial fibered space with fibers homeomorphic to $\Hhat^d$.

\subsubsection{Characteristic manifold and generic dual tangent bundle}
\label{subsubsect:charmanifold}

Keeping the same notation as in the previous section, if $(U_1,\mathbb{X_1})$ and $(U_2,\mathbb{X_2})$ are two overlapping local $\phi$-frames, then the map $\hat{\varphi}_{1,2}^\bX := \hat{\varphi}_{U_1}^{\bX_1,-1}\circ \hat{\varphi}_{U_2}^{\bX_2}: (U_1\cap U_2)\times \Hhat \rightarrow (U_1\cap U_2)\times \Hhat$ can be made explicit as follows.
For $x \in U_1\cap U_2$ and $\pi\in\Hhat$, we write $\tilde{\pi}\in\Hhat$ the representation given by 
\begin{equation*}
 \hat{\varphi}_{1,2}^\bX(x,\pi) = (x,\tilde{\pi}).
\end{equation*}
Suppose first that $\pi$ is in the same equivalence class as $\pi^\lambda$ for some $\lambda\in \fr^{*}\setminus\{0\}$. Then, as remarked in Section~\ref{subsubsect:unitarydualset}, to identify the equivalence class
of $\tilde{\pi}$, it is enough to observe the action of the center. Since the spaces $\fr^*$ and $\fr_{x}^* M$ are identified through the vector $\langle R\rangle_x \in\fr_x M$ and that this vector is identical in both 
the adapted frames $\langle \bX_1\rangle_x$ and $\langle \bX_2\rangle_x$, we easily deduce that for $r\in\R$, we have 
\begin{equation*}
 \tilde{\pi}(\Exp_\Heis(r \langle R\rangle_x)) = e^{i\lambda r}I_{\cH_{\tilde{\pi}}}.
\end{equation*}
Thus, we have $[\tilde{\pi}] = [\pi^\lambda]$ and the map $\hat{\varphi}_{1,2}^\bX(x,\cdot)$ is the identity on $\Hhat_\infty$.

We deduce that one can talk about the generic dual set of the groups $G_x M$ independently of the choice of the $\phi$-frame and that the transition maps are given by the identity on $\Hhat_\infty$. Similarly as in Section~\ref{subsubsect:unitarydualset}, we denote the resulting subset of $\Ghat_x M$ by $\Ghat_{\infty,x} M$. Considering only these subsets, we get the following definition.

\begin{definition}
    We define the \emph{generic dual tangent bundle}, denoted by $\Ghat_\infty M$, as the topological subset of $\Ghat M$ given by 
    \begin{equation*}
        \Ghat_\infty M := \{(x,[\pi])\,:\,[\pi]\in \Ghat_{\infty,x} M\}.
    \end{equation*}
    Moreover, $\Ghat_\infty M$ is open and dense in $\Ghat M$, and homeomorphic to $M\times \Hhat_\infty$.
\end{definition}


From these considerations, it follows that the generic dual tangent bundle can be endowed with a smooth structure, since it is naturally identified with the characteristic manifold $\Sigma$ of the contact manifold $(M,\eta)$. Using the parameterization $j$ of $\Sigma$ by $M\times\R\setminus\{0\}$ given at the end of Section~\ref{subsubsect:generalitiescontact}, it results that the map
\begin{equation}
    \label{eq:sigmagenericdual}
    (x,\lambda\eta_x)\in \Sigma \mapsto (x,[\pi_{x}^\lambda])\in \Ghat_\infty M,
\end{equation}
is a homeomorphism and we ask it to be a diffeomorphism.

\vspace{0.25cm}

Finally, the parameterizations $\hat{\varphi}_U^{\bX}$ also allow us to define a natural \emph{smooth Plancherel system}. 

\begin{definition}
    \label{def:smoothPlancherel}
    A smooth Plancherel system $\{\hat{\mu}_{\Ghat_x M}\}_{x\in M}$ for $\Ghat M$ is a choice of Plancherel measure for each of the dual $\Ghat_x M$ for $x\in M$ such that the map $x\mapsto \int_{\Ghat_x M} f(x,\pi)\,d\hat{\mu}_{\Ghat_x M}(\pi)$ is smooth for any function $f\in C_c(\Ghat M)$ which is smooth with respect to the base $M$.
\end{definition}

A function $f\in C_c(\Ghat M)$ is said to be smooth with respect to the base $M$ if and only if, when pulled back on $M\times \Hhat^d$ by any local parametrization $\hat{\varphi}_U^\bX$, it is smooth with respect to the first variable.

Using an atlas of $\phi$-frame, it is possible to define such a smooth Plancherel system by the use of the local parametrization $\hat{\varphi}_U^\bX$ and by pushing forwards the measure $d\nu(x)\otimes d\hat{\mu}_{\Hhat^d}(\pi)$ defined on $M\times \Hhat^d$. The previous discussion shows that such a definition makes sense since on the overlaps $U_1\cap U_2$, the transition maps $\hat{\varphi}_{1,2}^\bX(x,\cdot)$ are the identity on $\Hhat^d_\infty$, which is the support of the Plancherel measure $\hat{\mu}_{\Hhat^d}$. We will denote by $\hat{\mu}_{\Ghat M}$ the resulting Plancherel system and, by the same notation, the measure it induces on $\Ghat M$.

Moreover, this smooth Plancherel system is \emph{compatible} with the smooth Haar system $\mu_{GM}$ defined in Section~\ref{subsubsect:smoothhaarsys}, in the following sense: for any $x\in M$ and for any function $\kappa \in \mathcal{S}(G_x M)$, we have
\begin{equation}
    \label{eq:compatiblesystems}
    \int_{G_x M} |\kappa(v)|^2\,d\mu_{G_x M}(v) = \int_{\Ghat_x M} \|\hat{\kappa}(\pi)\|_{HS(\mathcal{H}_\pi)}^2\,d\hat{\mu}_{\Ghat_x M}(\pi).
\end{equation}

Restricted to the generic dual tangent bundle $\Ghat_\infty M$, the measure $\hat{\mu}_{\Ghat M}$ can be expressed as follows using the identification given by Equation~\eqref{eq:sigmagenericdual}:
\begin{equation*}
    d\hat{\mu}_{\Ghat M}(x,\pi^\lambda) = \frac{1}{(2\pi)^{3d+1}} d\nu(x)\otimes |\lambda|^dd\lambda.
\end{equation*}

\begin{remark}
    Observe that, up to a multiplicative constant, the measure $\hat{\mu}_{\Ghat M}$ coincides with the symplectic volume $|\omega_{|\Sigma}^{d+1}|$ described by Equation~\eqref{eq:symplecticvolume}.
\end{remark}

\subsubsection{Fock Hilbert bundle} 

The previous section allows us to think of the generic dual tangent bundle $\Ghat_\infty M$ as the characteristic manifold $\Sigma$, and thus as a global geometric object rather than only through local parameterizations. However, this is not the case for the representations defined by Equation~\eqref{eq:repgrouptangent} through the use of local $\phi$-frames, the reason being that one makes the arbitrary choice of the Lagrangian decomposition $\fv_x M = \mathfrak{p}_x^\bX \oplus\mathfrak{q}_x^\bX$.

To remedy this issue, one can use the Fock-Bargmann representations of the Heisenberg group defined by Equation~\eqref{eq:repfockbargmann}. More precisely, identifying $\fv_x M$ with $\mathfrak{p}_x^\bX \oplus i\mathfrak{q}_x^\bX = \C^d$ thanks to the $\phi$-frame $\bX$, we consider the Hilbert space $\mathscr{F}^\lambda(\mathfrak{v}_x M)$, defined as the space of functions $F:\mathfrak{v}_x M\rightarrow\C$ that are holomorphic if $\lambda >0$ and anti-holomorphic if $\lambda <0$, and such that 
\begin{equation*}
     \int_{\mathfrak{v}_x M} |F(z)|^2 e^{-\frac{|\lambda|}{2}|z|^2}dz < +\infty.
\end{equation*}
We then consider the representations of $G_xM$ defined by: for all $F\in\mathscr{F}^\lambda(\mathfrak{v}_x M)$,
\begin{equation*}
    \pi^\lambda_{x,\mathscr{F}}(g) F(w) = \exp\left(i \lambda r - \frac{\lambda}{4}|z|^2-\frac{\lambda}{2}w \overline{z}\right) F(w+z),\quad w\in\mathfrak{v}_x M = \mathfrak{p}_x M\oplus i \mathfrak{q}_x M,
\end{equation*}
where, for $g\in G_x M$, we wrote $z = x+iy$ with
\begin{equation*}
 g = \Exp_{G_x M}(x\cdot\langle X\rangle_x + y\cdot \langle Y\rangle_x + r\langle R \rangle_x),\quad\mbox{with}\ x,y\in\R^d,\,r\in\R.
\end{equation*}

However, the expression of this representation can be rewritten so as to be independent of the choice of the $\phi$-frame $\bX$ and thus of the splitting $\fv_x M = \mathfrak{p}_x^\bX \oplus\mathfrak{q}_x^\bX$.

First, for any point $\sigma = (x,\lambda\eta_x)\in\Sigma \sim \Ghat_\infty M$, we identify the fiber $\mathfrak{v}_x M$ with the fiber $N_\sigma \Sigma$ through the isomorphism $\Xi_\sigma$ defined by Equation~\eqref{eq:defXi}. Through this isomorphism, both the endomorphism field $\phi$ and the metric $g$ can be lifted to $\phi^{N\Sigma}$ and $g^{N\Sigma}$ respectively on the bundle $N \Sigma$ in the following way: we set
\begin{equation*}
    \phi_\sigma^{N\Sigma} = {\rm sgn}(\lambda(\sigma))\,\Xi^*(\phi_x),
\end{equation*}
as well as
\begin{equation}
    g_\sigma^{N\Sigma} = \lambda(\sigma) \Xi^*(g_x),
\end{equation}
which is a positive bilinear form when $\lambda(\sigma) >0$ and negative otherwise.

We introduce the following sesquilinear form (with respect to the complex structure $\phi_\sigma^{N\Sigma}$) on $N_\sigma \Sigma$: for all $V,W\in N_\sigma \Sigma$, we set
\begin{equation*}
    h_\sigma^{N\Sigma}(V,W) := g_\sigma^{N\Sigma}(V,W) + i \omega_\sigma^{N\Sigma}(V,W),
\end{equation*}
where the symplectic form $\omega_\sigma^{N\Sigma}$ was defined in Section~\ref{subsubsect:symnormbundle}.
Similarly to $g^{N\Sigma}$, $h^{N\Sigma}$ is positive on $\lambda >0$ and negative otherwise. 

Then, writing again $z = x+iy$ and $V = x\cdot\langle X\rangle_x + y\cdot \langle Y\rangle_x$, it is straightforward to check that by construction of the $\phi$-frame, we have 
\begin{equation*}
    \lambda(\sigma)|z|^2 = h_\sigma^{N\Sigma}(V,V) \quad\mbox{and}\quad \lambda(\sigma) w \Bar{z}  = h_\sigma^{N\Sigma}(W,V),
\end{equation*}
where $w = p+iq$ and $W =p\cdot\langle X\rangle_x + q\cdot \langle Y\rangle_x$.

Thus, the Fock space $\mathscr{F}^\lambda(\mathfrak{v}_x M)$ coincides with the space
\begin{multline*}
    \mathscr{F}(N_\sigma \Sigma) := \biggl\{F:N_\sigma\Sigma\rightarrow\C\,:\, F\ \mbox{is}\ \phi_\sigma^{N\Sigma}\mbox{-holomorphic}\\ \mbox{ and }\int_{N_\sigma\Sigma} |F(W)|^2 e^{-\frac{1}{2}|g_\sigma^{N\Sigma}(W,W)|}d{\rm Leb}_\sigma^{N\Sigma}(W) < +\infty\biggr\},
\end{multline*}
where ${\rm Leb}_\sigma^{N\Sigma}$ denotes the Lebesgue measure on $N_\sigma\Sigma$ associated with the metric $g_\sigma^{N\Sigma}$. Then, the representation $\pi_{x,\mathscr{F}}^\lambda$ of the group $G_x M$ coincides with the representation $\pi_{\sigma,\mathscr{F}}$ defined as follows:
\begin{equation}
    \label{eq:repfocktgtgroup}
    \pi_{\sigma,\mathscr{F}}(g) F(W) = \exp\left(i\lambda(\sigma) r - \frac{1}{2}h_\sigma^{N\Sigma}(W+\frac{1}{2}V,V)\right) F(W + V),
\end{equation}
where $g = \Exp_{G_x M} (V+ r\langle R\rangle)$ with $V\in \mathfrak{v}_x M$ and $r\in\R$. Observe that this formula does not depend on the choice of a $\phi$-frame.

\begin{definition}
    We denote by $\mathscr{F}(N\Sigma)$ the collection of all these Fock spaces:
    \begin{equation*}
        \mathscr{F}(N\Sigma) := \bigcup_{\sigma\in\Sigma}\mathscr{F}(N_\sigma \Sigma),
    \end{equation*}
    called the \emph{Fock Hilbert bundle}.
\end{definition}

Since $N\Sigma$ is a smooth bundle over $\Sigma$ endowed with a $SU(d)$-group structure, it is straightforward to prove that the set $\mathscr{F}(N\Sigma)$ is a smooth Hilbert bundle in the sense of \cite{Lang}.
Finally, the results of the previous sections and Equation~\eqref{eq:repfocktgtgroup} allow us to write
\begin{equation}
    \label{eq:fockrealisation}
    \mathscr{F}(N\Sigma) \sim \int_{\Ghat_\infty M}^\oplus \cH_{\pi_x} d\mu_{\Ghat M}(x,\pi_x).
\end{equation}

\section{Semiclassical pseudodifferential calculus}
\label{sect:semipseudocalc}

\subsection{Symbol classes} 
\label{subsect:symbolclasses}
To begin with, first recall that we defined at the end of Section~\ref{subsubsect:phiframe} a natural smooth Haar system denoted by $\{\mu_x\}_{x\in M}$, as well as a compatible smooth Plancherel system in Section~\ref{subsubsect:charmanifold}. This allows us to define unequivocally the spaces $L_{a,b}^\infty(\Ghat_x M)$, $x\in M$, for $a,b\in\R$, as in Section~\ref{subsubsect:extensionFourier}, as well as their associated norms $\|\cdot\|_{L_{a,b}^\infty(\Ghat_x M)}$.

\begin{definition}
    \label{def:symbols}
    A \emph{symbol} $\sigma= \{\sigma_\hbar\}_{\hbar>0}$ on $\Ghat M$ is a collection of invariant symbols 
    \begin{equation*}
       \sigma(x,\cdot) = \{\sigma(x,\pi):\cH_{\pi}^{\infty}\rightarrow\cH_{\pi}^{\infty}\,:\,\pi\in\Ghat_x M\},
    \end{equation*}
    on $\Ghat_x M$, parameterized by $x\in M$ and depending on the semiclassical parameter $\hbar$. 
    We also say that $\sigma$ admits a \emph{convolution kernel} if for any $x\in M$, $\sigma(x,\cdot)\in L_{a,b}^\infty(\Ghat_x M)$ for some $a,b\in\R$.
\end{definition}

If a symbol $\sigma$ on $\Ghat M$ admits a convolution kernel, then for $x\in M$, we denote by
\begin{equation*}
    \kappa_{\sigma,x} = \cF_{G_x M, \mu_x}^{-1}(\sigma(x,\cdot)),
\end{equation*}
the corresponding tempered distribution $\kappa_{\sigma,x}\in \cS'(G_x M)$ given by the inverse Fourier transform taken with respect to the Plancherel measure $\hat{\mu}_{\Ghat_x M}$. If a local Heisenberg frame $(\bX,U)$ is given, by the group exponential $\Exp_{G_x M}$, we can consider the corresponding distribution $\kappa_{\sigma,x}^{\bX}$ on $\fh^d$:
\begin{equation}
    \label{eq:localassociatedkernel}
    \kappa_{\sigma,x}^{\bX}(v) := \kappa_{\sigma,x}\left(\Exp_{x}^{\bX}(v)\right) \in\cS'(\fh^d).
\end{equation}

Such a collection of distributions can, in turn, be seen as a family of invariant symbols on $\Hhat^d$. We introduce thus the notation
\begin{equation}
\label{eq:symframe}
    \forall x \in U,\ \sigma^\bX(x,\cdot) := \cF_{\Heis^d}(\kappa_{\sigma,x}^\bX).
\end{equation}

\begin{definition}
\label{def:symclassGhatM}
    For $m\in\R$, the class of symbols $S^m(\Ghat M)$ is defined as the set of symbols $\sigma = \{\sigma_\hbar\}_{\hbar>0}$ on $\Ghat M$ admitting convolution kernels and such that, for any local Heisenberg frame $(\bX,U)$ as defined in Section~\ref{subsubsect:phiframe}, the map $x\in U\mapsto \sigma^\bX(x,\cdot)\in S^m(\Hhat)$ is well defined and smooth. In particular, for any compact set $\mathcal{C}\subset U$ and any index $\beta\in\N^{2d+1}$, we have uniformly for $\hbar$ small,
    \begin{equation}
    \label{eq:seminormSm}
        \forall N\in\N,\quad \sup_{x\in\mathcal{C}}\|D_{\bX}^\beta \sigma_\hbar^\bX(x,\cdot)\|_{S^m(\Hhat),N} < +\infty,
    \end{equation}
    where we use the notation $D_{\bX}^\beta = X_1^{\beta_1}\cdots X_d^{\beta_d}Y_1^{\beta_{d+1}}\cdots Y_d^{\beta_{2d}}R^{\beta_{2d+1}}$.
\end{definition}

\begin{remark}
    One could use general adapted frames to define the symbol classes, but then one would need to introduce the notion of Rockland symbols, as in \cite{FFF24}.
\end{remark}

To know whether or not a symbol $\sigma$ on $\Ghat M$ is in the class $S^m(\Ghat M)$, it is enough to check the condition given by Definition~\ref{def:symclassGhatM} for an open cover of $M$ by local Heisenberg frames. 

For $m\in\R$, we endow the set $S^m(\Ghat M)$ with the seminorms given by Equation~\eqref{eq:seminormSm}. Using a finite cover of $M$ by compact sets included in local Heisenberg frames, the seminorms obtained this way are enough to characterize the topology of $S^m(\Ghat M)$, which turns it into a Fréchet space.

Note that we have the inclusions
\begin{equation*}
    m_1 \leq m_2 \implies S^{m_1}(\Ghat M)\subset S^{m_2}(\Ghat M).
\end{equation*}
This observation leads to the following definition.

\begin{definition}
\label{def:smoothsym}
    We define the class of \emph{smoothing symbols} $S^{-\infty}(\Ghat M)$ by
    \begin{equation*}
        S^{-\infty}(\Ghat M) := \bigcap_{m\in\R} S^m(\Ghat M),
    \end{equation*}
    equipped with the topology of the projective limit.
\end{definition}

However, the symbols we will deal with belong to the subspace of symbols that are \emph{admissible}, whose definition we now give.

\begin{definition}
    \label{def:symadmissible}
    For $m\in \R$, a symbol $\sigma = \{\sigma_\hbar\}_{\hbar>0}$ in $S^m(\Ghat M)$ is said to be admissible if, for any $N\in N$, there exists symbols $\sigma_i$, $i= 0,\dots,N-1$, in respectively $S^{m-i}(\Ghat M)$ and independent of $\hbar$, as well as a symbol $r^N = {r_\hbar^N}_{\hbar>0}$ in $S^{m-N}(\Ghat M)$, such that
    \begin{equation*}
        \sigma_\hbar = \sum_{i=0}^{N-1}\hbar^i \sigma_i + \hbar^Nr_\hbar^N.
    \end{equation*}
\end{definition}

Finally, the following proposition aims at giving a realization of these symbols in terms of the Fock Hilbert bundle $\mathcal{F}(N\Sigma)$. 

\begin{proposition}
    The restriction on $\Ghat_\infty M$ of any symbol $\sigma$ of order 0 defines a smooth section on $\Sigma$ of the Banach bundle of endomorphisms on the Fock Hilbert bundle, denoted by ${\rm End}(\mathscr{F}(N\Sigma))$.
\end{proposition}

Indeed, for a symbol $\sigma \in S^0(\Ghat M)$, considering its restriction on $\Ghat_\infty M\sim\Sigma$, it immediately follows from Definition~\ref{def:symbols} and from Equation~\eqref{eq:fockrealisation}, that $\sigma$ can be seen as a field of bounded operators acting on each fiber $\mathcal{F}(N_\sigma \Sigma)$, $\sigma\in \Sigma$, of the Fock Hilbert bundle. Moreover, simply by Definition~\ref{def:symclassGhatM} of the class $S^0(\Ghat M)$, as well as by the characterization of the invariant class of symbols $S^0(\Hhat)$ obtained by \cite[Theorem 6.5.1]{FR}, we observe that this field of operators is smooth with respect to the variable $\sigma \in \Sigma$.

\subsection{Semiclassical pseudodifferential calculus} 
\label{subsubsect:semiquantization}

We recall here the definition of the global semiclassical pseudodifferential calculus on the contact manifold $(M,\eta)$ described for filtered manifolds in \cite{BFF25}, as well as introduce the notion of microlocality with respect to the generic dual tangent bundle $\Ghat_\infty M$.

\subsubsection{Local semiclassical quantization}
\label{subsubsect:locsemiquantiz}
We first define the local quantization procedure of symbols for any given local adapted frame $(\bX,U)$, not necessarily a Heisenberg frame.

\begin{definition}
    \label{def:cutoff}
    A cutoff for $\exp^\bX$ is a smooth function $\chi\in C^\infty(M\times M)$ valued in $[0,1]$, properly supported in $\exp^\bX(U_\bX)\subset U\times U$, and equal identically to 1 in a neighborhood of $\{(x,x)\,:\,x\in U\}$.
\end{definition}

Given a cutoff function $\chi$, we will use the shorthand notation
\begin{equation*}
    \chi_x(y) :=\chi(x,y).
\end{equation*}

Now for a symbol $\sigma \in S^m(\Ghat M)$, denoting by $\kappa_{\sigma}$ the associated convolution kernel in $\cS'(GM)$, we define the operator $\Op_{\hbar}^{\bX,\chi}(\sigma)$ at any function $f\in C^\infty(M)$ by 
\begin{equation}
    \label{eq:localquantization}
    \forall x \in U,\quad \Op_{\hbar}^{\bX,\chi}(\sigma)f(x) := \int_{w\in G_x M} \kappa_{\sigma,x}^\hbar(w) (\chi_x f)\left(\exp_{x}^\bX(-\Ln_{x}^\bX(w))\right)d\mu_x(w),
\end{equation}
where the distribution $\kappa_{\sigma,x}^\hbar\in \cS'(G_x M)$ is defined by 
\begin{equation*}
    \kappa_{\sigma,x}^\hbar(w) := \hbar^{-Q} \kappa_{\sigma,x}(\delta_{\hbar^{-1}} w)
\end{equation*}
and the integral in Equation~\eqref{eq:localquantization} is to be understood in a distributional sense. 

\vspace{0.25cm}

If $(\bX,U)$ is a $\phi$-frame, we can also write this local quantization as
\begin{equation*}
    \Op_{\hbar}^{\bX,\chi}(\sigma)f(x) = \int_{v \in \R^n}  \kappa_{\sigma,x}^{\bX,\hbar} (v) \ (\chi_x f)(\exp^\bX_x (- v)) \, dv,
\end{equation*}
where $\kappa_{\sigma,x}^{\bX,\hbar}(v) := \hbar^{-Q} \kappa_{\sigma,x}^\bX(\delta_{\hbar^{-1}} v)$, by the definition of the smooth Haar system $\mu$.
Performing the change of variable $y=\exp^\bX_x(-v)$ in Equation~\eqref{eq:localquantization}, or more accurately considering the pullback of $dv$ via $v\mapsto \exp^\bX_x (- v)$, 
the quantization formula is also given by
\begin{equation}\label{eq_int_kernel}
	\Op^{\bX,  \chi}_\hbar(\sigma)f(x)
=\hbar^{-Q}
\int_{M}  \kappa_{\sigma,x}^{\bX} (- \delta_\hbar^{-1} \ln^\bX_x y) \ (\chi_x f)(y)\ \jac_{y} (\ln^\bX_x) \ d{\rm vol}(y),
\end{equation}
where $\jac_{y} (\ln^\bX_x)$ denotes the smooth Radon-Nykodym derivative $\frac{d \ln^\bX_x}{dy} $  of $(\exp^\bX_x (- v))^*dv$ against the volume form ${\rm vol}$.
The latter formula provides us with the integral kernel of $\Op_\hbar^{\bX,\chi}(\sigma)$, given by
\begin{equation}
    \label{eq:integral_kernel}
    (x,y)\mapsto  \hbar^{-Q} \kappa_{\sigma,x}^{\bX} (- \hbar^{-1}  \ln^\bX_x (y)) \ \chi_x (y)\ \jac_{y} (\ln^\bX_x)=:K_\hbar^\bX(x,y).
\end{equation}

\vspace{0.25cm}

If now $(\tilde{\bX},U)$ is a Darboux frame, the corresponding local semiclassical quantization is related to the Kohn-Nirenberg semiclassical quantization $\Op_{\Heis^d,\hbar}$ on the Heisenberg group $\Heis^d$ (as defined in \cite{FL, FKF19, FM, FR}) in the following way. By choosing an origin point $x_0$ in $U$, the exponential map
\begin{equation*}
    \gamma :=\exp_{x_0}^{\tilde{\bX}}: V\subset \Heis^d\rightarrow U, 
\end{equation*}
defines a strict contact diffeomorphism (up to taking $U$ smaller), where we identify the Heisenberg group $\Heis^d$ with its Lie algebra and $V$ is an open neighborhood of $0$. Observe that since $\tilde{\bX}$ is a Darboux frame, Theorem~\ref{thm:composgeomexp} is valid without any remainder term. In particular, for $v\in \Heis^d\sim\R^{2d+1}$, we can write
\begin{equation*}
    \exp_x^{\tilde{\bX}}(v)= \gamma(v*_{\Heis^d}\gamma^{-1}(x)).
\end{equation*}

In the trivialization $(x,v)\in U\times\Heis^d$ of the tangent group bundle via the local frame $\tilde{\bX}$, the smooth Haar system $\{\mu_x\}_{x\in M}$, as defined in Section~\ref{subsubsect:smoothhaarsys}, writes
\begin{equation*}
    \forall x\in U,\quad d\mu_x(v) = c(x)dv,
\end{equation*}
with $c$ a smooth positive function. Moreover, recall that we set $\kappa_{\sigma,x} = \cF_{G_x M, \mu_x}^{-1}(\sigma(x,\cdot))$ for $x\in U$. Thus, it follows that using the smooth Haar system $\{\tilde{\mu}_x\}_{x\in U}$ induced by the Darboux frame $\tilde{\bX}$, we obtain
\begin{equation}
    \label{eq:symDarboux}
    \forall v\in \R^{2d+1},\quad\kappa_{\sigma,x}^{\tilde{\bX}}(v) = c(x)^{-1}\mathcal{F}_{G_x M, \tilde{\mu}_x}^{-1}(\sigma(x,\cdot))({\rm Exp}_x^{\tilde{\bX}}(v)).
\end{equation}

With this in mind, Equation~\eqref{eq:localquantization} pulled back on $V\subset \Heis^d$ by $\gamma$ has the following expression: for any $f\in C^\infty(V)$,
\begin{equation*}
    \begin{aligned}
    \forall x\in V,\quad\left(\gamma^{-1}\circ\Op_{\hbar}^{\bX,\chi}(\sigma)\circ \gamma \right)f(x) &= \int_{\Heis^d} \kappa_{\sigma,x}^{\tilde{\bX},\hbar}(v)(\chi_x f)(v^{-1}*_{\Heis^d}x)c(x)dv\\
    &= \int_{\Hhat^d}{\rm Tr}_{\mathcal{H}_\pi}\left(\sigma^{\tilde{\bX}}(x,\hbar\cdot\pi)\mathcal{F}_{\Heis^d}(\chi_x f)(\pi)\right)d\hat{\mu}_{\Hhat^d}(\pi),
    \end{aligned}
\end{equation*}
where the last line follows from the Plancherel formula and from Equation~\eqref{eq:symDarboux}, and where $\sigma^{\tilde{\bX}}$ is defined by Equation \eqref{eq:symframe}. The term on the right-hand side of the last equation is then the usual semiclassical quantization $\Op_{\Heis^d,\hbar}$ on the Heisenberg group, as defined, for example, in \cite{FKF19}, meaning that one can sum up this discussion by writing
\begin{equation}
\label{eq:equivquantizations}
    \gamma^{-1}\circ\Op_{\hbar}^{\bX,\chi}(\sigma)\circ \gamma = \Op_{\Heis^d,\hbar}(\sigma^{\tilde{\bX}}).
\end{equation}

\subsubsection{Global semiclassical pseudodifferential calculus}
This local quantization allows for the following global definition of a semiclassical pseudodifferential operator.

\begin{definition}
    \label{def:pseudodiffM}
    Let $m\in\R\cup\{-\infty\}$. The space of \emph{semiclassical pseudodifferential operators} $\Psi_{\hbar}^m(M)$ of order $m$ is defined as the space of all $\hbar$-dependent linear operators 
    \begin{equation*}
        P = \{P_\hbar\}_{\hbar>0}: C^\infty(M)\rightarrow C^\infty(M),
    \end{equation*}
    satisfying the following properties:
    \begin{itemize}
        \item[(i)] if $\psi_1,\psi_2\in C^\infty(M)$ are two maps such that $\supp(\psi_1)\cap\supp(\psi_2) = \emptyset$, then the operator $\psi_1 P \psi_2$ is \emph{semiclassically smoothing on the Sobolev scale}, i.e.~it satisfies for all $N\in \N$,
        \begin{equation*}
            \|\psi_1 P_\hbar \psi_2 \|_{H^{-N}(M)\rightarrow H^{N}(M)} = \mathscr{O}(\hbar^\infty),
        \end{equation*}
        where the spaces $H^m(M)$, $m\in \R$, are adapted Sobolev spaces (see \cite[Section 9.5]{FFF24});
        \item[(ii)] if $(\bX,U)$ is a local Heisenberg frame with cutoff $\chi$, and if $\psi\in C_{c}^\infty(U)$, then there exists $\sigma \in S^m(\Ghat M)$ with support in $\Ghat_{|U}M$ such that 
        \begin{equation*}
            \psi P_\hbar(\psi f) = \Op_{\hbar}^{\bX,\chi}(\sigma)f,\quad f\in C^\infty(M).
        \end{equation*}
    \end{itemize}
\end{definition}

If in the second point of this definition, the symbol $\sigma$ appearing is admissible, the corresponding semiclassical pseudodifferential operator $P$ is also said to be \emph{admissible}. This makes sense since the formula of change of frames preserves the class of admissible symbols.

\vspace{0.25cm}

Introducing the notation
\begin{equation*}
    \Psi_\hbar^\infty(M) := \bigcup_{m\in\R\cup\{-\infty\}} \Psi_\hbar^m(M),
\end{equation*}
it results from \cite{FFF24} that the \emph{semiclassical pseudodifferential calculus} $\Psi_\hbar^\infty(M)$ satisfies the following properties:
\begin{enumerate}
    \item[(i)] For each $m\in\R\cup\{-\infty\}$, the spaces $\Psi_\hbar^m(M)$ have a structure of Fréchet spaces coming from the one of the class of symbols $S^m(\Ghat M)$ and from Definition~\ref{def:pseudodiffM}; moreover, the inclusions $\Psi_\hbar^m(M)\subset \Psi_\hbar^{m'}(M)$, for $m\leq m'$, is then continuous;
    \item[(ii)] The inclusions of $\Psi_\hbar^m(M)$ in the spaces of endomorphisms on $C^\infty(M)$ and $\mathcal{D}'(M)$ are both continuous;
    \item[(iii)] If $P_1 \in \Psi_\hbar^{m_1}(M)$ and $P_2 \in \Psi_\hbar^{m_2}(M)$, then $P_1\circ P_2$ is in $\Psi_\hbar^{m_1+m_2}(M)$; similarly, the adjoint $P_1^*$ is in $\Psi_\hbar^{m_1}(M)$;
    \item[(iv)] Operators in $\Psi_\hbar^m(M)$ act naturally on the scale of adapted Sobolev spaces, i.e.~they belong to the spaces of operators $\mathcal{L}(H^s(M),H^{s-m}(M))$ for any $s\in\R$.
\end{enumerate}

\subsubsection{Global quantization procedure}
We define now the notion of \emph{quantization procedure}.

\begin{definition}
    \label{def:atlas}
    A finite collection $\mathcal{A} = (\bX_\alpha,U_\alpha,\chi_\alpha,\psi_\alpha)_{\alpha\in A}$ is called an \emph{atlas of local Heisenberg frames} (or simply an atlas) if it consists in
    \begin{itemize}
        \item[(i)] a collection $(\bX_\alpha,U_\alpha)_{\alpha\in A}$ of local Heisenberg frames which covers $M$;
        \item[(ii)] a family of cutoffs $(\chi_\alpha)_{\alpha\in A}$ corresponding to the previous local frames;
        \item[(iii)] a family of smooth functions $(\psi_\alpha)_{\alpha\in A}$ on $M$ such that $\sum_{\alpha}\psi_{\alpha}^2 = 1$ and $(x,y)\in M\times M\mapsto \psi_\alpha(x)\psi_\alpha(y)$ is supported in $\{\chi_\alpha(x,y)=1\}$.
    \end{itemize}
\end{definition}

\begin{remark}
    The global quantization procedures we are going to define would still make sense for atlases of general local adapted frames. We restrict our attention however to local Heisenberg frames since they are better suited to contact manifolds and since we used such frames to define our symbol classes.
\end{remark}

Once we have given an atlas, we can define a global quantization procedure.

\begin{definition}
    \label{def:globalquantization}
    Let $\mathcal{A} = (\bX_\alpha,U_\alpha,\chi_\alpha,\psi_\alpha)_{\alpha\in A}$ be an atlas. Then, for $\sigma \in S^m(\Ghat M)$ for $m\in\R\cup\{-\infty\}$ and $f\in C^\infty(M)$, we define 
    \begin{equation*}
        \forall x\in M,\quad \Op_{\hbar}^\mathcal{A}(\sigma)f(x) := \sum_{\alpha\in A}\Op_{\hbar}^{\bX_\alpha,\chi_\alpha}(\psi_\alpha \sigma)(\psi_\alpha f)(x).
    \end{equation*}
\end{definition}

Such a procedure, even if it is not canonical, allows for a nice description of the semiclassical pseudodifferential operators.

\begin{proposition}
    \label{prop:globalquantization}
    Let $\mathcal{A}$ be an atlas and $P = \{P_\hbar\}_{\hbar>0}:C^\infty(M)\rightarrow C^\infty(M)$ a linear operator. Then, $P$ is in $\Psi_{\hbar}^m$ for some $m\in\R\cup\{-\infty\}$ if, and only if there exists a symbol $\sigma = \{\sigma_\hbar\}_{\hbar>0}\in S^m(\Ghat M)$ and a remainder $R_\hbar: C^\infty(M)\rightarrow C^\infty(M)$ such that
    \begin{equation*}
        A_\hbar = \Op_{\hbar}^\mathcal{A}(\sigma) + R_\hbar,
    \end{equation*}
    and the operator $R_\hbar$ is semiclassically smoothing on the Sobolev scale.
\end{proposition}

In the following sections, we suppose to have given such a global quantization procedure, simply denoted by $\Op_\hbar$.

\vspace{0.25cm}

The existence of global quantization procedures leads to the following definition of \emph{principal symbol} associated with a semiclassical pseudodifferential operator.

\begin{proposition}
    \label{prop:principalsym}
    Let $m\in\R\cup\{-\infty\}$. There exists a unique surjective map
    \begin{equation*}
        \princ:\Psi_\hbar^m(M)\rightarrow S^m(\Ghat M)/\hbar S^{m-1}(\Ghat M),
    \end{equation*}
    such that, for any quantization procedure $\Op_\hbar$ as defined by Definition~\ref{def:globalquantization}, and for any $\sigma\in S^m(\Ghat M)$, we have
    \begin{equation*}
        \princ\left(\Op_\hbar(\sigma)\right) = [\sigma] \in S^m(\Ghat M)/\hbar S^{m-1}(\Ghat M).
    \end{equation*}
    Moreover, we have the following symbolic properties: for any two semiclassical pseudodifferential operators $P_1 \in \Psi_\hbar^{m_1}(M)$ and $P_2 \in \Psi_\hbar^{m_2}(M)$, we have
    \begin{equation*}
        \princ(P_1\circ P_2) = \princ(P_1)\circ \princ(P_2)\quad\mbox{and}\quad\princ(P_1^*) = \princ(P_1)^*.
    \end{equation*}
\end{proposition}

For differential operators, the notion of principal symbols coincides with the one of principal part, as defined in \cite[Section 3.4]{FFF24}, and is then easily computable given a local adapted frame $(\bX,U)$.

\begin{remark}
    If the semiclassical pseudodifferential operator $P = \{P_\hbar\}_{\hbar>0} \in \Psi_\hbar^m(M)$ is admissible, its principal symbol $\princ (P)$ can then be identified with a true $\hbar$-independent symbol in $S^m(\Ghat M)$, i.e.~not only as an equivalence class.
\end{remark}

\subsubsection{Semiclassical functional calculus}

We now investigate the functional calculus of the subLaplacian $-\hbar^2 \Delta_{sR}$ as part of the semiclassical pseudodifferential calculus $\Psi_\hbar^\infty(M)$. 

\begin{theorem}
    \label{thm:functionalcalc}
    Let $f\in C_c^\infty(\R)$. For any $N\in \N$, there exist $Q_{f,N}(\hbar)\in \Psi_\hbar^{-\infty}(M)$ and $R_{f,N}(\hbar)$ a bounded operator on $L^2(M)$ such that
    \begin{equation*}
        f(-\hbar^2\Delta_{sR}) = Q_{f,N}(\hbar) + R_{f,N}(\hbar),\quad\mbox{with}\ R_{f,N}(\hbar) = \O_{L^2(M)\rightarrow L^2(M)}(\hbar^{N+1}).
    \end{equation*}
    Moreover, the principal symbol of $Q_{f,N}(\hbar)$ is given by
    \begin{equation*}
        \princ(Q_{f,N}(\hbar)) = f(H),
    \end{equation*}
    where $H$ denotes the principal symbol $\princ(-\hbar^2\Delta_{sR})$ of the subLaplacian.
\end{theorem}

The proof of this theorem is postponed to Appendix~\ref{sect:functionalcalc}.

\subsubsection{Generic operator wave front set and microlocality}
\label{subsubsect:wavefrontset}

In this section, we define the notion of microlocal equality on the generic dual tangent bundle $\Ghat_\infty M$.

\begin{definition}
    \label{def:esssup}
    Let $\sigma = \{\sigma_\hbar\}_{\hbar>0}$ a semiclassical symbol in $S^{-\infty}(\Ghat M)$. We say that $\sigma$ has a \emph{generic essential support} if there exists a compact set $K\subset \Ghat_\infty M$ such that
    \begin{equation*}
        {\rm supp}(\chi)\cap K = \emptyset\quad\mbox{implies}\quad\chi\sigma_\hbar\in\hbar^\infty S^{-\infty}(\Ghat M),
    \end{equation*}
    for any $\chi\in C_c^\infty(\Ghat_\infty M)$. The smallest such compact set $K$ is called the generic essential support and is denoted by
    \begin{equation*}
        {\rm ess\,supp}_\infty(\sigma)\subset \Ghat_\infty M.
    \end{equation*}
\end{definition}

Let $P=\{P_\hbar\}_{\hbar>0}\in \Psi_\hbar^{-\infty}(M)$ a semiclassical pseudodifferential operator. We say that $P$ admits a \emph{generic operator wave front set} when, following Proposition~\ref{prop:globalquantization}, this operator can be written
\begin{equation*}
    P_\hbar = \Op_\hbar(\sigma_\hbar) + R_\hbar,
\end{equation*}
with $\sigma = \{\sigma_\hbar\}_{\hbar>0}\in S^{-\infty}(\Ghat M)$ having generic essential support. We then introduce the notation
\begin{equation}
    \label{eq:opwavefrontset}
    {\rm WF}'_\infty(P) := {\rm ess\,supp}(\sigma).
\end{equation}
Observe that this definition does not depend on the choice of global quantization $\Op_\hbar$ by the formula of change of frames given by \cite[Proposition 2.6]{BFF25} and \cite[Proposition A.5]{BFF25}, since difference operators are local operators on $\Ghat_\infty M\sim \Sigma$.

Observe moreover that the properties of the symbolic calculus given by \cite[Proposition 2.8-2.9]{BFF25} have the following consequences: if $P_1, P_2\in \Psi_\hbar^{-\infty}(M)$ have both generic operator wave front set, then we have
\begin{equation*}
    {\rm WF}'_\infty(P_1+P_2)\subset{\rm WF}'_\infty(P_1)\cup {\rm WF}'_\infty(P_2),\quad {\rm WF}'_\infty(P_1\circ P_2) \subset {\rm WF}'_\infty(P_1)\cap {\rm WF}'_\infty(P_2),
\end{equation*}
and 
\begin{equation*}
    {\rm WF}'_\infty(P_1^*) = {\rm WF}'_\infty(P_1).
\end{equation*}
This notion leads to the following definition of microlocality on $\Ghat_\infty M$.

\begin{definition}
    \label{def:microlocality}
    Let $K\subset \Ghat_\infty M$ a closed subset. For any two semiclassical pseudodifferential operators $P_1$ and $P_2$ in $\Psi_\hbar^{-\infty}(M)$, we say that these operators are microlocally equal on $K$ if we have
    \begin{equation*}
        K\cap {\rm WF}'_\infty(P_1-P_2) = \emptyset.
    \end{equation*}
\end{definition}

\begin{remark}
    From these definitions, one would also be able to define a generic wave front set for distributions depending on $\hbar$. However, we will have no use for this notion and we will not develop on it.
\end{remark}

\subsection{Landau projectors}
\label{subsect:Landauproj}

In this section, we suppose given $(M^{2d+1},\eta,g)$ a compact contact metric structure as in Section~\ref{subsect:sublaplacian} and we consider the associated semiclassical subLaplacian $-\hbar^2\Delta_{\rm sR}$, as defined by Definition~\ref{def:subLaplacian}.

\subsubsection{Principal symbol of the subLaplacian}

Recall that, given a local $\phi$-frame 
$$\bX = (X_1,\dots,X_d,Y_1,\dots,Y_d,R)$$ 
on an open subset $U$, the semiclassical subLaplacian writes 
\begin{equation*}
    -\hbar^2\Delta_{\rm sR} = \sum_{i=1}^d \left(\hbar^2 X_{i}^* X_i + \hbar^2 Y_{i}^*Y_i\right),
\end{equation*}
where the adjoint is taken with respect to the volume form $\nu$. Since the vector fields $(\hbar X_i,\hbar Y_i)_{i=1,\dots,d}$ are tangent to the contact distribution $\cD$, they corresponds to differential operators in $\Psi_\hbar^1(M)$. We deduce that the semiclassical subLaplacian is an admissible differential operator contained in $\Psi_\hbar^2(M)$ whose principal symbol is given over $\Ghat M_{|U}$ by the composition rule: 
\begin{equation*}
    \princ^2\left(-\hbar^2 \Delta_{\rm sR}\right)(x,\pi) = -\sum_{i=1}^d \left(\pi\left(\langle X_i\rangle_x\right)^2+ \pi\left(\langle Y_i\rangle_x\right)^2\right).
\end{equation*}
Recall that we simply denote this principal symbol by $H$. In the trivialization $\hat{\varphi}_U^\bX$ of the generic dual tangent bundle $\Ghat_\infty M$, the symbol $H$ is explicit: for $(x,[\pi_x^\lambda])\in \Ghat_\infty M$, we have
\begin{equation}
    \label{eq:symH}
    H(x,[\pi_x^\lambda]) = |\lambda| \sum_{i=1}^d \left(-\partial_{\xi_i}^2 + \xi_i^2\right).
\end{equation}

In order to lighten the notations, from now on we will simply denote by $(x,\lambda)$ the point $(x,[\pi_x^\lambda])$ of the generic dual tangent bundle given by the trivialization $\hat{\varphi}_U^\bX$ and write $H(\lambda)$ for the principal symbol of the subLaplacian seeing that it does not depend on the variable $x$ in this trivialization.

Observe now that this expression of $H$ is simply a sum of independent harmonic oscillators in the variable $\xi \in \langle \mathfrak{p}\rangle_x$, which we identify with $\R^d$. 

\bigskip

We now recall the spectral properties of the harmonic oscillator. The eigenfunctions of $H(\lambda)$ can be expressed in terms of the basis of Hermite functions $(e_n)_{n\in\N}$, normalized
in $L^2(\R)$ and satisfying for all $\xi\in\R$ the equation
\begin{equation*}
    -e_{n}''(\xi) + \xi^2 e_n(\xi)  = (2n+1)e_n(\xi).
\end{equation*}
Indeed, for each multi-index $\alpha\in\N^d$, the function $h_\alpha$ defined by
\begin{equation*}
    e_\alpha(\xi) := \prod_{i=1}^d e_{\alpha_i}(\xi_i),\quad \xi = (\xi_1,\dots,\xi_d)\in\R^d,
\end{equation*}
is an eigenfunction of $H(\lambda)$ for the eigenvalue $|\lambda|(2|\alpha|+d)$, that is,
\begin{equation*}
    H(\lambda)e_\alpha = |\lambda|(2|\alpha|+d)e_\alpha,\quad \alpha\in\N^d.
\end{equation*}
As the family $(e_\alpha)_{\alpha\in\N^d}$ is a Hilbert basis in $L^2(\R^d)$, these eigenvalues completely describe the spectrum of $H(\lambda)$. Thus, the entire spectrum of $H(\lambda)$ is given by the set $\{|\lambda|(2n+d)\}_{n\in\N}$ and from the description of the eigenfunctions $(e_\alpha)_{\alpha\in\N^d}$, the multiplicity of the $n$-th eigenvalue is given by
\begin{equation}
    \label{eq:multiplicity}
    {\rm mult}(|\lambda|(2n+d)) = \binom{n+d-1}{n}.
\end{equation}

\begin{remark}
    \label{rem:oddeven}
    We recall here, as a useful future reminder, that eigenfunctions of the harmonic oscillator are necessarily odd or even, depending on the parity of $n$, the integer parameter of the Landau levels.
\end{remark}

\subsubsection{Homogeneous symbol class of the Landau projectors}

For $n\in \N$, we introduce the notation $\Pi_n$, called the \emph{$n$-th Landau projector}, for the orthogonal projector onto the space of eigenfunctions $e_\alpha$, $\alpha\in\N^d$, satisfying $|\alpha| = n$. This defines a measurable symbol on $\Ghat_U M$ through the trivialization given by $\bX$. Equivalently, $\Pi_n(x,[\pi_x^\lambda])$ can be described as the spectral projector of $H(x,[\pi_x^\lambda])$ associated with the eigenvalue $|\lambda|(2n+d)$ for any $(x,[\pi_x^\lambda]) \in \Ghat_\infty M$. This definition allows us to see that $\Pi_n$ is actually globally defined as a measurable symbol on $\Ghat M$, called the \emph{$n$-th Landau symbol}. 

In reality, the Landau projectors are smooth homogeneous symbols, as defined in \cite[Section 5.5.1]{FFF24}.

\begin{proposition}
    \label{prop:Pinsmoothsym}
    For $n\in\N$, the $n$-th Landau symbol $\Pi_n$ is a homogenous symbol in $\dot{S}^0(\Ghat M)$.
\end{proposition}

Consequently, for any cut-off function $\rho\in C_c^\infty(\Ghat_\infty M,[0,1])$, the symbol $\Pi_{n}^\rho$ defined by 
\begin{equation}
    \label{eq:Pinrho}
    \forall (x,[\pi_x^\lambda])\in \Ghat_\infty M,\quad \Pi_{n}^\rho(x,[\pi_x^\lambda]) := \rho(x,\lambda)\Pi_n(x,[\pi_x^\lambda]),
\end{equation}
is in the class $S^{-\infty}(\Ghat M)$, since the function $\rho$ defines a symbol in $S^{-\infty}(\Ghat M)$ by \cite[Theorem 6.5.1]{FR}. 

Since the symbols $H$ and $\Pi_n$ are independent of variable $x\in M$ in any local trivialization induced by a $\phi$-frame, the proof of Proposition~\ref{prop:Pinsmoothsym} follows directly from that of \cite[Proposition 3.6]{FKF21}.

\subsubsection{Construction of the Landau projectors} 

Let $\rho\in C_c^\infty(\Ghat_\infty M,[0,1])$ again denote a compactly supported cutoff function. Since the symbol $\Pi_n^\rho$ defined by Equation~\eqref{eq:Pinrho} is in $S^{-\infty}(\Ghat M)$, one can consider its quantization $\Op_\hbar(\Pi_n^\rho)$. In particular, we are interested in its commutator with the subLaplacian: by the composition rule of semiclassical pseudodifferential operators, the principal symbol
of the operator
\begin{equation*}
    [-\hbar^2 \Delta_{\rm sR},\Op_\hbar(\Pi_n^\rho)],
\end{equation*}
is given by $[H,\rho \Pi_n] = 0$, and thus we have
\begin{equation*}
    [-\hbar^2 \Delta_{\rm sR},\Op_\hbar(\Pi_n^\rho)] \in \hbar \Psi_\hbar^{-\infty}(M).
    \end{equation*} 
Similarly, we can write
\begin{equation*}
    \Op_\hbar(\Pi_n^\rho)\circ \Op_\hbar(\Pi_n^\rho)=\Op_\hbar(\Pi_n^{\rho^2}) + \hbar\Psi_\hbar^{-\infty}(M).
\end{equation*}

The last two equations can be interpreted as follows: the operator $\Op_\hbar(\Pi_n^\rho)$ is a projector microlocally on $\{\rho = 1\}\subset\Ghat_\infty M$, up to a term of order $\O(\hbar)$; moreover, this operator also commutes up to $\O(\hbar)$ with the subLaplacian $-\hbar^2\Delta_{\rm sR}$.

The next result shows that, up to correcting the symbol $\Pi_n^\rho$ with lower order terms, the previous remainder in $\O(\hbar)$ can be improved to any desired order in $\hbar$.

\begin{theorem}
    \label{thm:superadiabproj}
    Let $n\in \N$. There exists a unique admissible self-adjoint semiclassical pseudodifferential operator $\hat{\Pi}_n^{\rho,\hbar}$ in $\Psi_\hbar^{-\infty}(M)$, called the \emph{$n$-th Landau projector}, with principal symbol given by $\Pi_n^\rho$ such that microlocally on $\{\rho=1\}$,
    \begin{equation*}
        [-\hbar^2 \Delta_{\rm sR},\hat{\Pi}_n^{\rho,\hbar}] = 0
    \quad\mbox{and}\quad
        \hat{\Pi}_n^{\rho,\hbar} = (\hat{\Pi}_n^{\rho,\hbar})^* = \hat{\Pi}_n^{\rho,\hbar}\circ \hat{\Pi}_n^{\rho,\hbar}.
    \end{equation*}
\end{theorem}

\begin{proof}
    The construction of the projector $\hat{\Pi}_n^{\rho,\hbar}$ is performed recursively, following the proof of \cite[Theorem 1.2]{BFKRV}. We are looking for semiclassical pseudodifferential operators $(\hat{\Pi}_n^{\rho,k})_{k\in\N}$ such that 
    \begin{equation*}
        \hat{\Pi}_n^{\rho,0} = \Op_\hbar(\Pi_n^{\rho})\quad \mbox{and}\quad \hat{\Pi}_n^{\rho,k} \in \Psi_\hbar^{-\infty}(M),\ k\geq 1,
    \end{equation*}
    and for which the following holds true:
    \begin{equation}
    \begin{aligned}\label{eq:recursiveeq}
	&\hat{\Pi}_n^{\rho,[k]}\circ \hat{\Pi}_n^{\rho,[k]} -\hat{\Pi}_n^{\rho,[k]} =\hbar^{k+1}R_{k+1}\quad\mbox{microlocally on }\ \{\rho=1\},\\
	&[-\hbar^2 \Delta_{\rm sR},\hat{\Pi}_n^{\rho,[k]}]=\hbar^{k+1}T_{k+1},
    \end{aligned}
    \end{equation}
    for some semiclassical pseudodifferential operators $R_{k+1}$ and $T_{k+1}$ belonging to $\Psi_\hbar^{-\infty}(M)$, and where we introduced the notation
    \begin{equation*}
        \hat{\Pi}_n^{\rho,[k]} = \hat{\Pi}_n^{\rho,0} + \hbar \hat{\Pi}_n^{\rho,1}+\cdots+ \hbar^k \hat{\Pi}_n^{\rho,k}.
    \end{equation*}
    The existence of the desired operator $\hat{\Pi}_n^{\rho,\hbar}$ will follow from a standard Borel summation procedure.

    To begin with, observe that the relations~\eqref{eq:recursiveeq} hold for $k=0$ as discussed in the beginning of the section, simply by composition of principal symbols. Now, assume that we have constructed $\hat{\Pi}_n^{\rho,0},\hat{\Pi}_n^{\rho,1},\dots, \hat{\Pi}_n^{\rho,k}$ for some $k\geq 0$ such that the relations~\eqref{eq:recursiveeq} hold.
    Introducing the notation $\Pi_n^{\rho,k+1}\in S^{-\infty}(\Ghat M)$, we are looking for a symbol that satisfies
    \begin{equation}
        \label{eq:recursivePik+1}
          \Pi_n^{\rho}\,\Pi_n^{\rho,k+1} - \Pi_n^{\rho,k+1} \Pi_n^{\rho,\perp}+ R_{k+1,0} = 0\quad \mbox{and}\quad [H,\Pi_n^{\rho,k+1}]+ T_{k+1,0} = 0,
    \end{equation}
    where we denoted $\princ\left(R_{k+1}\right)$ and $\princ(T_{k+1})$ respectively by $R_{k+1,0}$ and $T_{k+1,0}$, and where we set $\Pi_n^{\rho,\perp} = {\rm Id}-\Pi_n^\rho$.

    The first relation in~\eqref{eq:recursivePik+1} defines the diagonal part of $\Pi_n^{\rho,k+1}$ on $\{\rho=1\}$ in the following sense:
    \begin{equation}
        \label{eq:Pi_k+1_diag}
        \Pi_n^{\rho}\,\Pi_n^{\rho,k+1}\Pi_n^{\rho} = -\Pi_n^{\rho}R_{k+1,0}\Pi_n^{\rho}\quad\mbox{and}\quad \Pi_n^{\rho,\perp}\Pi_n^{\rho,k+1}\Pi_n^{\rho,\perp} = \Pi_n^{\rho,\perp}R_{k+1,0}\Pi_n^{\rho,\perp},
    \end{equation}
    and it is enough to satisfy the first relation in~\eqref{eq:recursivePik+1} if and only if one can check the following first compatibility relation 
    \begin{equation}\label{eq:comp_1}
        [\Pi_n^{\rho},R_{k+1,0}]=0.
    \end{equation}
    The second relation in~\eqref{eq:recursivePik+1} determines the anti-diagonal part of $\Pi_n^{\rho,k+1}$: indeed, it writes 
    \begin{equation}\label{eq:rec_k}
        [H,\Pi_n^{\rho} \Pi_n^{\rho,k+1}\Pi_n^{\rho,\perp}]=-\Pi_n^{\rho} 
        T_{k+1,0}\Pi_n^{\rho,\perp}\quad\mbox{and}\quad [H,\Pi_n^{\rho,\perp} \Pi_n^{\rho,k+1}\Pi_n^{\rho}]=-\Pi_n^{\rho,\perp} T_{k+1,0}\Pi_n^{\rho}.
    \end{equation}
    This system can be solved if and only if 
    \begin{equation}\label{eq:comp_2}
        \Pi_n^{\rho}[H,R_{k+1,0}]\Pi_n^{\rho}=\Pi_n^{\rho} T_{k+1,0}\Pi_n^{\rho}\quad\mbox{and}\quad
	   \Pi_n^{\rho,\perp}[H,R_{k+1,0}]\Pi_n^{\rho,\perp}=-\Pi_n^{\rho,\perp} T_{k+1,0}\Pi_n^{\rho,\perp}\,.
	\end{equation}
    Thus, Equations~\eqref{eq:comp_1} and~\eqref{eq:comp_2} are compatibility relations. 

    Postponing the verification of these compatibility conditions, we can reformulate the first equation in Equation~\eqref{eq:rec_k} as follows: let $Y = -\Pi_n^{\rho} 
    T_{k+1,0}\Pi_n^{\rho,\perp}$, we are looking for $X = \Pi_n^{\rho} \Pi_{k+1}\Pi_n^{\rho,\perp}$ such that
    \begin{equation*}
        (H\Pi_n^{\rho}) X - X (H \Pi_n^{\rho,\perp}) = Y.
    \end{equation*}
    This problem is solved by the Sylvester’s Theorem since the spectral gap between eigenvalues of $H$ is uniformly bounded by below on the set $\{\rho = 1\}$. Moreover, following \cite[Theorem A.1]{BFKRV}, we have the following formula for the solution $X$: locally around any point $(x,\lambda)\in \Ghat_\infty M$, we can write
    \begin{equation*}
        X = -\frac{1}{2\pi i} \int_\gamma (H\Pi_n^\rho-z)^{-1} Y(H\Pi_n^{\rho,\perp}-z)^{-1} \,dz,
    \end{equation*}
    where $\gamma$ is a closed contour in the plane with winding number one around the spectrum of $\Pi_n^\rho$ and zero around the one of $\Pi_n^{\rho,\perp}$. Such a closed curve $\gamma$ exists thanks to the spectral gap between eigenvalues of the symbol $H$ in a neighborhood of $(x,\lambda)\in \supp(\rho)$. This allows us to conclude that $X$ is a symbol in $S^{-\infty}(\Ghat M)$ and thus that Equations~\eqref{eq:Pi_k+1_diag} and~\eqref{eq:rec_k} define uniquely the symbol $\Pi_n^{\rho,k+1}$, and thus a semiclassical pseudodifferential operator by setting
    \begin{equation*}
        \hat{\Pi}_n^{\rho,k} = \Op_\hbar(\Pi_n^{\rho,k})\in\Psi_\hbar^{-\infty}(M).
    \end{equation*}

    It remains to verify that the compatibility relations are satisfied. Let us start with Equation~\eqref{eq:comp_1}. We deduce from the first relation of Equation~\eqref{eq:recursiveeq} that 
    \begin{equation*}
        \begin{aligned}
            \hat{\Pi}_n^{\rho,[k]} R_{k+1} (1-\hat{\Pi}_n^{\rho,[k]}) 
            &= \hbar^{-k-1} \hat{\Pi}_n^{\rho,[k]}( \hat{\Pi}_n^{\rho,[k]}\circ\hat{\Pi}_n^{\rho,[k]} - \hat{\Pi}_n^{\rho,[k]})(1-\hat{\Pi}_n^{\rho,[k]})\\
            &= -\hbar^{-k-1} (\hat{\Pi}_n^{\rho,[k]}\circ \hat{\Pi}_n^{\rho,[k]}-\hat{\Pi}_n^{\rho,[k]})^2 = -\hbar^{k+1} R_{k+1}^2,
        \end{aligned}
    \end{equation*}
    whence $\Pi_n^\rho R_{k+1,0}(1-\Pi_n^\rho) = 0$. One argues similarly with $(1-\Pi_n^\rho)R_{k+1,0}\Pi_n^\rho$.
    To prove Equation~\eqref{eq:comp_2}, we use Equation~\eqref{eq:recursiveeq} and write
    \begin{equation*}
        \begin{aligned}
            \hat{\Pi}_n^{\rho,[k]} T_{k+1} \hat{\Pi}_n^{\rho,[k]}
            &= \hbar^{-k-1}\hat{\Pi}_n^{\rho,[k]} [-\hbar^2 \Delta_{sR}, \hat{\Pi}_n^{\rho,[k]}]\hat{\Pi}_n^{\rho,[k]}\\
            &= \hbar^{-k-1} \hat{\Pi}_n^{\rho,[k]}\left(\hat{\Pi}_n^{\rho,[k]}(-\hbar \Delta_{sR})(\hat{\Pi}_n^{\rho,[k]})^2 - (\hat{\Pi}_n^{\rho,[k]})^2(-\hbar^2\Delta_{sR})\hat{\Pi}_n^{\rho,[k]}\right)\hat{\Pi}_n^{\rho,[k]}\\
            &= \hat{\Pi}_n^{\rho,[k]} (-\hbar^2\Delta_{sR})R_{k+1} - R_{k+1}(-\hbar^2\Delta_{sR})\hat{\Pi}_n^{\rho,[k]},
        \end{aligned}
    \end{equation*}
    whence $\Pi_n^\rho H R_{k+1,0} - R_{k+1,0}H\Pi_n^\rho = \Pi_n^\rho T_{k+1,0}\Pi_n^\rho$, which implies
    \begin{equation*}
        \Pi_n^\rho [H,R_{k+1,0}]\Pi_n^\rho = \Pi_n^\rho T_{k+1,0}\Pi_n^\rho,
    \end{equation*}
    since $[\Pi_n^\rho,R_{k+1,0}] = 0$. One argues similarly for the other relation. 
    This concludes the induction argument. By standard Borel summation procedure, we have the existence of a semiclassical pseudodifferential operator $\hat{\Pi}_n^{\rho,\hbar} \in \Psi_\hbar^{-\infty}(M)$ such that
    \begin{equation*}
        \hat{\Pi}_n^{\rho,\hbar} \sim \sum_{k\geq 0} \hbar^k \hat{\Pi}_n^{\rho,k}.
    \end{equation*}
    Observe that, since for every step $k\geq 0$, the symbol $\Pi_n^{\rho,k+1}$ is uniquely determined by the operator $\hat{\Pi}_n^{\rho,[k]}$ and by Equations~\eqref{eq:recursivePik+1}, the final operator $\hat{\Pi}_n^\rho$ is uniquely determined microlocally on $\{\rho=1\}$.
    Finally, in order to end the proof, we check that we have
    \begin{equation}
        \label{eq:projmicrolocal}
        (\hat{\Pi}_n^{\rho,\hbar})^* = \hat{\Pi}_n^{\rho,\hbar},
    \end{equation}
    microlocally on $\{\rho=1\}$ by the following uniqueness argument: indeed, the semiclassical pseudodifferential operator $(\hat{\Pi}_n^{\rho,\hbar})^*$ has principal symbol $\Pi_n^\rho$, and by self-adjointness of the subLaplacian $-\hbar^2\Delta_{sR}$, it also satisfies
    \begin{equation*}
        [-\hbar^2 \Delta_{sR},(\hat{\Pi}_n^{\rho,\hbar})^*] = 0 \quad\mbox{microlocally on}\ \{\rho=1\},
    \end{equation*}
    and is also trivially microlocally a projector. Thus, Equation~\eqref{eq:projmicrolocal} holds.
\end{proof}

\begin{remark}
    Since this recursive construction of the Landau projectors $\hat{\Pi}_n^{\rho,\hbar}$, $n\geq 0$, once having singled out a particular quantization procedure $\Op_\hbar$, entirely follows the proof of \cite[Theorem 1.2]{BFKRV}, the propositions \cite[Proposition 1.3]{BFKRV} and \cite[Proposition 2.5]{BFKRV}, also extends to our setting. In particular, we have that for $n,m \geq 0$ two distinct integers the corresponding Landau projectors are microlocally orthogonal:
    \begin{equation*}
        \hat{\Pi}_n^{\rho,\hbar}\circ \hat{\Pi}_m^{\rho,\hbar} = 0\quad\mbox{microlocally on}\ \{\rho = 1\}.
    \end{equation*}
    Similarly, for $f\in C_c^\infty(\R)$, we also have
    \begin{equation*}
         [f(-\hbar^2\Delta_{sR}),\hat{\Pi}_n^{\rho,\hbar}] = 0\quad\mbox{microlocally on}\ \{\rho=1\}.
     \end{equation*}
     However, we will not make use of these two further properties in the following sections.
\end{remark}

\section{Egorov theorem}
\label{sect:egorov}

This section is dedicated to the proof of an Egorov theorem for the subLaplacian when restricted to each Landau level, i.e.~the ranges of the Landau projectors.

\subsection{Toeplitz operators}

Let $\rho\in C_c^\infty(\Ghat_\infty M,[0,1])$ be an admissible cutoff function. We introduce particular spaces of observables in $\Psi_\hbar^{-\infty}(M)$ corresponding to each Landau level.

\begin{definition}
    \label{def:gentoeplitz}
    For $n\geq 0$, the space of Toeplitz operators associated with the $n$-th Landau level, denoted by $\mathscr{T}_n^\rho(M)$, is the subspace of $\Psi_\hbar^{-\infty}(M)$ composed of semiclassical pseudodifferential operators of the form
    \begin{equation*}
        \hat{\Pi}_n^{\rho,\hbar}\circ A \circ\hat{\Pi}_n^{\rho,\hbar},
    \end{equation*}
    for $A\in\Psi_\hbar^{\infty}(M)$ which is admissible and with its wavefront set contained in $\{\rho = 1\}$.
\end{definition}

The denomination of Toeplitz operators has been introduced in a similar setting by various authors, e.g.\ \cite{Ep} and \cite{Tay}, where the role of the Landau projectors is played by similar operators called generalized Szegö projectors.

\vspace{0.25cm}

Working with the space of observables $\mathscr{T}_n^\rho(M)$, we are led to consider the following vector bundle over $\Ghat_\infty M\sim\Sigma$.

\begin{definition}
    \label{def:fibreLandau}
    For $n\geq 0$, we define the \emph{$n$-th Landau bundle} $\mathscr{L}_n(N\Sigma)$ as the smooth bundle over $\Sigma$ obtained as the range of the smooth field of orthogonal Landau projectors $\Pi_n$ acting on the Fock Hilbert bundle $\mathscr{F}(N\Sigma)$:
    \begin{equation*}
        \mathscr{L}_n(N\Sigma) = \Pi_n \mathscr{F}(N\Sigma).
    \end{equation*}
\end{definition}

This bundle is of use for the following particular reason: for any observable in $\mathscr{T}_n^\rho(M)$ of the form $\hat{\Pi}_n^{\rho,\hbar}A\hat{\Pi}_n^{\rho,\hbar}$, the notion of principal symbol as introduced in Proposition~\ref{prop:principalsym}, can then naturally be reinterpreted as a section of the bundle ${\rm End}(\mathscr{L}_n(N\Sigma))$.

\vspace{0.25cm}

The following crucial result clarifies the role played by the Landau projectors in the understanding of the relation between the subLaplacian and the Reeb vector field on the contact manifold $(M,\eta)$.

\begin{proposition}
    \label{prop:opeffectif}
    For $n\geq 0$, there exists a unique Toeplitz operator $\Theta_n\in\mathscr{T}_n^\rho(M)$ satisfying, microlocally on $\{\rho = 1\}$,
    \begin{equation}
        \label{eq:effectiflandau}
        -\hbar^2 \Delta_{sR}\,\hat{\Pi}_n^{\rho,\hbar} = \hat{\Pi}_n^{\rho,\hbar}\left((2n+d)\hbar^2|R|\right)\hat{\Pi}_n^{\rho,\hbar} + \hbar^2 \Theta_n.
    \end{equation}
    We denote by $\theta_n\in C^\infty(\Ghat_\infty M,{\rm End}(\mathscr{L}_n(N\Sigma)))$ the principal symbol of $\Theta_n$.
\end{proposition}

This result can be interpreted as a global normal form of the subLaplacian when restricted to a Landau level.

\begin{remark}
    Since the remainder $\Theta_n$ is uniquely determined microlocally on $\{\rho = 1\}$, a sequence of cutoff functions $\rho$ converging to $1$ can be used to uniquely define the principal symbol $\theta_n$ as a section on the whole characteristic manifold $\Sigma$. Thus, we do not mark the dependence on the cutoff $\rho$ in its notation.
\end{remark}

\begin{proof}[Proof of Proposition~\ref{prop:opeffectif}]
    By definition of the space $\mathscr{T}_n^\rho(M)$, the proposition can be deduced from the following result:
    \begin{equation}
        \label{eq:noorder1}
        \hat{\Pi}_n^{\rho,\hbar}\left(-\hbar^2\Delta_{sR} - (2n+d)\hbar^2|R|\right)\hat{\Pi}_n^{\rho,\hbar} \in \hbar^2\Psi_\hbar^{-\infty}(M)\quad\mbox{microlocally on}\ \{\rho=1\}.
    \end{equation}
    Note that by simple symbolic calculus, by definition of $\Pi_n^\rho$, we know that this operator is at least in the space $\hbar\Psi_\hbar^{-\infty}(M)$.
    
    Following the definition of the pseudodifferential calculus, one can localize on $M$ in order to make use of a frame and a local quantization to compute this pseudodifferential operator. We choose to do so with respect to a Darboux frame $\tilde{\bX} = (\tilde{X}_1,\dots,\tilde{X}_d,\tilde{Y}_1,\dots,\tilde{Y}_d,R)$. In such a frame, as Section~\ref{subsubsect:locsemiquantiz} and Equation~\eqref{eq:equivquantizations} demonstrate, the local quantization given by Equation~\eqref{eq:localquantization} then coincides with the semiclassical quantization $\Op_{\Heis^d,\hbar}$ on the Heisenberg group $\Heis^d$.

    This allows us to write these different operators as follows 
    \begin{equation*}
        -\hbar^2 \Delta_{sR} = \Op_{\Heis^d,\hbar}(\tilde{H}_0 + \hbar \tilde{H}_1) \quad\mbox{and}\quad \hbar^2|R| = \Op_{\Heis^d,\hbar}(\lambda),
    \end{equation*}
    where $\tilde{H}_0$ and $\tilde{H}_1$ are respectively in $S^2(\Heis^d\times\Hhat^d)$ and $S^1(\Heis^d\times\Hhat^d)$, as well as
    \begin{equation*}
        \hat{\Pi}_n^{\rho,\hbar} = \Op_{\Heis^d,\hbar}(\tilde{\Pi}_n^0 + \hbar\tilde{\Pi}_n^1+\hbar^2 \tilde{\Pi}_n^{\geq2}),
    \end{equation*}
    with $\tilde{\Pi}_n^0$, $\tilde{\Pi}_n^1$ and $\tilde{\Pi}_n^{\geq2}$ being smoothing symbols on $\Heis^d$. Moreover, by definition, we have the following relation:
    \begin{equation*}
        \tilde{H}_0 \tilde{\Pi}_n^0 = (2n+d)\tilde{\Pi}_n^0.
    \end{equation*}

    In order to proceed with the proof, we need to say more about the principal symbol $\tilde{H}_0$ of the subLaplacian, computed with respect to the Darboux frame $\tilde{\bX}$. Introducing the notation $(\tilde{g}_{i,j}(x))_{1\leq i,j\leq 2d}$, $x\in\Heis^d$, for the matrix associated to the inner product $g_{\cD}(x)$ computed in the left-invariant basis $(\tilde{V}_1,\dots,\tilde{V}_d) = (\tilde{X}_1,\dots,\tilde{X}_d,\tilde{Y}_1,\dots,\tilde{Y}_d)$. The symbol $\tilde{H}_0$ is then given by
    \begin{equation*}
        \tilde{H}_0(x,\tilde{\pi}^\lambda) = \sum_{1\leq i,j\leq 2d} \tilde{g}_{i,j}(x)\,\tilde{\pi}^\lambda(\tilde{V}_i)\tilde{\pi}^\lambda(\tilde{V}_j).
    \end{equation*}
    Thus, it is a second order differential operator acting on $\cH_{\tilde{\pi}^\lambda}^{-\infty} = L^2(\tilde{\mathfrak{p}}_x)$. Moreover, it can be written as a pseudodifferential operator whose Weyl symbol is given by the quadratic form associated to the inner product $g_{\cD}(x)$ acting on $\cD_x \sim \tilde{\mathfrak{v}}_x$. Here, the vector space $\tilde{\mathfrak{v}}_x$ is seen as a symplectic space with symplectic form $(d\eta)_{|\cD,x}$, split as the sum of two Lagrangian subspaces $\tilde{\mathfrak{v}}_x = \tilde{\mathfrak{p}}_x\oplus\tilde{\mathfrak{q}}_x$. In particular, since the inner product $g_\cD$ comes from an adapted complex structure given by the endomorphism $\phi_{|\cD}$, it follows from the construction of $\phi$-frames in Section~\ref{subsubsect:phiframe} or equivalently from Williamson's theorem (see \cite[Section 10]{Ep} and \cite[Proposition 32]{MAG}) that one can find symplectic coordinates $(\xi,\zeta)\in\R^{2d}$ such that the quadratic form associated to $g_\cD(x)$ can be written as
    \begin{equation*}
        |\lambda|\sum_{i=1}^d  (\xi_i^2 + \zeta_i^2).
    \end{equation*}
    We conclude that, up to conjugation with a metaplectic unitary operator acting on $L^2(\tilde{\mathfrak{p}}_x)$, the symbol $\tilde{H}_0(x,\tilde{\pi}^\lambda)$ can be reduced to a sum of harmonic oscillators.

    Now observe that by the composition formula for the Heisenberg quantization (see for example \cite[Theorem 5.1]{FM}), we can write
    \begin{equation*}
        \left(-\hbar^2\Delta_{sR} - (2n+d)\hbar^2|R|\right)\hat{\Pi}_n^{\rho,\hbar} = \hbar\Op_{\Heis^d,\hbar}\left(T_1\right) + \hbar^2 \Psi_\hbar^{-\infty}(\Heis^d),
    \end{equation*}
    where 
    \begin{equation*}
        T_1 = (\tilde{H}_0-(2n+d)|\lambda|)\tilde{\Pi}_n^1 + \tilde{H}_1\tilde{\Pi}_n^0 + \Delta_{\tilde{v}} (\tilde{H}_0-(2n+d)|\lambda|) \cdot \tilde{V}\tilde{\Pi}_n^0.
    \end{equation*}
    Thus, using that the operator $\hat{\Pi}_n^{\rho,\hbar}$ is microlocally a projector on $\{\rho=1\}$, in order to prove Equation~\eqref{eq:noorder1}, it is enough to prove that we have
    \begin{equation*}
        \tilde{\Pi}_n^0 T_1 \tilde{\Pi}_n^0 = 0\quad\mbox{on}\  \{\rho=1\}.
    \end{equation*}
    In order to simplify the expression of the symbol $T_1$, we need to make use of the symbolic equation~\eqref{eq:recursivePik+1} satisfied by $\tilde{\Pi}_n^1$ as the first corrector of $\tilde{\Pi}_n^0$ in the construction of the Landau projector using the Heisenberg quantization $\Op_{\Heis^d,\hbar}$ associated to the Darboux frame $\tilde{\bX}$. Explicitly, one finds, once again by the composition formula for semiclassical pseudodifferential operators, that
    \begin{equation*}
        [\tilde{H}_0,\tilde{\Pi}_n^1] + [\tilde{H}_1,\tilde{\Pi}_n^0] + \Delta_{\tilde{v}} \tilde{H}_0\cdot \tilde{V}\tilde{\Pi}_n^0 - \Delta_{\tilde{v}} \tilde{\Pi}_n^0\cdot \tilde{V}\tilde{H}_0 = 0.
    \end{equation*}
    Thus, on $\{\rho=1\}$ we have
    \begin{equation*}
        \tilde{\Pi}_n^0 T_1 \tilde{\Pi}_n^0 = \tilde{\Pi}_n^0\tilde{H}_1 \tilde{\Pi}_n^0 + \tilde{\Pi}_n^0\left(\Delta_{\tilde{v}} \tilde{\Pi}_n^0\cdot \tilde{V}\tilde{H}_0\right) \tilde{\Pi}_n^0.
    \end{equation*}
    One then concludes as follows: up to metaplectic conjugation as described above, one is reduced to the case where $\tilde{H}_0$ is a sum of harmonic oscillators and $\tilde{\Pi}_n^0$ is a projector on the $n$-th Landau level. By Remark~\ref{rem:oddeven}, the range of $\tilde{\Pi}_n^0$ is composed of either odd or even functions depending on $n$. Since the symbol $\Tilde{H}_1$ is a sum of first order homogeneous differential operators and of multiplication operators by first order homogeneous polynomials, it follows from a parity argument that we have $\tilde{\Pi}_n^0\tilde{H}_1 \tilde{\Pi}_n^0 = 0$.
    By the same argument, since the difference operators $\Delta_{\tilde{v}}$ are commutators with first order homogeneous differential operators or multiplication operators by first order homogeneous polynomials (see \cite[Section 6.3.3]{FR}) and since we have
    \begin{equation*}
        \tilde{V}\tilde{H}_0(x,\tilde{\pi}^\lambda) = \sum_{1\leq i,j\leq 2d} \tilde{V}\tilde{g}_{i,j}(x)\,\tilde{\pi}^\lambda(\tilde{V}_i)\tilde{\pi}^\lambda(\tilde{V}_j), 
    \end{equation*}
    we also deduce that $\tilde{\Pi}_n^0\left(\Delta_{\tilde{v}} \tilde{\Pi}_n^0\cdot \tilde{V}\tilde{H}_0\right) \tilde{\Pi}_n^0 = 0$, and this ends the proof.
\end{proof}

In order to make use of the previous proposition to compare the quantum propagation of the subLaplacian on each Landau level with the one of the Reeb vector field, we are led to rewrite Equation~\eqref{eq:effectiflandau} as follows: for $n\geq 0$, we have
\begin{equation}
    \label{eq:effectivesubLapequality}
    -\hbar^2 \Delta_{\rm sR}\,\hat{\Pi}_n^{\rho,\hbar} = -\hbar^2\Delta_{\rm sR}^{{\rm eff},n}\,\hat{\Pi}_n^{\rho,\hbar},
\end{equation}
microlocally on $\{\rho = 1\}$, where we introduce the \emph{effective subLaplacian} $\Delta_{\rm sR}^{{\rm eff},n}$ defined by
\begin{equation*}
    -\hbar^2\Delta_{\rm sR}^{{\rm eff},n} = (2n+d)\hbar^2|R| + \hbar^2\tilde{\Theta}_n^\hbar,
\end{equation*}
and where we set
\begin{equation*}
    \tilde{\Theta}_n^\hbar = \Theta_n^\hbar + \frac{2n+d}{\hbar^2}\left[\hat{\Pi}_n^{\rho,\hbar},\hbar^2 |R|\right]\in \Psi_\hbar^{-\infty}(M).
\end{equation*}

We stress that the fact that the commutator $\left[\hat{\Pi}_n^{\rho,\hbar},\hbar^2|R|\right]$ on the right-hand side of the previous equation defines a semiclassical pseudodifferential operator in $\hbar^2\Psi_\hbar^{-\infty}(M)$ follows from a straightforward computation in a local Darboux frame: denoting by $\tilde{\Pi}_n^{\rho,0}$ the principal symbol of $\hat{\Pi}_n^{\rho,\hbar}$ in the trivialization given by this frame, we can apply the symbolic calculus associated with the quantization of the Heisenberg group, as in \cite[Theorem 5.7]{FM}, to obtain that, locally,
\begin{equation*}
    \left[\hat{\Pi}_n^{\rho,\hbar},\hbar^2 |R|\right] = \hbar^2\Op_{\Heis^d,\hbar}(\vec{|\lambda|} \tilde{\Pi}_n^{\rho,0}) + \hbar^3\Psi_\hbar^{-\infty}(M),
\end{equation*}
where we used the fact that the action of the difference operators of the first stratum $\Delta_{x_j}, \Delta_{y_j}$ as defined in \cite[Section 6.3]{FR}, have trivial actions on the symbol $|\lambda|$ of the Reeb vector field. In particular, we observe that the principal symbol of this commutator is anti-diagonal with respect to $\Pi_n$, meaning that if we denote by $\tilde{\theta}_n$ the principal symbol of $\tilde{\Theta}_n^\hbar$, we have
\begin{equation}
    \label{eq:relationtildethetatheta}
    \Pi_n \tilde{\theta}_n\Pi_n = \theta_n.
\end{equation}

We wish to emphasize here that the \emph{effective subLaplacian} $-\hbar^2\Delta_{\rm sR}^{{\rm eff},n}$ on the $n$-th Landau level is no longer a self-adjoint operator but will, however, be crucial for the proof of our Egorov theorem in the next section.

\subsection{Egorov theorem for Toeplitz operators}

Proposition~\ref{prop:opeffectif} leads us to consider a particular geometric flow on the Landau bundles that we now explicit.

\subsubsection{Partial connections on the Landau bundles}
\label{subsubsect:partialconnectLandau}

Indeed, the result of Proposition~\ref{prop:opeffectif} tells us that the subLaplacian $-\hbar^2\Delta_{sR}$ acts effectively on the $n$-th Landau level, $n\geq 0$, as does the  operator $\hbar^2|R| + \hbar^2\Theta_n$. This effective operator defines a geometric flow $\Phi^n_t$ on the $n$-th Landau bundle $\mathscr{L}_n(N\Sigma)$ as follows: the flow lines $(\sigma(t),v(t))$ consist of the orbits $\sigma(t)\in\Sigma$ of the Hamiltonian vector field $\vec{\lambda}\in\Gamma(\Sigma)$ as defined by Proposition~\ref{prop:hamliftReeb} and such that $v(t)\in\mathscr{L}_n(N_{\sigma(t)}\Sigma)$ satisfies 
\begin{equation*}
    \frac{d}{dt}v(t) = i\theta_n(\sigma(t))v(t).
\end{equation*}
Note that by Proposition~\ref{prop:hamliftReeb}, $\Phi^n_t$ lifts the Reeb flow on $M$ to the $n$-th Landau bundle.

Observe that the induced flow $(\Phi^n_t)^*$ on $C^\infty(\Sigma,\mathscr{L}_n(N\Sigma))$ defines a partial connection on the space of section of $\mathscr{L}_n(N\Sigma)$. Hence, there is also a flow ${\rm Ad}(\Phi^n_t)$ on the bundle ${\rm End}(\mathscr{L}_n(N\Sigma))$, which in turn defines an action of $\R$ on $C^\infty(\Sigma,{\rm End}(\mathscr{L}_n(N\Sigma)))$, denoted by ${\rm Ad}(\Phi^n_t)^*$.

It is straightforward to check that if $\sigma\in C^\infty(\Sigma,{\rm End}(\mathscr{L}_n(N\Sigma)))$, for any local trivialization given, for example, by a $\phi$-frame $(\bX,U)$, we have
\begin{equation}
    \label{eq:connectionLandau}
    \frac{d}{dt}{\rm Ad}(\Phi^n_t)^*\sigma = \vec{\lambda}\sigma + i[\theta_n,\sigma].
\end{equation}

\subsubsection{The Egorov Theorem}

We are now able to state an Egorov theorem for each class of Toeplitz operators.

\begin{theorem}
    \label{thm:Egorov}
    Let $n\geq 0$ and fix $T\geq 0$. For any Toeplitz operator $B\in \mathscr{T}_n^\rho(M)$ with a principal symbol denoted by $\beta\in C^\infty(\Ghat_\infty M, \mathscr{L}_n(N\Sigma))$, the operators defined by
    \begin{equation*}
        \tilde{B}(t) := e^{it\Delta_{sR}} \circ B\circ e^{-it\Delta_{sR}},
    \end{equation*}
    form a family of Toeplitz operators in $\mathscr{T}_n^\rho(M)$ whose principal symbols $(\tilde{\beta}(t))_{t\in[-T,T]}$ satisfy
    \begin{equation*}
        \forall\,t\in[-T,T],\quad \tilde{\beta}(t) = {\rm Ad}(\Phi^n_t)^* \beta.
    \end{equation*}
    In particular, we have
    \begin{equation*}
        \forall\,t\in[-T,T],\quad{\rm WF}'_\infty(\tilde{B}(t))\subset \{\rho = 1\}.
    \end{equation*}
\end{theorem}

\begin{proof}
    We start by recalling that $\Op_\hbar$ denotes a fixed global quantization procedure on $M$. This allows us to write the Toeplitz operator as follows:
    \begin{equation*}
        B = \hat{\Pi}_n^{\rho,\hbar}(\Op_\hbar(\beta_\hbar) + R^\hbar)\hat{\Pi}_n^{\rho,\hbar},
    \end{equation*}
    where $\beta_\hbar = \beta + \O(\hbar)$ is in $S^{-\infty}(\Ghat M)$ and $R^\hbar$ is semiclassically smoothing on the Sobolev scale; moreover, we have ${\rm ess\,supp}(\beta_\hbar) \subset\{\rho = 1\}$.
    It is enough to construct a smooth family $t\in[-T,T]\mapsto \tilde{\beta}_\hbar(t)\in S^{-\infty}(\Ghat M)$ of symbols contained in $C^\infty(\Ghat_\infty M, \mathscr{L}_n(N\Sigma))$ such that $\tilde{\beta}_\hbar(0) = \beta_\hbar$ and satisfying
    \begin{equation}
        \label{eq:egorovnecessaryeq}
        \hat{\Pi}_n^{\rho,\hbar} \Op_\hbar(\partial_t\tilde{\beta}_\hbar(t))\hat{\Pi}_n^{\rho,\hbar} = \frac{i}{\hbar^2}\left[-\hbar^2\Delta_{sR},\hat{\Pi}_n^{\rho,\hbar} \Op_\hbar(\tilde{\beta}_\hbar(t))\hat{\Pi}_n^{\rho,\hbar}\right] + \tilde{R}^\hbar(t),
    \end{equation}
    where $\tilde{R}^\hbar(t)$ depends smoothly on $t$ and is semiclassically smoothing on the Sobolev scale, and where, furthermore, the essential support of $\tilde{\beta}_\hbar(t)$ is contained in $\{\rho =1\}$ at all times. Then, by simple integration in time and the use of Duhamel's formula associated with the propagator $e^{it\Delta_{sR}}$, we obtain that $\tilde{B}(t)$ coincides with $\hat{\Pi}_n^{\rho,\hbar} \Op_\hbar(\tilde{\beta}_\hbar(t))\hat{\Pi}_n^{\rho,\hbar}$ up to a semiclassically smoothing remainder.

    Now observe that using Equation~\eqref{eq:effectivesubLapequality}, the commutator in Equation~\eqref{eq:egorovnecessaryeq} can be rewritten as follows:
    \begin{multline*}
            \left[-\hbar^2\Delta_{sR},\hat{\Pi}_n^{\rho,\hbar} (\Op_\hbar(\tilde{\beta}_\hbar(t))\hat{\Pi}_n^{\rho,\hbar}\right] = -\hbar^2\Delta_{sR} \hat{\Pi}_n^{\rho,\hbar} \,\Op_\hbar(\tilde{\beta}_\hbar(t)\,\hat{\Pi}_n^{\rho,\hbar} - \hat{\Pi}_n^{\rho,\hbar} \,\Op_\hbar(\tilde{\beta}_\hbar(t)\,\hat{\Pi}_n^{\rho,\hbar}(-\hbar^2\Delta_{sR})\\
            = \hat{\Pi}_n^{\rho,\hbar} (-\hbar^2\Delta_{\rm sR}^{{\rm eff},n})^*\,\Op_\hbar(\tilde{\beta}_\hbar(t))\,\hat{\Pi}_n^{\rho,\hbar} - \hat{\Pi}_n^{\rho,\hbar} \,\Op_\hbar(\tilde{\beta}_\hbar(t))\, (-\hbar^2\Delta_{\rm sR}^{{\rm eff},n})\hat{\Pi}_n^{\rho,\hbar} + \O(\hbar^\infty)\\
            = \hat{\Pi}_n^{\rho,\hbar}((-\hbar^2\Delta_{\rm sR}^{{\rm eff},n})^*\,\Op_\hbar(\tilde{\beta}_\hbar(t)) - \Op_\hbar(\tilde{\beta}_\hbar(t))\, (-\hbar^2\Delta_{\rm sR}^{{\rm eff},n}))\hat{\Pi}_n^{\rho,\hbar}+\O(\hbar^\infty).
    \end{multline*}
    Thus, at the level of the principal symbols, using Equation~\eqref{eq:relationtildethetatheta} and any local trivialization given by a Darboux frame, Equation~\eqref{eq:egorovnecessaryeq} becomes
    \begin{equation}
        \label{eq:firsteqegorov}
        \partial_t \tilde{\beta}(t) = \vec{\lambda}\tilde{\beta}(t) + i[\theta_n,\tilde{\beta}(t)],
    \end{equation}
    where we introduce the notation $\tilde{\beta}$ for the leading term of $\tilde{\beta}_\hbar$.
    This leads us to define $\tilde{\beta}(t)$ by
    \begin{equation*}
        \tilde{\beta}(t) = {\rm Ad}(\Phi_t^n)^*\beta.
    \end{equation*}
    Observe that $t\mapsto\tilde{\beta}(t)$ satisfies the requirement of being a section of $\mathscr{L}_n(N\Sigma)$ with support in $\{\rho = 1\}$.
    
    Now, suppose that for some $N\geq 0$, a family $\tilde{\beta}(t),\tilde{\beta}_1(t),\dots,\tilde{\beta}_N(t)$ has been constructed such that 
    \begin{equation*}
        \hat{\Pi}_n^{\rho,\hbar} \Op_\hbar(\partial_t\tilde{\beta}_\hbar^{\leq N}(t))\hat{\Pi}_n^{\rho,\hbar} - \frac{i}{\hbar^2}\left[-\hbar^2\Delta_{sR},\hat{\Pi}_n^{\rho,\hbar} \Op_\hbar(\tilde{\beta}_\hbar^{\leq N}(t))\hat{\Pi}_n^{\rho,\hbar}\right] = \hbar^{N+1} \tilde{R}_N^\hbar(t)\in \hbar^{N+1}\Psi_\hbar^{-\infty}(M),
    \end{equation*}
    where we set $\tilde{\beta}_\hbar^{\leq N}(t) = \tilde{\beta}(t) + \hbar \tilde{\beta}_1(t)+\dots+\hbar^N\tilde{\beta}_N(t)$, and satisfying, moreover, $\tilde{\beta}_1(0) = \dots = \tilde{\beta}_N(0) = 0$. We point out that the wavefront set of $\tilde{R}_N^\hbar(t)$ is contained in $\{\rho = 1\}$ by simple application of the results obtained in Section~\ref{subsubsect:wavefrontset}. To perform our induction, we are now looking for a symbol $\tilde{\beta}_{N+1}(t)$ that allows us to go one degree further in $\hbar$; using the same computations that led to Equation~\eqref{eq:firsteqegorov}, this requirement translates into the following condition:
    \begin{equation*}
        \partial_t \tilde{\beta}_{N+1}(t) = \vec{\lambda}\tilde{\beta}_{N+1}(t) + i[\theta_n,\tilde{\beta}_{N+1}(t)] - \tilde{r}_n(t),
    \end{equation*}
    where $\tilde{r}_n$ denotes the principal symbol of $\tilde{R}_N$, which is also a section of $\mathscr{L}_n(N\Sigma)$. Using Duhamel's formula for the flow ${\rm Ad}(\Phi_t^n)^*$, this equation can be solved for the initial condition $\tilde{\beta}_{N+1}(0) = 0$.

    By Borel summation, we can define the desired symbol $\tilde{\beta}_\hbar(t)$ as $\tilde{\beta}(t)+ \sum_{k\geq 1}\hbar^k\tilde{\beta}_k(t)$, which is easily seen to satisfy all requirements. This concludes our proof.
\end{proof}

\section{Quantum ergodicity}
\label{sect:quantumergodicity}

The proof of Theorem~\ref{thm:main} follows the general strategy of proving quantum ergodicity in the elliptic setting and begins with investigating the microlocal Weyl law for observables in $\Psi_\hbar^{-\infty}(M)$.

\subsection{Microlocal Weyl laws}

We recall the following notations concerning the eigenvalue problem for the subLaplacian. The self-adjoint operator $(-\Delta_{sR},{\rm Dom}(-\Delta_{sR}))$ has a discrete point spectrum given by the sequence 
\begin{equation*}
    0\leq \delta_0\leq \delta_1\leq \dots\leq\delta_k\leq\dots\rightarrow +\infty,
\end{equation*} 
(with repetitions according to multiplicity). Moreover, there exists an orthonormal basis $(\varphi_k)_{k\in\N}$, such that for every $k\in \N$,
\begin{equation*}
    -\Delta_{sR}\varphi_k = \delta_k \varphi_k \quad\mbox{and}\quad \|\varphi_k\|_{L^2(M)} = 1.
\end{equation*}

Thus, reintroducing the semiclassical parameter $\hbar>0$ in front of the subLaplacian, we can write
\begin{equation*}
    \forall k\in \N,\quad -\hbar^2\Delta_{sR}\varphi_k = \delta_k(\hbar)\varphi_k\quad\mbox{where}\ \delta_k(\hbar) = \hbar^2\delta_k.
\end{equation*}

We then have the following theorem concerning the microlocal repartition of the eigenvalues, to be compared with the results of \cite{Tay20}.

\begin{theorem}
    \label{thm:microlocalweyllaw}
    Let $B\in \Psi_\hbar^{-\infty}(M)$ be an admissible semiclassical pseudodifferential operator. Then, for any two real numbers $a < b$, we have the following asymptotic
    \begin{equation*}
        \hbar^{Q} \sum_{a\leq \delta_k(\hbar)\leq b}\langle B\varphi_k,\varphi_k\rangle \Tend{\hbar}{0} {\rm vol}(M)\int_{\Ghat M} {\rm Tr}\left({\bm 1}_{[a,b]}(H)(x,\pi)\beta(x,\pi)\right) d\mu_{\Ghat M}(x,\pi),
    \end{equation*}
    where $\beta \in S^{-\infty}(\Ghat M)$ denotes the principal symbol of $B$ and we recall that $Q = 2d+2$ denotes the \emph{homogeneous dimension} of $(M,\eta)$.
\end{theorem}

\begin{remark}
    For $n\geq 0$, applying this result to $B = \hat{\Pi}_n^{\rho_k,\hbar}$, a Landau projector with $(\rho_k)_{k\in\N}$ a sequence of cutoff functions converging to the constant 1, and using Equation~\eqref{eq:multiplicity}, we compute that the main term in the microlocal Weyl law for $a=0$ and $b=1$, tends to 
    \begin{equation*}
        2\frac{{\rm vol}(M)}{(d+1)(2\pi)^{3d+1}}\frac{\binom{n+d-1}{n}}{(2n+d)^{d+1}}.
    \end{equation*}
    This coefficient can be interpreted as the proportion of eigenvalues in the Weyl law for $-\Delta_{sR}$ associated to the $n$-th Landau level.
\end{remark}

We first start with a general consideration of the trace of semiclassical pseudodifferential operators.

\begin{lemma}
    \label{lem:traceclass}
    Any operator $B\in \Psi_\hbar^{-\infty}(M)$ is of trace class on $L^2(M)$. Moreover, if $\beta\in S^{-\infty}(\Ghat M)$ denotes the principal symbol of $B$, then we have
    \begin{equation*}
        {\rm Tr}(B) = {\rm vol}(M)\int_{\Ghat M} {\rm Tr}_{\mathcal{H}_\pi}(\beta(x,\pi))\,d\hat{\mu}_{\Ghat M}(x,\pi) + \O(\hbar).
    \end{equation*}
\end{lemma}

\begin{proof}
We make use of Proposition~\ref{prop:globalquantization}, to decompose the operator $B$ as follows:
\begin{equation}
    \label{eq:decompositionquantization}
    B = \sum_{\alpha\in A} \Op_\hbar^{\bX_\alpha,\chi_\alpha}(\psi_\alpha \beta_\hbar)\psi_\alpha + R_\hbar,
\end{equation}
where $\beta_\hbar$ is a symbol in $S^{-\infty}(\Ghat M)$ such that $\beta_\hbar = \beta + \hbar S^{-\infty}(\Ghat M)$, $\mathcal{A} = (\bX_\alpha,U_\alpha,\chi_\alpha,\psi_\alpha)$ is an atlas and $R_\hbar: C^\infty(M)\rightarrow C^\infty(M)$ is semiclassically smoothing on the Sobolev scale. Thus, it is enough to show that each term in this decomposition is of trace class.

For a smoothing symbol $\sigma\in S^{-\infty}(\Ghat M)$, following Equation~\eqref{eq:integral_kernel}, we have that for $\alpha\in A$, the integral kernel of the operator $\Op_\hbar^{\bX_\alpha,\chi_\alpha}(\psi_\alpha \sigma)\psi_\alpha$ is given by 
\begin{equation*}
     K_\hbar^{\bX_\alpha}(x,y) = \hbar^{-Q} \kappa_{\sigma,x}^{\bX_\alpha} (- \hbar^{-1}  \ln^\bX_x (y)) \ \chi_x (y)\ \jac_{y} (\ln^\bX_x)\psi_\alpha(x)\psi_\alpha(y),
\end{equation*}
where $\kappa_\sigma^{\bX_\alpha}$ is defined by Equation~\eqref{eq:localassociatedkernel}.
Thus, by Proposition~\ref{prop:fourierschwartz}, $K_\hbar^{\bX_\alpha}$ is a smooth function on $M\times M$. Thus, the operator $\Op_\hbar^{\bX_\alpha,\chi_\alpha}(\psi_\alpha \sigma)\psi_\alpha$ is of trace class (see, for example, \cite[Theorem C.16]{Zwo}). By the same argument, the remainder $R_\hbar$, having a smooth kernel since it is smoothing on the Sobolev scale, and $M$ being compact, is also of trace class and its trace is of order $\O(\hbar^\infty)$.

Moreover, the trace of the operator $\Op_\hbar^{\bX_\alpha,\chi_\alpha}(\psi_\alpha \sigma)\psi_\alpha$ can be related to its kernel $K_\hbar^{\bX_\alpha}$ by the following formula (see, e.g., \cite[Theorem C.18]{Zwo}):
\begin{equation*}
    \begin{aligned}
    {\rm Tr}(\Op_\hbar^{\bX_\alpha,\chi_\alpha}(\psi_\alpha \sigma)\psi_\alpha) 
    &= \int_{M} K_\hbar^{\bX_\alpha}(x,x)\,d{\rm vol}(x)\\
    &= \hbar^{-Q}\int_{M} \kappa_{\sigma,x}^{\bX_\alpha}(0)\psi_\alpha(x)^2\,d{\rm vol}(x)\\
    &= \hbar^{-Q}\int_{M} \int_{\Ghat_x M} {\rm Tr}_{\mathcal{H}_\pi}\left(\sigma(x,\pi)\right)\psi_{\alpha}(x)^2\,d\hat{\mu}_{\Ghat_x M}(\pi)d\nu(x),
    \end{aligned}
\end{equation*}
where the last equation is obtained using the Fourier inversion formula (see Equation~\eqref{eq:fourierinversion}), the smooth Plancherel system being compatible with the smooth Haar system chosen to define the kernel $\kappa_{\sigma}$.

Plugging this formula in the decomposition~\eqref{eq:decompositionquantization}, we find that
\begin{equation*}
    \begin{aligned}
    {\rm Tr}(B) &= \hbar^{-Q}\sum_{\alpha\in A}\int_{M} \int_{\Ghat_x M} {\rm Tr}_{\mathcal{H}_\pi}\left(\beta(x,\pi)\right)\psi_{\alpha}(x)^2\,d\hat{\mu}_{\Ghat_x M}(\pi)d\nu(x) + \O(\hbar)\\
    &=\hbar^{-Q}\int_{\Ghat M} {\rm Tr}_{\mathcal{H}_\pi}\left(\beta(x,\pi)\right)\,d\hat{\mu}_{\Ghat M}(x,\pi) + \O(\hbar).
    \end{aligned}
\end{equation*}
\end{proof}

We are ready to give the proof of the microlocal Weyl's law.

\begin{proof}[Proof of Theorem~\ref{thm:microlocalweyllaw}]
    Let $\eps>0$ and let $\underline{f_\eps}$ and $\overline{f_\eps}$ be two smooth functions on $\R$ valued in $[0,1]$ satisfying
    \begin{equation*}
        \underline{f} {\bm 1}_{[a,b]} = \underline{f}\quad \mbox{and}\quad 
        \overline{f} {\bm 1}_{[a,b]} = {\bm 1}_{[a,b]}, 
    \end{equation*}
    and such that ${\rm supp}(\underline{f_\eps})\subset [a+\eps,b-\eps]$ and ${\rm supp}(\overline{f_\eps})\subset [a-\eps,b+\eps]$.
    We also denote by $\Pi_{[a,b]}$ the spectral projector of $-\hbar^2\Delta_{sR}$ on the energy window $[a,b]$. Then,we can write
    \begin{equation*}
        \sum_{a\leq \delta_k(\hbar)\leq b}\langle B\varphi_k,\varphi_k\rangle = {\rm Tr}\left( B \,\Pi_{[a,b]}\right) = {\rm Tr}\left(\underline{f_\eps}(-\hbar^2\Delta_{sR})B\right)+ {\rm Tr}\left((\overline{f_\eps}(1-\underline{f_\eps}))(-\hbar^2\Delta_{sR}) B\,\Pi_{[a,b]}\right).
    \end{equation*}
    The second term can be estimated using the semiclassical functional calculus given by Theorem~\ref{thm:functionalcalc} and Lemma~\ref{lem:traceclass}, from which we can deduce that
    \begin{equation*}
        {\rm Tr}\left(\left|\overline{f_\eps}(1-\underline{f_\eps}))(-\hbar^2\Delta_{sR})\right|\right)\leq (c_1(\eps) + c_2(\hbar,\eps))\hbar^{-Q},
    \end{equation*}
    where $\lim_{\eps\rightarrow 0}c_1(\eps) = 0$ by support condition on the principal symbol of this operator relatively to $\eps$, following the definition of $\underline{f_\eps}$ and $\overline{f_\eps}$, and where $\lim_{\hbar\rightarrow 0}c_2(\hbar,\eps) = 0$.

    By the same argument applied to the first term, we readily get that
    \begin{equation*}
        \begin{aligned}
            \hbar^Q {\rm Tr}\left(B \,\Pi_{[a,b]}\right)
            &= \hbar^Q {\rm Tr}\left(\underline{f_\eps}(-\hbar^2\Delta_{sR})B\right) + \O\left(\|B\,\Pi_{[a,b]}\|_{\mathcal{L}(L^2(M))}(c_1(\eps)+c_2(\hbar,\eps)\right)\\
            &= \int_{\Ghat M}{\rm Tr}_{\mathcal{H}_\pi}\left(\underline{f_\eps}(H)(x,\pi)\beta(x,\pi)\right)d\hat{\mu}_{\Ghat M}(x,\pi) + \O_\eps(\hbar) + \O(c_1(\eps)+c_2(\hbar,\eps)).
        \end{aligned}
    \end{equation*}
    Thus, we have
    \begin{equation*}
        \limsup_{\hbar\rightarrow 0} \hbar^Q {\rm Tr}\left(B \,\Pi_{[a,b]}\right) = \int_{\Ghat M}{\rm Tr}_{\mathcal{H}_\pi}\left(\underline{f_\eps}(H)(x,\pi)\beta(x,\pi)\right)d\hat{\mu}_{\Ghat M}(x,\pi) + \O(c_1(\eps)),
    \end{equation*}
    and similarly for $\liminf_{\hbar\rightarrow 0}$. We conclude by dominated convergence letting $\eps$ go to 0.
\end{proof}

\subsection{Proof of Theorem~\ref{thm:main}}
\label{subsect:proofmain}

Denoting by $N(\delta)$, $\delta \geq 0$, the number of eigenvalues $(\delta_k)_{k\in\N}$ of $-\Delta_{sR}$ that are below $\delta$, it is enough to prove that for any $a\in C^\infty(M)$, we have
\begin{equation}
    \label{eq:vartoprove}
    \frac{1}{N(\delta)} \sum_{\delta_{k}\leq \delta}\left|\langle a(\cdot)\varphi_k,\varphi_k\rangle - \int_{M}a(x)\,d\nu(x)\right|^2\Tend{\delta}{+\infty} 0.
\end{equation}
A well-known lemma due to Koopman and Von Neumann (see, e.g., \cite[Chapter 2.6, Lemma 6.2]{Pet}) and a diagonal extraction argument (as in \cite[Theorem 1.5]{Zwo}) allows us then to deduce the theorem.
        
To prove Equation~\eqref{eq:vartoprove}, we are going to fix the following relation between the semiclassical parameter $\hbar$ and the eigenvalues of $-\Delta_{sR}$: for $n\in\N^*$, we fix
\begin{equation*}
    \hbar_n = \frac{1}{\sqrt{\delta_n}}.
\end{equation*}
With this semiclassical parameter, we can rewrite Equation~\eqref{eq:vartoprove} as follows. We first set the notation
\begin{equation*}
    N_{\hbar_n} = {\rm Card}\left(\{k\in\N\,:\,\delta_k(\hbar_n) \in [0,1]\}\right),
\end{equation*}
for which we have $N_{\hbar_n} = N(\delta_n)$, and, for $B$ any bounded operator on $L^2(M)$, we define the quantum variance by
\begin{equation*}
    {\rm Var}_{\hbar_n}(B) = \frac{1}{N_{\hbar_n}}\sum_{0\leq \delta_k(\hbar_n)\leq 1} \left|\langle B\varphi_k,\varphi_k\rangle\right|^2.
\end{equation*}
Consider the following bounded operator 
\begin{equation*}
    A = \left(a(\cdot)- m_a\right)f(-\hbar_n^2\Delta_{sR}) \in \mathcal{L}(L^2(M)),
\end{equation*}
where $m_a := \int_{M}a(x)\,d\nu(x)$ and $f\in C_c^\infty(\R)$ is a smooth cutoff function equal to 1 on the interval $[0,1]$. Equation~\eqref{eq:vartoprove} is then equivalent to proving
\begin{equation*}
    \limsup_{\hbar_n\rightarrow 0} {\rm Var}_{\hbar_n}(A) = 0.
\end{equation*}
Observe that by Theorem~\ref{thm:functionalcalc}, $A$ is a semiclassical pseudodifferential operator in $\Psi_\hbar^{-\infty}(M)$, up to a semiclassically smoothing remainder that will not play a role in the rest of the proof, with principal symbol given by $(a(\cdot)-m_a)f(H)$.

\begin{remark}
    Observe here the importance of the use of the spectral localization implemented by the operator $f(-\hbar_n^2\Delta_{sR})$, which allows us to consider a pseudodifferential operator $A$ in $\Psi_\hbar^{-\infty}(M)$, but does not affect the value of the variance ${\rm Var}_{\hbar_n}(A)$ as long as the support condition on $f$ is satisfied.
\end{remark}
    
We start by giving a first estimate on the quantum variance. If $B$ is in $\Psi_\hbar^{-\infty}(M)$ with principal symbol denoted by $\beta$, and if $C$ is a bounded operator, then by the Cauchy-Schwarz inequality we have
\begin{equation*}
    \left|\langle C B\varphi_k,\varphi_k\rangle\right|^2\leq \|C B\varphi_k\|_{L^2(M)}^2\leq \|C\|_{\mathcal{L}(L^2(M))}^2\langle B^*B\varphi_k,\varphi_k\rangle.
\end{equation*}
Thus, we can write
\begin{equation}
    \label{eq:aprioribound}
    \begin{aligned}
        \limsup_{\hbar_n\rightarrow 0} {\rm Var}_{\hbar_n}(C B) 
        &\leq \|C\|_{\mathcal{L}(L^2(M))}^2\limsup_{\hbar_n\rightarrow 0}\frac{1}{N_{\hbar_n}}\sum_{0\leq \delta_k(\hbar_n)\leq 1}\langle B^*B\varphi_k,\varphi_k\rangle\\
        &\lesssim \|C\|_{\mathcal{L}(L^2(M))}^2\int_{\Ghat M}{\rm Tr}_{\mathcal{H}_\pi}\left(\beta(x,\pi)^*\beta(x,\pi){\bm 1}_{[0,1]}(H)(x,\pi)\right)d\hat{\mu}_{\Ghat M}(x,\pi),
    \end{aligned}
\end{equation}
where the last equation follows from Theorem~\ref{thm:microlocalweyllaw} since the principal symbol of $B^* B$ is simply given by $\beta^*\beta$ and since, up to a constant that we will neglect, $N_{\hbar_n}$ is of order $\hbar_n^{-Q}$ as $n$ goes to infinity. Note here that the symbol $\lesssim$ implies the existence of a constant (independent of the operators $B$ and $C$) in factor of the right-hand side : since it does not affect our following analysis, we do not bother computing it. 

\begin{remark}
    We note that the estimate given by Equation~\eqref{eq:aprioribound} still holds for $B$ equal to the identity ${\rm Id}$, and this gives a bound of $\limsup_{\hbar_n\rightarrow 0} {\rm Var}_{\hbar_n}(C)$ simply by $\|C\|_{\mathcal{L}(L^2(M))}^2$.
\end{remark}

\vspace{0.25cm}
    
In order to estimate the quantum variance of $A$, we now fix some integer $m\in \N$ and introduce two cutoff functions $\rho_m$ and $\Bar{\rho}_m\in C_c^\infty(\Ghat_\infty M)$ such that $\Bar{\rho}_m(x,\lambda) = 1$ if $|\lambda|$ is in $[(2m-1+d)^{-1},1]$ and $\Bar{\rho}_m(x,\lambda) = 0$ for $|\lambda|\leq (2m+d)^{-1}$, and with $\rho_m = 1$ on the support of $\Bar{\rho}_m$. We consider the associated $m$ first Landau projectors $\hat{\Pi}_j^{\rho_m,\hbar_n}$, $j=1,\dots,m-1$, and we define the following operator
\begin{equation*}
    \hat{\Pi}_{<m}^{\rho_m,\hbar_n} = \sum_{j=0}^{m-1} \hat{\Pi}_j^{\rho_m,\hbar_n} \in \Psi_\hbar^{-\infty}(M).
\end{equation*}
Its principal symbol is given by $\Pi_{<m}^{\rho_m} := \rho_m \sum_{j=0}^{m-1}\Pi_j$.
We now write
\begin{equation}
    \label{eq:decompositionvarA}
    \limsup_{\hbar_n\rightarrow 0}{\rm Var}_{\hbar_n}(A) \leq 3\limsup_{\hbar_n\rightarrow 0}{\rm Var}_{\hbar_n}(B_1)+3\limsup_{\hbar_n\rightarrow 0}{\rm Var}_{\hbar_n}(B_2)+3\limsup_{\hbar_n\rightarrow 0}{\rm Var}_{\hbar_n}(B_3),
\end{equation}
where we have introduced the notations
\begin{equation*}
    B_1 = \hat{\Pi}_{<m}^{\rho_m,\hbar_n}\Op_{\hbar_n}(\Bar{\rho}_m)A\hat{\Pi}_{<m}^{\rho_m,\hbar_n},\quad B_2 =  \hat{\Pi}_{<m}^{\rho_m,\hbar_n}\left({\rm Id}-\Op_{\hbar_n}(\Bar{\rho}_m)\right)A\hat{\Pi}_{<m}^{\rho_m,\hbar_n},
\end{equation*}
and 
\begin{equation*}
    B_3 = \hat{\Pi}_{<m}^{\rho_m,\hbar_n} A({\rm Id}-\hat{\Pi}_{<m}^{\rho_m,\hbar_n}) + ({\rm Id}-\hat{\Pi}_{<m}^{\rho_m,\hbar_n})A \hat{\Pi}_{<m}^{\rho_m,\hbar_n} + ({\rm Id}-\hat{\Pi}_{<m}^{\rho_m,\hbar_n})A({\rm Id}-\hat{\Pi}_{<m}^{\rho_m,\hbar_n}).
\end{equation*}

\vspace{0.25cm}

We first treat the term involving the operator $B_3$.
Observe that, by considering the symbol of this pseudodifferential operator, we have
\begin{equation*}
    B_3 = ({\rm Id}-\hat{\Pi}_{<m}^{\rho_m,\hbar_n})A({\rm Id}-\hat{\Pi}_{<m}^{\rho_m,\hbar_n}) + \O_{L^2(M)\rightarrow L^2(M)}(\hbar).
\end{equation*}
By the estimate of Equation~\eqref{eq:aprioribound}, we can bound its quantum variance as follows:
\begin{equation}
    \label{eq:varB3}
    \limsup_{\hbar_n\rightarrow 0}{\rm Var}_{\hbar_n}(B_3)
    \leq \|A\|_{\mathcal{L}(L^2(M))}^2\int_{\Ghat M} {\rm Tr}_{\mathcal{H}_\pi}\left((1-\Pi_{<m}^{\rho_m})^2{\bm 1}_{[0,1]}(H)\right)d\hat{\mu}_{\Ghat M}.
\end{equation}
Now using the parameterization of $\Ghat_\infty M$ given by Equation~\eqref{eq:sigmagenericdual}, we observe that for $\lambda\in [(2m+d)^{-1},1]$, we have
\begin{equation*}
    {\rm Tr}_{\mathcal{H}_{\pi^\lambda}}\left((1-\Pi_{<m}^{\rho_m}(x,\lambda))^2{\bm 1}_{[0,1]}(H(\lambda)\right) = 0,
\end{equation*}
since $\rho_m(x,\lambda) = 1$ and no eigenfunction of the harmonic oscillator $H(\lambda)$ can simultaneously be orthogonal to the range of the projector $\Pi_{<m} = \sum_{j=0}^{m-1}\Pi_j$ and be associated with an eigenvalue in the energy window $[0,1]$. Thus, the integral in Equation~\eqref{eq:varB3} is supported on the set $\{\lambda\leq (2m+d)^{-1}\}$. Furthermore, we easily deduce from the explicit description of the spectrum of $H(\lambda)$ and from the multiplicity of each of its eigenvalues given by Equation~\eqref{eq:multiplicity} that there exists a positive constant $C_d>0$ (independent of $m$) such that
\begin{equation}
    \label{eq:estimatetraceH}
    {\rm Tr}_{\mathcal{H}_{\pi^\lambda}}\left({\bm 1}_{[0,1]}(H(\lambda))\right)\leq C_d |\lambda|^{-d}.
\end{equation}
We are thus left with the following bound for the quantum variance of $B_2$:
\begin{multline}
    \label{eq:firstestimateB3}
    \limsup_{\hbar_n\rightarrow 0}{\rm Var}_{\hbar_n}(B_3)\leq C_d \|A\|_{\mathcal{L}(L^2(M))}^2\int_{M\times \R^*} {\bm 1}_{[0,(2m+d)^{-1}]}(|\lambda|)\,d\lambda d\nu(x)\\
    =  2C_d{\rm vol}(M)(2m+d)^{-1} \|A\|_{\mathcal{L}(L^2(M))}^2,
\end{multline}
which, we already observe, tends to $0$ as $m$ goes to infinity, even though we will take this limit in a second time.
    
\vspace{0.25cm}

The estimation of the quantum variance of $B_2$ is similar to the one of $B_3$. Indeed, once again relying on Equation~\eqref{eq:aprioribound}, we can write
\begin{multline*}
    \limsup_{\hbar_n\rightarrow 0}{\rm Var}_{\hbar_n}(B_2)
    \leq \|A\|_{\mathcal{L}(L^2(M))}^2\int_{\Ghat M} {\rm Tr}_{\mathcal{H}_\pi}\left((1-\Bar{\rho}_m)^2(\Pi_{<m}^{\rho_m})^2{\bm 1}_{[0,1]}(H)\right)d\hat{\mu}_{\Ghat M}\\
    \leq \|A\|_{\mathcal{L}(L^2(M))}^2\int_{\Ghat M} {\rm Tr}_{\mathcal{H}_\pi}\left((1-\Bar{\rho}_m)^2{\bm 1}_{[0,1]}(H)\right)d\hat{\mu}_{\Ghat M}.
\end{multline*}
Once again, considering the support of the integrand, we observe that by the definition of the cutoff function $\Bar{\rho}_m$, we see that it is enough to integrate on $\{|\lambda|\leq (2m-1+d)^{-1}\}$. Thus, again using Equation~\eqref{eq:estimatetraceH}, we obtain the same bound as previously:
\begin{equation}
    \label{eq:estimateB2}
    \limsup_{\hbar_n\rightarrow 0}{\rm Var}_{\hbar_n}(B_2)
    \leq  2C_d{\rm vol}(M)(2m-1+d)^{-1} \|A\|_{\mathcal{L}(L^2(M))}^2.
\end{equation}

\vspace{0.25cm}
    
We now need to estimate the quantum variance of $B_1$. For this, observe that, once again by symbolic consideration, we have
\begin{equation*}
    B_1 = \sum_{j=0}^{m-1}\hat{\Pi}_j^{\rho_m,\hbar_n} \Op_{\hbar_n}(\Bar{\rho}_m)A \hat{\Pi}_j^{\rho_m,\hbar_n} + \O_{L^2(M)\rightarrow L^2(M)}(\hbar).
\end{equation*}
We denote by $B_1^j$ the operator $\hat{\Pi}_j^{\rho_m,\hbar_n} \Op_{\hbar_n}(\Bar{\rho}_m) A \hat{\Pi}_j^{\rho_m,\hbar_n}$. We now estimate the quantum variance of each of these operators.

We see that for $j= 0,\dots,m-1$, the operator $B_1^j$ is a Toeplitz operator in $\mathcal{T}_j^{\rho_m}(M)$ since, following Definition~\ref{def:gentoeplitz} and thanks to the cutoff function $\Bar{\rho}_m$, the wavefront set of $B_1^j$ is contained in $\{\rho_m = 1\}$. Moreover, its principal symbol $\beta_j\in C^\infty(\Ghat_\infty M, {\rm End}(\mathcal{L}_m(N\Sigma))$ is of scalar-type, which means that it acts as a homothety on each fiber of the $j$-th Landau bundle $\mathcal{L}_j(N\Sigma)$. Moreover, the coefficient of this homothety is equal in each fiber to the function $a(\cdot)-m_a$, up to a multiplicative constant that only depends on $\lambda$, i.e.~we can write
\begin{equation}
    \label{eq:formulebetaj}
    \beta_j(x,\lambda) = (a(x)-m_a)f_j(\lambda),
\end{equation}
for some smooth function $f_j\in C_c^\infty(\R^*)$.
    
Now, for $k\in\N$, we can write
\begin{equation*}
    \langle B_1^j \varphi_k,\varphi_k\rangle = \langle B_1^j e^{-it \delta_k}\varphi_k, e^{-it \delta_k}\varphi_k\rangle = \langle B_1^j e^{-it\Delta_{sR}}\varphi_k,e^{-it\Delta_{sR}}\varphi_k\rangle.
\end{equation*}
Therefore, with the notation $B_1^j(t) = e^{it\Delta_{sR}}B_1^j e^{-it\Delta_{sR}}$ and by averaging with respect to time $t$ on the interval $[0,T]$, we obtain
\begin{equation*}
    \langle B_1^j \varphi_k,\varphi_k\rangle =  \left\langle \langle B_1^j\rangle_T\varphi_k,\varphi_k\right\rangle,
\end{equation*}
where 
\begin{equation*}
    \langle B_1^j\rangle_T = \frac{1}{T}\int_{0}^T B_1^j(t)\,dt.
\end{equation*}
Theorem~\ref{thm:Egorov} then tells us that we can write
\begin{equation*}
    \langle B_1^j \rangle_T = \langle \tilde{B}_1^j \rangle_T + \O_{L^2(M)\rightarrow L^2(M)}(\hbar),
\end{equation*}
where $\tilde{B}_1^j(t)$ is in $\mathcal{T}_m^\rho(M)$ with its principal symbol $\tilde{\beta}_j(t)$ satisfying 
\begin{equation*}
    \tilde{\beta}_j(t) = {\rm Ad}(\Phi_t^j)^*\beta_j,
\end{equation*}
and where we set $\langle \tilde{B}_1^j\rangle_T = \frac{1}{T}\int_0^T \tilde{B}_1^j(t)\,dt$. By definition of the flow ${\rm Ad}(\Phi_t^j)^*$, the symbol $\tilde{\beta}_j(t)$ is still of scalar-type and is, in fact, simply given by
\begin{equation*}
    \tilde{\beta}_j(t) = \beta_j(\Phi_t^j),
\end{equation*}
where we identified $\tilde{\beta}_j$ and $\beta_j$, respectively, with the smooth coefficient function that defines the homothety in each fiber of $\mathcal{L}_j(N\Sigma)$. We also introduce the notation $\langle \tilde{\beta}
_j\rangle_T = \frac{1}{T}\int_0^T \tilde{\beta}_j(t)\,dt$.
    
Thus, we are able to write
\begin{equation*}
    \begin{aligned}
    \limsup_{\hbar_n\rightarrow 0}{\rm Var}_{\hbar_n}(B_1^j) 
    &= \limsup_{\hbar_n\rightarrow 0}{\rm Var}_{\hbar_n}(\langle B_1^j\rangle_T)\leq \int_{\Ghat M}{\rm Tr}_{\mathcal{H}_\pi}\left(\langle \tilde{\beta}_j\rangle_T^*\langle \tilde{\beta}_j\rangle_T{\bm 1}_{[0,1]}(H)\right)d\hat{\mu}_{\Ghat M}\\
    &\leq\int_{\Ghat M} |\langle \tilde{\beta}_j\rangle_T|^2{\rm Tr}_{\mathcal{H}_\pi}\left(\Pi_m^{\rho_m}{\bm 1}_{[0,1]}(H)\right)d\hat{\mu}_{\Ghat M},
    \end{aligned}
\end{equation*}
where we once again identified the symbol $\tilde{\beta}_j(t)$ with its scalar coefficient. 

Using Equation~\eqref{eq:formulebetaj} and the estimate~\eqref{eq:estimatetraceH}, we can bound this last quantity by
\begin{equation*}
    C_d\int_{\R^*}\left(\int_M |\langle a\rangle_T^{\lambda,j}(x)-m_a|^2\,d\nu(x)\right)|f_j(\lambda)|^2\,d\lambda,
\end{equation*}
where we set
\begin{equation*}
    \langle a\rangle_T^{\lambda,j} := \frac{1}{T}\int_0^T a(\Phi_{(2j+1)t}^{{\rm sgn}(\lambda)})\,dt,
\end{equation*}
and where we denoted by $\Phi_t^{{\rm sgn}(\lambda)}$ the Reeb flow on $M$ going forward or backward in time, depending on the sign of $\lambda$.
Thus, by the assumption of ergodicity of the Reeb flow and by the von Neumann theorem, we have that for every $\lambda\in \R^*$, 
\begin{equation*}
    \int_M |\langle a\rangle_T^{\lambda,j}(x)-m_a|^2\,d\nu(x)\Tend{T}{+\infty} 0,
\end{equation*}
and thus by dominated convergence, letting $T$ go to $+\infty$, we have
\begin{equation}
    \label{eq:esimationvarB1}
    \limsup_{\hbar_n\rightarrow 0}{\rm Var}_{\hbar_n}(B_1^j) = 0.
\end{equation}

\vspace{0.25cm}

Putting all the estimates together, i.e., by Equations~\eqref{eq:decompositionvarA}, \eqref{eq:firstestimateB3}, \eqref{eq:estimateB2} and \eqref{eq:esimationvarB1}, we obtain that
\begin{equation*}
    \limsup_{\hbar_n\rightarrow 0}{\rm Var}_{\hbar_n}(A) \leq 12 C_d {\rm vol}(M)(2m-1+d)^{-1} \|A\|_{\mathcal{L}(L^2(M))}^2.
\end{equation*}

As $m\in \N$ is arbitrary, letting $m$ go to $+\infty$, we conclude that $\limsup_{\hbar_n\rightarrow 0}{\rm Var}_{\hbar_n}(A)$ equals 0 and this ends the proof. 

\appendix

\section{Semiclassical functional calculus}
\label{sect:functionalcalc}

This appendix is dedicated to the proof of Theorem~\ref{thm:functionalcalc}. It is inspired by the proof of the semiclassical functional calculus for the usual pseudodifferential calculus on manifolds proposed by \cite{BK}, and makes use of the semiclassical functional calculus on graded nilpotent Lie groups given by \cite{FM}.

\vspace{0.25cm}

We fix $f\in C_c^{\infty}(\R)$ a smooth compactly supported function and we consider the operator $f(-\hbar^2\Delta_{sR})$ acting on $L^2(M)$, which is defined by functional calculus.

Following Definition~\ref{def:globalquantization}, we choose a particular global quantization with which to work, or equivalently a particular atlas, inspired by the discussion at the end of Section~\ref{subsubsect:locsemiquantiz}. In what follows, we consider $\tilde{\mathcal{A}} = (\tilde{\bX}_\alpha,U_\alpha,\chi_\alpha,\psi_\alpha)_{\alpha\in A}$, an atlas in the sense of Definition~\ref{def:atlas}, such that all frames $\tilde{\bX}_\alpha$, $\alpha\in A$, are Darboux frames.

Doing so, by fixing a reference point $x_\alpha \in U_\alpha$ for all $\alpha\in A$, the exponential map $\exp_{x_\alpha}^{\tilde{\bX}_\alpha}$ with base point $x_\alpha$ defines a map between an open neighborhood $V_\alpha\subset \Heis^d$ of zero and $U_\alpha$, which we simply denote by $\gamma_\alpha$.

For each $\alpha\in A$, we choose in addition a compact set $K_\alpha \subset U_\alpha$ such that $\supp(\psi_\alpha)\subset {\rm Int}(K_\alpha)$, and two cutoff functions $\overline{\psi}_\alpha, \overline{\overline{\psi}}_\alpha \in C_c^\infty(M)$ with supports contained in ${\rm Int}(K_\alpha)$ and such that $\overline{\psi}_\alpha \equiv 1$ on $\supp(\psi_\alpha)$ and $\overline{\overline{\psi}}_\alpha \equiv 1$ on $\supp(\overline{\psi}_\alpha)$. We keep the same notation for these cutoffs functions pulled back on $V_\alpha \subset \Heis^d$ by $\gamma_\alpha$.

\vspace{0.25cm}

Now, for each $\alpha \in A$, we define the differential operator $T_\alpha(\hbar)$, $\hbar>0$, in the following way: considering a cutoff function $\rho_\alpha\in C_c^\infty(\Heis^d)$ such that $\rho_\alpha \equiv 1$ on $\gamma_\alpha^{-1}(K_\alpha) \subset V_\alpha$, we introduce
\begin{equation}
    \label{eq:locallaplacien}
    T_\alpha(\hbar) := \rho_\alpha (\gamma_\alpha^{-1})^* (-\hbar^2\Delta_{sR})\gamma_\alpha^* + (1-\rho_\alpha)(-\hbar^2\Delta_{\Heis^d}),
\end{equation}
where $\Delta_{\Heis^d}$ denotes the left-invariant subLaplacian on the Heisenberg group $\Heis^d$. Moreover, we denote by $\tau_\alpha \in S^2(\Heis^d\times\Hhat^d)$ its $\hbar$-independent principal symbol:
\begin{equation*}
    T_\alpha(\hbar) = \Op_{\Heis^d,\hbar}(\tau_\alpha) + \hbar \Psi_\hbar^1(\Heis^d).
\end{equation*}
This way, the operator $T_\alpha(\hbar)$ satisfies all the hypotheses of \cite{FM}, specifically Setting 6.1 and Hypotheses 6.2 and 6.3. Moreover, we then have
\begin{equation}
    \label{eq:localexpression}
    -\hbar^2\Delta_{sR}\varphi = T_\alpha(\hbar)(\varphi\circ\gamma_\alpha)\circ \gamma_\alpha^{-1},
\end{equation}
for any $\varphi\in C_c^\infty(M)$ with $\supp(\varphi)\subset K_\alpha$.

\vspace{0.25cm}

By self-adjointness, the operator $-\hbar^2\Delta_{sR}-z$ is invertible in $\mathscr{L}(L^2(M))$ for $z\in \C\setminus\R$, and by the Helffer-Sjöstrand formula \cite[Theorem 14.8]{Zwo}, one has
\begin{equation*}
    f(-\hbar^2\Delta_{sR}) = \frac{1}{i\pi}\int_\C \Bar{\partial}_z \tilde{f}(z) \left(-\hbar^2\Delta_{sR}-z\right)^{-1}dz,
\end{equation*}
where $dz$ denotes the Lebesgue measure on $\C$ and $\tilde{f}:\C\rightarrow \C$ is an almost analytic extension of $f$ as in \cite[Theorem 3.6]{Zwo}.

We now construct a parametrix that approximates $(-\hbar^2\Delta_{sR}-z)^{-1}$, using the operators $T_\alpha(\hbar)$, up to a given explicit remainder term. To this end, subelliptic regularity (see \cite[Lemma 6.9]{FM}) implies that the resolvents $(T_\alpha(\hbar)-z)^{-1}$ induce operators
\begin{equation*}
    (T_\alpha(\hbar)-z)^{-1} : C_c^\infty(\Heis^d)\rightarrow C^\infty(\Heis^d).
\end{equation*}
We now define the operator $P(\hbar,z) : C^\infty(M)\rightarrow C^\infty(M)$:
\begin{equation*}
    P(\hbar,z) := \sum_{\alpha\in A} (\gamma_\alpha^{-1})^* \left(\psi_\alpha^2\circ (T_\alpha(\hbar)-z)^{-1}\right) \circ \overline{\psi}_\alpha.
\end{equation*}
By Equation~\eqref{eq:localexpression}, and since the subLaplacian is a local operator, we can write
\begin{equation*}
    \begin{aligned}
        P(\hbar,z) \circ (-\hbar^2\Delta_{sR}-z) 
        &= \sum_{\alpha\in A}  (\gamma_\alpha^{-1})^* \left(\psi_\alpha^2\circ (T_\alpha(\hbar)-z)^{-1}\right) \circ \overline{\psi}_\alpha \circ (-\hbar^2 \Delta_{sR}-z)\circ \overline{\overline{\psi}}_\alpha\\
        &= \sum_{\alpha\in A}\Big[(\gamma_\alpha^{-1})^* \left(\psi_\alpha^2\circ (T_\alpha(\hbar)-z)^{-1}\right) \circ (\gamma_\alpha^{-1})^* (-\hbar^2 \Delta_{sR}-z)\circ \overline{\overline{\psi}}_\alpha \\
        &\qquad -(\gamma_\alpha^{-1})^* \left(\psi_\alpha^2\circ (T_\alpha(\hbar)-z)^{-1}\right) \circ (1-\overline{\psi}_\alpha) \circ (-\hbar^2 \Delta_{sR}-z)\circ \overline{\overline{\psi}}_\alpha\Big]\\
        &= 1 - \sum_{\alpha\in A} \tilde{\mathcal{R}}_\alpha(\hbar,z),
    \end{aligned}
\end{equation*}
where we set
\begin{equation*}
    \tilde{\mathcal{R}}_\alpha(\hbar,z) := (\gamma_\alpha^{-1})^* \left(\psi_\alpha^2\circ (T_\alpha(\hbar)-z)^{-1}\right) \circ (1-\overline{\psi}_\alpha) \circ (-\hbar^2 \Delta_{sR}-z)\circ \overline{\overline{\psi}}_\alpha.
\end{equation*}
Observe that by the introduction of the cutoff functions $\overline{\psi}_\alpha$, $\alpha\in A$, the definition of each remainder $\tilde{R}_\alpha(\hbar,z)$ involves an operator that is composed from the left and from the right with multiplication operators by functions whose supports are disjoint.

It now immediately follows that
\begin{equation}
    \label{eq:patchedparametrix}
    (-\hbar^2\Delta_{sR}-z)^{-1} = P(\hbar,z) + \sum_{\alpha\in A}\mathcal{R}_\alpha(\hbar,z),
\end{equation}
with
\begin{equation*}
    \mathcal{R}_\alpha(\hbar,z) := \tilde{\mathcal{R}}_\alpha(\hbar,z)\circ (-\hbar^2\Delta_{sR}-z)^{-1}.
\end{equation*}

Plugging Equation~\eqref{eq:patchedparametrix} into the Helffer-Sjöstrand formula, we obtain
\begin{equation*}
    \begin{aligned}
        f(-\hbar^2\Delta_{sR}) 
        &= \frac{1}{i\pi}\int_{\C}\Bar{\partial}_z\tilde{f}(z)\left(\sum_{\alpha\in A} (\gamma_\alpha^{-1})^* \left(\psi_\alpha^2\circ (T_\alpha(\hbar)-z)^{-1}\right) \circ \overline{\psi}_\alpha + \mathcal{R}_\alpha(\hbar,z)\right)\,dz\\
        &= \sum_{\alpha\in A}  \left[(\gamma_\alpha^{-1})^*\left(\psi_\alpha^2\circ\frac{1}{i\pi}\int_{\C}\Bar{\partial}_z\tilde{f}(z)(T_\alpha(\hbar)-z)^{-1}\,dz \right)\circ \overline{\psi}_\alpha + \mathcal{R}_\alpha(\hbar)\right]\\
        &= \sum_{\alpha\in A}\left[ (\gamma_\alpha^{-1})^*(\psi_\alpha^2\circ f(T_\alpha(\hbar))\circ \overline{\psi}_\alpha + \mathcal{R}_\alpha(\hbar)\right],
    \end{aligned}
\end{equation*}
where we set
\begin{equation*}
    \mathcal{R}_\alpha(\hbar) := \frac{1}{i\pi}\int_\C \Bar{\partial}_z \tilde{f}(z) \mathcal{R}_\alpha(\hbar,z)\,dz.
\end{equation*}

We now make use of \cite[Theorem 6.7]{FM} which allows us to say that each operator $f(T_\alpha(\hbar))$, $\alpha\in A$, decomposes as 
\begin{equation*}
    f(T_\alpha(\hbar)) = \Op_{\Heis^d,\hbar}(\sigma_\alpha) + \mathfrak{R}_\alpha(\hbar),
\end{equation*}
where the symbol $\sigma_\alpha$ is in $S^{-\infty}(\Heis^d\times\Hhat^d)$, coincides at first order in $\hbar$ with $f(\tau_\alpha)$, and where the remainder $\mathfrak{R}_\alpha(\hbar)$ is semiclassically smoothing on the Sobolev scale, in the sense of \cite[Definition 5.9]{FM}.

\vspace{0.25cm}

Thus, in order to deduce the result of Theorem~\ref{thm:functionalcalc}, we are left with proving that the remainders $\mathcal{R}_\alpha(\hbar)$ are each of order $\O(\hbar^\infty)$ as bounded operators acting on $L^2(M)$. To do so, we estimate the boundedness of the integrand in the definition of the remainder $\mathcal{R}_\alpha(\hbar)$.

\begin{lemma}
    \label{lem:easyestimate}
    Let $\psi_1,\psi_2\in C_c^\infty(\Heis^d)$ be two smooth cutoff functions with disjoint support. Then, for each $N\in\N$, there exists a constant $C_N>0$ such that, for $z\in \C\setminus\R$ and $|z|$ bounded, we have
    \begin{equation*}
        \|\psi_1\circ (T_\alpha(\hbar)-z)^{-1}\circ \psi_2\|_{\mathcal{L}(L^2(\Heis^d))} \leq C_N \hbar^N |{\rm Im}\,z|^{-N-1}.
    \end{equation*}
\end{lemma}

\begin{proof}
    The proof goes exactly as in the one of \cite[Lemma 4.3]{BK}, the only estimate that one needs being that, for bounded $z$ in $\C\setminus\R$, we must have
    \begin{equation*}
        \|[T_\alpha(\hbar),\psi]\circ (T_\alpha(\hbar)-z)^{-1}\|_{\mathcal{L}(L^2(\Heis^d))} = \O(\hbar |{\rm Im}\,z|),
    \end{equation*}
    for any compactly supported function $\psi\in C_c^\infty(\Heis^d)$. This is an easy consequence of the commutator formula of two semiclassical pseudodifferential operators.
\end{proof}

Using the estimate of Lemma~\ref{lem:easyestimate}, we get for each $N\in\N$ a constant $C_N>0$ such that for bounded $z$, we have
\begin{equation*}
    \|(\gamma_\alpha^{-1})^* \left(\psi_\alpha^2\circ (T_\alpha(\hbar)-z)^{-1}\right) \circ (1-\overline{\psi}_\alpha)\|_{\mathcal{L}(L^2(M))} \le C_N \hbar^N |{\rm Im}\,z|^{-N-1}.
\end{equation*}

Moreover, we also have
\begin{multline*}
    \|(-\hbar^2 \Delta_{sr}-z)\circ \overline{\overline{\psi}}_\alpha\circ (-\hbar^2\Delta_{sR}-z)^{-1}\|_{\mathcal{L}(L^2(M))} \\= \|[-\hbar^2 \Delta_{sR},\overline{\overline{\psi}}_\alpha]\circ (-\hbar^2\Delta_{sR}-z)^{-1} + \overline{\overline{\psi}}_\alpha\|_{\mathcal{L}(L^2(M))} \leq C \hbar|{\rm Im}\,z|^{-1} + 1.
\end{multline*}
Finally, this gives, for a new constant $C_N'>0$ the following estimate
\begin{equation*}
    \|\tilde{\mathcal{R}}_\alpha(\hbar,z)\|_{\mathcal{L}(L^2(M))}\leq C_N'\hbar^N |{\rm Im}\,z|^{-N-1}.
\end{equation*}

Putting these estimates together in the definition of the remainder $\mathcal{R}_\alpha(\hbar)$, and since $\tilde{f}$ is a compactly supported analytic extension of $f\in C_c^\infty(\R)$, meaning it can be chosen to satisfy the estimate
\begin{equation*}
    \exists C_N>0,\forall z\in\supp(\tilde{f}),\quad|\Bar{\partial}_z \tilde{f}(z)|\leq C_N |{\rm Im}\,z|^N,
\end{equation*}
we obtain for each $\alpha\in A$,

\begin{equation*}
    \|\mathcal{R}_\alpha(\hbar)\|_{\mathcal{L}(L^2(M))} = \O_{L^2(M)\rightarrow L^2(M)}(\hbar^\infty).
\end{equation*}

\end{document}